\documentclass[12pt,a4paper]{amsart}
\PassOptionsToPackage{showonlyrefs}{mathtools}
\usepackage{preamble}
\usepackage{ulem}
\usepackage{amsmath}

\usepackage{amssymb}
\allowdisplaybreaks

\renewcommand{\leq}{\leqslant}
\renewcommand{\geq}{\geqslant}

\newcommand{\tf}{{\tilde f}}
\newcommand{\tth}{{\tilde h}}
\newcommand{\tg}{{\tilde g}}
\newcommand{\tW}{{\widetilde W}}
\newcommand{\thF}{{\widetilde {\mathcal {F}}}}
\newcommand{\tP}{{\mathcal P}}
\newcommand{\tR}{{\mathcal R}}
\newcommand{\tL}{{\mathcal L}}
\newcommand{\tQ}{{\mathcal Q}}
\newcommand{\tA}{{\mathcal A}}
\newcommand{\tB}{{\mathcal B}}
\newcommand{\tpi}{{\tilde \pi}}

\newcommand{\tgamma}{{\tilde \gamma}}
\newcommand{\tm}{{\tilde m}}
\newcommand{\Diff}{\mathrm{Diff^1}}
\newcommand{\uh}{\uline h}
\newcommand{\oh}{\overline h}
\newcommand{\udelta}{\uline \delta}
\newcommand{\odelta}{\overline{\delta}}
\newcommand{\e}{\mathrm{e}}
\newcommand{\td}{\mathrm{d}}
\newtheorem*{maintheorem}{Main Theorem}
\newcommand{\tF}{{\mathcal F}}

\newcommand{\tC}{{\mathcal C}}

\newcommand{\ud}{\uline d}
\newcommand{\od}{\overline d}

\title
[$C^1$ entropy formula part 2]
{
    Ledrappier-Young entropy formula for $C^1$ diffeomorphisms with dominated splitting \\
    Part 2: entropy formulas and measure dimension
}
\begin{document}

\begin{abstract}

In this paper, we partially extend the Ledrappier-Young entropy formula to invariant measures of $C^1$ diffeomorphisms with dominated splittings. For such measures, we show that whenever the $i$-th Lyapunov exponent has multiplicity one, the $i$-th transverse entropy equals the product of the $i$-th Lyapunov exponent and the corresponding transverse measure dimension. Furthermore, if all intermediate non-negative Lyapunov exponents have multiplicity one, then the Ledrappier-Young entropy formula holds. 

As applications, we derive $C^{1}$ versions of numerous results in measure dimension theory, including the famous works by Ledrappier-Young \cite[Section 12]{LEDRAPPIER_YOUNG_B}, Barreira-Pesin-Schmeling \cite{Barreira_Pesin_Schmeling}, and Ledrappier-Xie \cite{Ledrappier_Xie}.

\end{abstract}
\maketitle

\tableofcontents

\section{Introduction and main results}
\subsection{Introduction}
Entropy, Lyapunov exponents, and measure dimension are three fundamental concepts in ergodic theory, and they play central roles in the study of chaotic dynamical systems. Metric entropy was introduced by Kolmogorov in the 1950s as a quantitative measure of dynamical complexity, while Lyapunov exponents arise from Oseledets' multiplicative ergodic theorem \cite{Oseledets} from the 1960s. Precise definitions can be found, for example, in \cite{Rokhlin_1967}.

The Margulis-Ruelle inequality, see \cite{1978Ruelle-inequality}, states that the metric entropy of an invariant measure for a diffeomorphism is bounded above by the sum of the positive Lyapunov exponents, counted with multiplicity. For $C^{1+\alpha}$ volume-preserving diffeomorphisms, Pesin's entropy formula \cite{Pesin-entropyformula_1977} shows that the entropy of the volume measure is equal to this sum. In the $C^1$ setting, Sun and Tian \cite{SunTianC1pesinformula} proved Pesin's entropy formula for volume-preserving diffeomorphisms with dominated splittings; see also Tahzibi \cite{Tahzibi_2002}. Moreover, Wang, Wang, and Zhu \cite{WangwangZhu2018} proved $C^1$ versions, under dominated splitting assumptions, of the Margulis-Ruelle inequality and Pesin's entropy formula for partial entropy.

In \cite{Young_1982}, Young established a relation between entropy, Lyapunov exponents, and measure dimension for $C^{1+\alpha}$ surface diffeomorphisms. Subsequently, in the seminal works of Ledrappier and Young \cite{LEDRAPPIER_YOUNG_A,LEDRAPPIER_YOUNG_B}, they established the Ledrappier-Young entropy formula. This formula states that, for $C^2$ diffeomorphisms, the metric entropy is equal to the sum of the products of Lyapunov exponents and the corresponding measure dimensions. They also introduced the notion of partial entropy and proved analogous entropy formulas for partial entropies.

 Following these developments, many Ledrappier-Young type entropy formulas have been established in various settings under $C^2$ regularity assumptions. A random version of the Ledrappier-Young entropy formula was proved by Qian, Qian, and Xie in \cite{QIAN_QIAN_XIE_2003}. Entropy formulas for endomorphisms have also been extensively studied; see, for example, \cite{QianZhu02,Liu03,QianXie08,Liu08,Shu09}. In these works, the $C^2$ regularity assumption is used to guarantee the Lipschitz regularity of (strong) unstable holonomy maps, which is a crucial ingredient in the proof. This Lipschitz property also holds for $C^{1+\alpha}$ diffeomorphisms, as recently shown by Brown \cite{AaronBrown2022} and Saghin \cite{SaghinRadu2025}. Consequently, the arguments of Ledrappier and Young extend to the $C^{1+\alpha}$ setting. A natural question is then whether the Ledrappier-Young entropy formula continues to hold for invariant measures of $C^1$ diffeomorphisms with dominated splittings. In this paper, we provide a partial positive answer.

To state the result precisely, let $M$ be a compact $C^{\infty}$ Riemannian manifold without boundary, and let $f\in \Diff(M)$ be a $C^1$ diffeomorphism. Let $m$ be an $f$-invariant ergodic Borel probability measure on $M$. Our main theorem asserts that if $m$ admits a dominated splitting over its support, see \Cref{section 7.2} for precise definitions, then whenever the $i$-th Lyapunov exponent $\lambda_i$ has multiplicity one, the $i$-th transverse entropy $h_i-h_{i-1}$ equals the product of $\lambda_i$ and the corresponding transverse measure dimension $\gamma_i$. More precisely, we prove the following.

\begin{maintheorem}[\Cref{transverse entropy formula}]
     Let $f\in \Diff(M)$ preserve an ergodic Borel probability $m$, $2\leq i\leq u$. If the splitting $E^{(i-1)}\oplus E^{i}\oplus E^{\widehat{(i)}}$ is dominated on the support of $m$ and $\dim E^{i}=1$, then 
    \begin{equation*}
    \udelta_i-\udelta_{i-1}=\odelta_i-\odelta_{i-1}  :=  \gamma_i\leq 1
    \end{equation*}
    and
    \begin{equation*}
    h_i-h_{i-1}=\lambda_i\cdot\gamma_i.
    \end{equation*}
\end{maintheorem}

Here $E^{(i-1)}$, $E^i$, and $E^{\widehat{(i)}}$ denote the Oseledets bundles corresponding to the top $i-1$ Lyapunov exponents, the $i$-th Lyapunov exponent, and the remaining exponents, respectively. The quantities $\udelta_i$ and $\odelta_i$ denote the $i$-th lower and upper partial measure dimensions, respectively, and $h_i$ denotes the $i$-th partial entropy.

Moreover, if we assume a stronger assumption that all the intermediate Lyapunov exponents have multiplicity one, we can obtain exactly the Ledrappier-Young entropy formula for $C^1$ diffeomorphisms. See \Cref{section 7.2} for precise statements and details.

It is worth noting that the Lipschitz regularity of the holonomy maps induced by strong unstable leaves is an important ingredient in several earlier works devoted to weakening the smoothness assumptions. However, for $C^{1}$ diffeomorphisms with dominated splittings, the (strong) unstable holonomy maps are only Hölder continuous in general. This paper makes progress toward establishing Ledrappier-Young's theory under $C^1$ regularity by working with Hölder continuous fake (strong) unstable holonomy maps.

Ledrappier-Young's theory has numerous applications, particularly in measure dimension theory. Ledrappier and Young used this in \cite[Section~12]{LEDRAPPIER_YOUNG_B} to obtain an upper bound for the upper measure dimension and to derive several consequences for full-dimensional measures. Barreira, Pesin, and Schmeling proved the Eckmann-Ruelle conjecture for $C^{1+\alpha}$ diffeomorphisms in \cite{Barreira_Pesin_Schmeling}, showing that the dimension of a hyperbolic ergodic measure is exact. Ledrappier and Xie proved in \cite{Ledrappier_Xie} that zero transverse dimension is equivalent to the conditional measures being transversely Dirac. All these results require the $C^{1+\alpha}$ regularity assumption on $f$. Using our main results, we obtain $C^1$ versions of these results. See \Cref{section 7.3} for details. For further developments related to the Ledrappier-Young's theory and measure dimension, see also the works of Ben Ovadia on the Ledrappier-Young property \cite{Snir-LY-ETDS} and on tubular dimension \cite{Sinr-tubular}, and the works of Ben Ovadia and Rodriguez Hertz \cite{Sinr-Rodriguez-HertzIMRN}, Dong and Qiao \cite{DongqiaoCMP} on neutralized local entropy.

\subsection{Entropy formulas}\label{section 7.2}
In this section, we provide the technical definitions and state the main results of this paper.
As before, $M$ is a $C^{\infty}$ compact Riemannian manifold without boundary. Let $f$ be a $C^1$ diffeomorphism of $M$ onto itself, denoted by $f \in \Diff(M)$. Let $m$ be an $f$-invariant ergodic Borel probability measure on $M$.

A point $x\in M$ is said to be a regular point if there exist numbers $\lambda_1(x)>\lambda_2(x)>\cdots >\lambda_{r(x)}(x)$ and a decomposition of the tangent space $$T_xM=E^1(x)\oplus E^2(x)\oplus \cdots \oplus E^{r(x)}(x)$$   such that   for each $v\in E^i(x)\setminus\{0\}$, we have $$\lim_{n\rightarrow \pm\infty} \frac{1}{n}\log | Df^n_x (v)|=\lambda_i(x),$$ and $$\lim_{n\rightarrow \pm \infty}\frac{1}{n}\log |\mathrm{Jac}(Df^n_x)|=\sum^{r(x)}_{i=1}\lambda_i(x)\dim E^i(x).$$
By the Oseledets theorem \cite{Oseledets}, the set of regular points with respect to $m$, denoted by $\Gamma_{\mathrm{reg}}$, has full $m$-measure. The quantities $\lambda_i(x)$, $r(x)$, and $\dim E^i(x)$ are $f$-invariant. Since $m$ is ergodic, for $m$-a.e. $x\in M$,
$\lambda_i(x)=\lambda_i$ are called the Lyapunov exponents, and $\dim E^i(x)=\dim E^i$ is called the multiplicity of $\lambda_i$.

We recall some basic properties of dominated splittings (see \cite{ABC}). A $Df$-invariant splitting $T_\Lambda M=F\oplus E$ of the tangent bundle over an $f$-invariant set $\Lambda$ is dominated if there exists $N>0$   such that   given any $x\in \Lambda$ and any unit vectors $u\in F(x),v\in E(x)$, we have $$\frac{1}{2}\cdot| Df_x^N(u)|\geq| Df_x^N(v)|.$$
More generally, a $Df$-invariant splitting $T_\Lambda M=E_1\oplus\cdots\oplus E_t$ is dominated if, for every $\ell\in\{ 1,\cdots,t-1\}$, the splitting $(E_1\oplus\cdots\oplus E_\ell)\oplus(E_{\ell+1}\oplus\cdots\oplus E_t)$ is dominated. Recall that if an $f$-invariant set $\Lambda$ admits a dominated splitting $T_\Lambda M=E_1\oplus\cdots\oplus E_t$, then the splitting depends continuously on $x\in \Lambda$ and extends uniquely to a dominated splitting over the closure of $\Lambda$. In particular, the angles between different bundles are uniformly bounded away from zero.

In general, we use $(i)$ to denote all the indices less than or equal to $i$, use $\widehat i$ to denote all the indices except $i$, and use $\widehat{(i)}$ to denote all the indices larger than $i$. Now assume that $\lambda_1>\lambda_2>\cdots>\lambda_u>0$ are all the positive Lyapunov exponents of $m$.
For $1\leq i\leq u$, we assume that the sub-splitting of the Oseledets splitting $E^{(i-1)}\oplus E^{i}\oplus E^{\widehat{(i)}}$ is a dominated splitting over $\mathrm{supp}m\cap \Gamma_{\mathrm{reg}}$, where $E^{(i-1)}$ is the sub-bundle of exponents $\{\lambda_1,\cdots, \lambda_{i-1}\}$, $E^{i}$ is the sub-bundle of exponent $\{\lambda_i\}$, and $E^{\widehat{(i)}}$ is the sub-bundle of exponents $\{\lambda_{i+1},\cdots,\lambda_r\}$. We also use the notations $E^u,E^c,E^s,E^{cs}$ with the same meaning as in \cite{Part1} to represent the bundles of the positive, zero, negative, and non-positive Lyapunov exponents, respectively. 

In what follows, we assume throughout that the splitting  
$E^{(i-1)}\oplus E^{i}\oplus E^{\widehat{(i)}}$  
is dominated on $\mathrm{supp}\,m\cap \Gamma_{\mathrm{reg}}$. 
Since $\Gamma_{\mathrm{reg}}$ has full measure, the dominated splitting can be extended to a dominated splitting over $\mathrm{supp}m$.

For any $x\in \Gamma_{\mathrm{reg}}\cap\mathrm{supp}m$, we define $W^{i-1}$ and $W^i$ by $$W^{i-1}(x)=\{ y\in M: \limsup_{n\rightarrow \infty}\frac{1}{n}\log (d(f^{-n}x,f^{-n}y))\leq-\lambda_{i-1}+\epsilon,\ \forall \epsilon>0\},$$
and
$$W^{i}(x)=\{ y\in M: \limsup_{n\rightarrow \infty}\frac{1}{n}\log (d(f^{-n}x,f^{-n}y))\leq-\lambda_{i}+\epsilon,\ \forall \epsilon>0\},$$
where $d$ is the distance induced by the Riemannian metric on $M$. When  $E^{(i-1)}\oplus E^{i}\oplus E^{\widehat{(i)}}$ is dominated on the support of $m$, Abdenur, Bonatti, and Crovisier proved in \cite[Section 8]{ABC} that both $W^{i-1}(x)$ and $W^i(x)$ are injectively immersed $C^1$ manifolds with dimensions $\dim E^{(i-1)}(x)$ and $\dim E^{(i)}(x)$, respectively. Moreover, we have $W^{i-1}(x)\subseteq W^i(x)$.

We define the partial entropy of $f$ following the approaches in \cite{LEDRAPPIER_YOUNG_B,wuweisheng-u-entropy,YANG_expanding_entropy,Part1}. The concepts of measurable partitions, disintegration, and entropy are used in this paper; the reader can find details in \cite{Rokhlin_1967}.

To begin, we introduce the measurable partitions subordinate to $W^i$. This notion was first introduced by Ledrappier and Strelcyn in \cite{Ledrappier_Strelcyn_1982}. 
\begin{definition}[Subordinate partitions]\label{def subordiante partition}
Let $\xi$ be a measurable partition of $M$ and $1\leq i\leq u$. We say $\xi$ is subordinate to $W^i$ if for $m-$a.e. $x\in M$, we have
\begin{itemize}
\item[(a)] $\xi(x)\subseteq  W^i(x)$ and $\xi(x)$ contains an open neighborhood of $x$ inside $W^i(x)$.
\item[(b)] $\xi$ is an increasing partition, meaning that $f\xi \prec \xi$.
\item[(c)] $\bigvee_{n=0}^\infty f^{-n}\xi$ is the partition into points.
\end{itemize}
\end{definition}

Similarly, we can define the $i$-th partial entropy $h_i  :=  H(\xi|f\xi)$ for any $\xi$ subordinate to $W^i$. The value $h_i$ can be shown to be independent of the choice of $\xi$. See \cite[Lemma 3.1.2]{LEDRAPPIER_YOUNG_A} for details.

As in \cite{LEDRAPPIER_YOUNG_B}, we can also define partial entropy through Bowen balls on unstable manifolds. We now introduce the definition of Bowen balls on $W^i$. Let $\epsilon>0$ and $n\in\mathbb N^+$. We define 
$$V^i(x,n,\epsilon):=\{y\in W^i(x):d^i(f^k(x),f^k(y))<\epsilon,\ \forall 0\leq k<n\},$$
where $d^i$ is the distance induced by the Riemannian submanifold metric on $W^i$. Let $\xi$ be a measurable partition subordinate to $W^i$, and let $\{m^i_x\}$ be the conditional measures of $m$ associated with the measurable partition $\xi$. We define
\begin{equation}\label{bowen ball def 1}
    \uline h_i(x,\epsilon,\xi):=\liminf_{n\rightarrow\infty}-\frac{1}{n}\log m^i_xV^i(x,n,\epsilon).
\end{equation}
\begin{equation}\label{bowen ball def 2}
    \overline h_i(x,\epsilon,\xi):=\limsup_{n\rightarrow\infty}-\frac{1}{n}\log m^i_xV^i(x,n,\epsilon).
\end{equation}
Then we prove the following proposition in \Cref{section partial entropy}.
\begin{proposition}\label{proposition 1}
    For $m$-a.e. $x\in M$ and $1\leq i\leq u$, one has
    $$\lim_{\epsilon\rightarrow 0}\uline h_i(x,\epsilon,\xi)=\lim_{\epsilon\rightarrow0}\overline h_i(x,\epsilon,\xi)=h_i$$ and $h_u=h_m^u(f)$, where $h^u_m(f)$ is the $u$-entropy of $(f,m)$ defined in  \cite{Part1}.
\end{proposition}
Hence, the two definitions of partial entropy, by subordinate partitions and by Bowen balls, coincide as $h_i$.

We now turn to the definition of measure dimensions. Again fix $1\leq i\leq u$. Let $B^i(x,\epsilon)$ denote the $d^i$-ball in $W^i(x)$ centered at $x$ with radius $\epsilon$, and let $\xi$ be a measurable partition subordinate to $W^i$ with conditional measures $\{m^i_x\}$ of the measure $m$. For $x\in \Gamma_{\mathrm{reg}}\cap \mathrm{supp}m$, we define
\begin{equation*}
    \udelta_i(x,\xi)=\liminf_{\epsilon\rightarrow 0}\frac{\log m^i_x(B^i(x,\epsilon))}{\log\epsilon}
\end{equation*}
and
\begin{equation*}
    \odelta_i(x,\xi)=\limsup_{\epsilon\rightarrow 0}\frac{\log m^i_x(B^i(x,\epsilon))}{\log\epsilon}.
\end{equation*}
It is easy to verify, as in \cite{LEDRAPPIER_YOUNG_B}, that $\udelta_i(x,\xi)$ and $\odelta_i(x,\xi)$ are constant along the  orbits and are defined independently of $\xi$. Since $m$ is ergodic, we have for $m$-a.e. $x\in M$, $\udelta_i(x,\xi)=\uline\delta_i$, and $\odelta_i(x,\xi)=\odelta_i$. We call $\udelta_i$ the lower dimension of $m$ along $W^i$ and $\odelta_i$ the upper dimension of $m$ along $W^i$.

Now, let us present our first main theorem: the transverse entropy formula.
\begin{theorem}[Transverse entropy formula]\label{transverse entropy formula}
    Let $f\in \Diff(M)$ preserve an ergodic Borel probability $m$, and let
    $2\leq i\leq u$. Assume that the splitting
    $E^{(i-1)}\oplus E^{i}\oplus E^{\widehat{(i)}}$ is dominated on the support
    of $m$ and that $\dim E^{i}=1$. Then the $i$-th transverse dimension of $m$
    is exact, in the sense that
    \begin{equation*}
        \udelta_i-\udelta_{i-1}
        =
        \odelta_i-\odelta_{i-1}.
    \end{equation*}
    Denote this common value by
    \begin{equation*}
        \gamma_i
        :=
        \udelta_i-\udelta_{i-1}
        =
        \odelta_i-\odelta_{i-1}
        \leq 1,
    \end{equation*}
    then we have
    \begin{equation*}
        h_i-h_{i-1}=\lambda_i\gamma_i.
    \end{equation*}

\begin{remark}
    We do not prove here that the dimensions along $W^i$ and $W^{i-1}$, which are $\delta_i$ and $\delta_{i-1}$,
    are separately exact; it remains unknown at this level of generality. Nevertheless, we do prove that in the
    present setting, the $i$-th transverse dimension is
    exact, which is the quantity denoted by $\gamma_i$.
\end{remark}
\end{theorem}

If we further assume that all the intermediate Lyapunov exponents have multiplicity one, then we obtain the Ledrappier-Young entropy formula for partial entropy as follows. It is worth noting that we do not assume that the top Lyapunov exponent is simple.
\begin{theorem}[Partial entropy formula]\label{partial entropy formula}
    Let $f\in \Diff(M)$ preserve an ergodic Borel probability $m$. Fix $2\leq i\leq u$. If the splitting $E^{1}\oplus E^{2}\oplus\cdots\oplus E^{i}\oplus E^{\widehat {(i)}}$ is dominated on the support of $m$ and $\dim E^{2}=\cdots=\dim E^{i}=1$, then for any $1\leq k\leq i$,
\begin{equation*}
    \udelta_k=\odelta_k  :=  \delta_k\leq\sum_{j=1}^{k}\dim E^{j}=\dim E^{1}+k-1.
\end{equation*}
Denote $\gamma_1:=\delta_1$ and $\gamma_k:=\delta_k-\delta_{k-1}$ for $2\leq k\leq i$. Then
\begin{equation*}
    h_i=\sum_{k=1}^{i}\lambda_k\gamma_k.
\end{equation*}
\end{theorem}

If we make the stronger assumption that all non-negative Lyapunov exponents, except the largest, have multiplicity at most one, then we obtain an entropy formula for the metric entropy $h_m(f)$. In the following theorem, $E^{u,i}$ denotes the bundle of the positive Lyapunov exponent $\{\lambda_i\}$, $E^c$ denotes the bundle of the zero Lyapunov exponent, and $E^s$ denotes the bundle of the negative Lyapunov exponents.

\begin{theorem}[Metric entropy formula]\label{metric entropy formula}
    Let $f\in \Diff(M)$ preserve an ergodic Borel probability $m$. If the splitting $E^{u,1}\oplus E^{u,2}\oplus\cdots\oplus E^{u,u}\oplus E^{c}\oplus E^s$ is dominated on the support of $m$, 
    $\dim E^{u,2}=\cdots=\dim E^{u,u}=1$,
    and $\dim E^c\leq 1$, then 
    \begin{equation*}
    h_m(f)=\sum_{k=1}^{u}\lambda_k\cdot\gamma_k.
    \end{equation*}
\end{theorem}

\subsection{Measure dimension}\label{section 7.3}
The Ledrappier-Young entropy formula is a fundamental tool in the dimension theory of invariant measures. In the previous section, we showed that, in our setting, the standard $C^{1+\alpha}$ regularity assumption can be weakened to $C^1$. This allows one to extend a number of dimension-theoretic consequences of the Ledrappier-Young formula to the $C^1$ category in our setting. In this section, we present several such applications.

First, consider ergodic measures with zero transverse dimension. 
Based on Ledrappier-Young's theory, Ledrappier and Xie proved in \cite{Ledrappier_Xie} that zero transverse dimension is equivalent to the conditional measures along higher-dimensional unstable foliations being supported on a lower-dimensional strong unstable leaf.
Using the transverse entropy formula \Cref{transverse entropy formula}, we obtain a result analogous to theirs under $C^1$ regularity. 
Since the remainder of the proof is analogous and only requires $C^1$ smoothness, we omit the details.
\begin{theorem}
    Let $f\in\Diff(M)$ preserve an ergodic Borel probability $m$. For $0\leq i<j\leq u$, if the splitting $E^{(i)}\oplus E^{i+1}\oplus\cdots\oplus E^{j}\oplus E^{\widehat{(j)}}$ is dominated on the support of $m$ and $\dim E^{i+1}=\cdots =\dim E^{j}=1$, where we define $E^{(0)}  = \{0\}$, $\odelta_0=\udelta_0=0$, and $W^0(x)=\{x\}$, then the following are equivalent:
    \begin{itemize}
    \item[(1)] $\odelta_i=\odelta_j$. 
    \item[(2)] $\udelta_i=\udelta_j$.
    \item[(3)] For any measurable partition $\tA$ subordinate to $W^j$ and $m$-a.e. $x\in M$, the conditional measures of $m$ at $x$ with respect to $\tA$ are supported on a single leaf $W^i(x)$.
    \end{itemize}
\end{theorem}

We now provide some estimates of the measure dimension for ergodic measures of $C^1$ diffeomorphisms with dominated Oseledets splittings that satisfy several one-dimensional assumptions; we call such measures simply dominated.
\begin{definition}[Simply dominated]
    Let $f\in\Diff(M)$ preserve an ergodic Borel probability $m$. Then $m$ is called simply dominated if the Oseledets splitting is dominated on the support of $m$, and the Lyapunov exponents, except for the largest positive and the smallest negative, have multiplicity one. 
\end{definition}

In particular, for a simply dominated measure, one has either $\dim E^c=0$ or $\dim E^c=1$.  If $m$ is simply dominated, then the metric entropy formula \Cref{metric entropy formula} holds for both $f$ and $f^{-1}$. We can define the unstable dimension of $f$ as $d^u  :=  \sum_{i=1}^{u}\gamma_i$. Similarly, applying the metric entropy formula \Cref{metric entropy formula} to $f^{-1}$, we can define the transverse dimensions for $f^{-1}$ as $\bar\gamma_1,\cdots,\bar\gamma_s$, and the stable dimension of $f$ as $d^s  :=  \sum_{i=1}^{s}\bar\gamma_i$. 

Let $\xi^u$ be a measurable partition subordinate to $W^u$ for $f$, and let $\xi^s$ be a measurable partition subordinate to $W^s$ for $f^{-1}$. 
Let $\{m^u_x\}$ and $\{m^s_x\}$ be the conditional measures of $m$ associated with $\xi^u$ and $\xi^s$, respectively. Then, by the definition of the partial measure dimensions and \Cref{metric entropy formula}, it follows that
\begin{equation*}
    \lim_{\rho\rightarrow0}\frac{\log m^u_x(B^u(x,\rho))}{\log\rho}= d^u,
\end{equation*}
\begin{equation*}
       \lim_{\rho\rightarrow0}\frac{\log m^s_x(B^s(x,\rho))}{\log\rho}= d^s.
\end{equation*}

For any Borel measure $m$, the upper and lower pointwise measure dimension functions are defined as follows:
\[
\overline d(x):=\limsup_{\epsilon\rightarrow0}\frac{\log m(B(x,\epsilon))}{\log\epsilon}
\]
\[
\ud(x):=\liminf_{\epsilon\rightarrow0}\frac{\log m(B(x,\epsilon))}{\log\epsilon}.
\]
When $m$ is ergodic, they are both $f$-invariant functions and hence are constants almost everywhere. These two constants are called the lower and upper measure dimensions of the ergodic measure $m$.

Ledrappier and Young proved in \cite[Section 12]{LEDRAPPIER_YOUNG_B} that, in the $C^2$ setting, the upper dimension of an invariant measure $m$ is bounded above by $d^u+d^s+\dim E^c$. 
Using the metric entropy formula \Cref{metric entropy formula} in place of the Ledrappier-Young entropy formula, we obtain a result analogous to their theorem under $C^1$ regularity.

\begin{theorem}\label{C1 dimension estimate 1}
    If $m$ is simply dominated, then for $m$-a.e. $x\in M$, 
    \begin{equation*}
        \overline d(x)\leq d^u+d^s+\dim E^c.
    \end{equation*}
\end{theorem}
It is worth noting that the proof of \cite[Section 12.2]{LEDRAPPIER_YOUNG_B} relies on Lyapunov charts in the $C^2$ setting, which are not available in our $C^1$ setting. In this paper, we give a proof of \Cref{C1 dimension estimate 1} in \Cref{Appendix} following their ideas, using fake foliation charts instead.

We now recall the notion of exact dimensionality. An ergodic measure is said to
be exact dimensional if its upper and lower pointwise dimensions coincide
almost everywhere. Exact dimensional measures enjoy several useful consequences:
for instance, the Hausdorff dimension, the lower and upper measure-theoretic
box dimensions, and the lower and upper information dimensions all coincide
with the common pointwise dimension, denoted by $\dim m$. See
\cite{Young_1982} for details.

An ergodic measure is called \textit{hyperbolic} if it has no zero Lyapunov exponent. An important problem in dimension theory is whether the measure dimensions of hyperbolic measures are exact. This is often referred to as the Eckmann-Ruelle conjecture \cite{EckmannRuelle_85}.

Based on Ledrappier-Young's theory, Barreira, Pesin, and Schmeling proved in \cite{Barreira_Pesin_Schmeling} that, under $C^{1+\alpha}$ regularity, if $m$ is hyperbolic, then $m$ is exact dimensional and $\dim m=d^u+d^s$.

Using the metric entropy formula in \Cref{metric entropy formula}, we obtain \Cref{C^1 dimension estimate 2}, a $C^1$ analogue of their main theorem. Although our overall strategy is inspired by the framework of \cite{Barreira_Pesin_Schmeling}, the proof here relies essentially on fake foliations and the geometry of dominated splittings. This approach highlights the effectiveness of geometric tools arising from fake foliation charts and provides additional techniques for studying systems with dominated splittings. We refer the reader to \Cref{section erconjecture} for a more detailed and general discussion. To the best of the author's knowledge, this result is new even for simply dominated measures of $C^1$ Anosov diffeomorphisms.

\begin{theorem}\label{C^1 dimension estimate 2}
    If $m$ is simply dominated and hyperbolic, then for $m$-a.e. $x\in M$, 
    \begin{equation*}
        \overline d(x)=\underline d(x)=d^u+d^s.
    \end{equation*}
\end{theorem}

Finally, we introduce the concept of full-dimensional measures and Pesin entropy formulas.
We say that the $u$-Pesin entropy formula holds for $m$ if  $$h_m(f)=\sum_{\lambda_i>0}\lambda_i\cdot\dim E^i,$$ where $E^i$ is the sub-bundle of the exponent $\{\lambda_i\}$. Similarly, we say that the $s$-Pesin entropy formula holds for $m$ if $$h_m(f)=\sum_{\lambda_i<0}-\lambda_i\cdot\dim E^i.$$ We say that $m$ has full dimension if, for $m$-a.e.  $x\in M$, $$\overline d(x)=\dim M.$$

Using \Cref{C1 dimension estimate 1}, we directly obtain a relationship between the full dimension and Pesin entropy formulas, as in \cite[Corollary G]{LEDRAPPIER_YOUNG_B}.
\begin{proposition}\label{pesin full dimension 1}
    If $m$ is simply dominated and has full dimension, then both the $u$-Pesin entropy formula and the $s$-Pesin entropy formula hold for $m$. In this case, we have $$\int \log|\mathrm{Det} Df|\,\td m=0.$$
\end{proposition}

If we further assume that $m$ is hyperbolic, then by \Cref{C^1 dimension estimate 2}, we also obtain the converse part of the previous proposition.
\begin{proposition}\label{Pesin full dimension 2}
     If $m$ is simply dominated and hyperbolic, then $m$ has full dimension if and only if both the $u$-Pesin entropy formula and the $s$-Pesin entropy formula hold for $m$.
\end{proposition}

\subsection{Non-ergodic case} As in \cite[Section 1.5]{Part1}, all the theorems in \Cref{section 7.2} and \Cref{section 7.3} have non-ergodic versions. Obtaining the non-ergodic version from the ergodic results uses a standard reduction argument based on the Ergodic Decomposition Theorem. While other authors have presented similar arguments in the $C^2$ setting, extending them to our context presents no essential difficulties. However, the formal statements and proofs of these theorems can be quite lengthy. Therefore, we refer the reader to the following references for a detailed treatment of the non-ergodic versions of the aforementioned results:

\begin{itemize}
    \item For the entropy formulas (\Cref{transverse entropy formula,partial entropy formula,metric entropy formula}), see \cite[Section 7.5]{LEDRAPPIER_YOUNG_B}.
    \item For the dimension estimates (\Cref{C1 dimension estimate 1,C^1 dimension estimate 2}), see \cite[Section 12]{LEDRAPPIER_YOUNG_B} and \cite[Section 7]{Barreira_Pesin_Schmeling}.
    \item For the full dimensional measures results (\Cref{pesin full dimension 1,Pesin full dimension 2}), see \cite[Section 7.6]{LEDRAPPIER_YOUNG_B}.
\end{itemize}
\bigskip

\textbf{Strategy of proof:} 
We prove \Cref{transverse entropy formula} following essentially the ideas in \cite{LEDRAPPIER_YOUNG_B}, using fake foliation charts and the local argument (see \Cref{local arg here}) to replace the classical Lyapunov charts, as the author did in \cite{Part1}; see \Cref{section 8}.

Although the strong unstable holonomy induced by the strong unstable fake foliation is no longer Lipschitz for $C^1$ diffeomorphisms, the fake $(i-1)$-th unstable holonomy is $\alpha$-bi-Hölder inside the $i$-th unstable leaves, and $\alpha$ can be chosen arbitrarily close to $1$; see \Cref{uniform Hölder}. Thus, although the $\alpha$-Hölder straightening maps and holonomies distort dimension and roughly produce an error factor of $\alpha^m$, this error term disappears when we finally let $\alpha\rightarrow 1$; see \Cref{upper estimate}.

Compared with \cite{LEDRAPPIER_YOUNG_B}, in our paper, we have to carry out the arguments for both the lower partial dimension $\udelta_i$ and the upper partial dimension $\odelta_i$. The reason is that, knowing only that one intermediate Lyapunov exponent $\lambda_i$ is simple, we cannot obtain the exactness of $\delta_{i-1}$, since the induction in \cite{LEDRAPPIER_YOUNG_B} is not available in this case. However, with some additional effort, we can show that similar properties hold for both upper and lower partial dimensions; see \Cref{section 10} and \Cref{section 11}.

\bigskip

\textbf{Organization:} \Cref{section 8} contains some geometric preparations, including the introduction of fake foliation charts, straightening maps, and their Hölder properties. \Cref{section partial entropy} introduces partial entropy for ergodic measures of $C^1$ diffeomorphisms with dominated splittings, in parallel with \cite{LEDRAPPIER_YOUNG_B}. In \Cref{section 10} and \Cref{section 11}, we complete the proof of \Cref{transverse entropy formula}, \Cref{partial entropy formula}, and \Cref{metric entropy formula}, which are the main results of this paper. In \Cref{Appendix}, we prove an application concerning the upper bound for the upper measure dimension. In \Cref{section erconjecture}, we state and prove a $C^1$ plus domination version of the Eckmann-Ruelle conjecture.

\section{Geometry preparation}\label{section 8}
In this section, which is parallel to \cite[Section 8]{LEDRAPPIER_YOUNG_B}, we present some geometric preparations for the proof of the entropy formulas. The main difference is that we use the fake foliation charts induced by the dominated splitting instead of the classical Pesin charts, since the Pesin charts fail in our setting. 

In \Cref{section fake foliation} we construct the fake foliation parallel to \cite[Section 2.1]{Part1}. In \Cref{sec 2.2} we establish the uniform Hölder property of the fake strong unstable holonomy and consider the geometric properties of unstable manifolds, which are analogous to \cite[Section 2.3]{Part1}. In \Cref{straightening map section} we construct the straightening maps parallel to \cite[Section 8.4]{LEDRAPPIER_YOUNG_B}.

\subsection{Fake foliations}\label{section fake foliation}
In this section, we consider the fake foliations induced by the dominated splitting as in \cite[Section 2.1]{Part1}. We use the construction of \cite[Section 2.1]{Part1}, with the roles of
\(u,c,s\) replaced by \((i-1),i,\widehat{(i)}\), respectively.

Recall the assumption that $f\in \Diff(M)$ preserves an ergodic Borel probability $m$ and the splitting $E^{(i-1)}\oplus E^{i}\oplus E^{\widehat{(i)}}$ is dominated on the support of $m$. By the continuity of dominated splitting, there exists a constant $\theta>0$   such that   for any $x\in \mathrm{supp}m$ we have $|\mathrm{sin}(E^*(x),E^{**}(x))|>\theta,*\neq**\in\{(i-1),i,\widehat{(i)}\}$. This means that the angles between different bundles are uniformly bounded away from zero. Note that $E^{(i)}=E^{(i-1)}\oplus E^{i}$, $E^{{\widehat{(i-1)}}}=E^{i}\oplus E^{\widehat{(i)}}$, and $E^{{\widehat{i}}}=E^{(i-1)}\oplus E^{\widehat{(i)}}$.

We have the following lemma, which is analogous to \cite[Lemma 2.2]{Part1}:
\begin{lemma}\label{ABC ergodic lemma}
Given any $\epsilon>0$, there exist a full measure set $\Gamma'\subseteq \mathrm{supp}m\cap \Gamma_{\mathrm{reg}}$, a measurable function $A:\Gamma'\rightarrow[1,\infty)$, and a constant $N\in\mathbb N$  such that   $A(f^{\pm}x)\leq \mathrm{e}^{\epsilon}A(x)$ and for any $k\in \mathbb N $ we have:
\begin{equation*}
    \prod_{\ell=0}^{k-1}\parallel Df^N|_{E^{\widehat{(i)}}(f^{\ell N}(x))}\parallel \leq A(x)\cdot \mathrm{e}^{kN(\lambda_{i+1}+\epsilon)} 
\end{equation*}
\begin{equation*}
    \prod_{\ell=0}^{k-1}\parallel Df^N|_{E^{i}(f^{\ell N}(x))}\parallel \leq A(x)\cdot \mathrm{e}^{kN(\lambda_i+\epsilon)} 
\end{equation*}
\begin{equation*}
    \prod_{\ell=0}^{k-1}\parallel Df^{-N}|_{E^{i}(f^{-\ell N}(x))}\parallel \leq A(x)\cdot \mathrm{e}^{kN(-\lambda_i+\epsilon)} 
\end{equation*}
\begin{equation*}
    \prod_{\ell=0}^{k-1}\parallel Df^{-N}|_{E^{(i-1)}(f^{-\ell N}(x))}\parallel \leq A(x)\cdot \mathrm{e}^{kN(-\lambda_{i-1}+\epsilon)} 
\end{equation*}
\end{lemma}
Fix a constant $C_f\geq 100\max\{|Df^\pm|,\mathrm{e}^{|\lambda_1|+100\epsilon},\mathrm{e}^{|\lambda_r|+100\epsilon}\}$. There exists a constant $\Delta>0$ small enough such that for any $ x\in M$, the exponential map $\exp_x:\{|v|\leq \Delta\}\rightarrow M$ is a diffeomorphism to the image with $|D\exp^{\pm}|\leq 2$. We define $$\tf_x=\exp^{-1}_{f(x)}\circ f\circ \exp_x$$ whenever it makes sense.

Choose a $C^\infty$ bump function $\rho:\mathbb R\rightarrow\mathbb  R$   such that  
\begin{equation*}
	\rho(x)=\left\{
	\begin{aligned}
		1 & , &|x|&<1\\
	    \in[0,1] & , & |x|&\in [1,2]\\
            0 & , & |x|&>2
	\end{aligned}
	\right.
\end{equation*}
For any $r_0<\frac{0.1\Delta}{C_f}$, we can define $\tg_x:T_xM\rightarrow T_{f(x)}M$ for $x\in M$ as
\begin{equation*}
    \tg_x(v)=\left\{
	\begin{aligned}
		&\rho(|v|\cdot r_0^{-1})\tf_x(v)+(1-\rho(|v|\cdot r_0^{-1}))D\tf_x(0)v  , &|v|\leq 2r_0,\\
	   &D\tf_x(0)v   , & |v|>2r_0.\\
	\end{aligned}
	\right.
\end{equation*}
Then, as in \cite[Section 2.1]{Part1}, it can be proved that for any $\epsilon_0>0$, there exists $r_0>0$  such that   $D(\tg_x-D\tf_x(0))\leq \epsilon_0$. Denoting $L(g)$ as the Lipschitz constant of $g$, it implies that $L(\tg_x-D\tf_x(0))\leq \epsilon_0$ and $\tg_x=\tf_x$ when $|v|\leq r_0$. Hence, we can choose an $\epsilon_0'>0$ sufficiently small such that for any $\epsilon_0\leq\epsilon_0'$ and the $\tg_x$, $N$ defined above, 
\begin{equation}\label{local argument 1}
    \mathrm{e}^{-\epsilon N}\leq \frac{|D\tf^{\pm N}_x(0)v|}{|D\tg^{\pm N}_x(y) v|}\leq \mathrm{e}^{\epsilon N}, \qquad \forall x\in M,\ v ,y\in T_xM .   
\end{equation}

Unlike \cite{LEDRAPPIER_YOUNG_B}, in this article, we will not use the Lyapunov norms. Instead, we use the box norm $|\cdot|'$ on each $T_xM,x\in\Gamma'$ with respect to the splitting $T_xM= E^{(i-1)} _x\oplus E_x^i\oplus  E^{\widehat{(i)}}_x$, that is, $|v|'=\max\{|v_{(i-1)}|,|v_i|,|v_{\widehat{(i)}}|\}$, where $v=v_{(i-1)}+v_i+v_{\widehat{(i)}}$ is the decomposition as above and $|\cdot|$ is the norm induced by the Riemannian metric on $M$.

Denote $R^*_x(\rho)$ as the ball centered at $0$ with radius $\rho$ under the box norm in $E^{*}_x$, where $*\in\{(i-1),i,\widehat{(i)},(i),{\widehat{i}},{\widehat{(i-1)}}\}$. Similarly, $R_x(\rho)$ (without any superscript) is the ball centered at $0$ with radius $\rho$ under the box norm in $T_xM$. Since the angles between different bundles are uniformly bounded away from zero, there exists a constant $K>1$   such that   $K^{-1}|\cdot|\leq |\cdot|'\leq K|\cdot|$.

Henceforth, we slightly abuse notation by treating $*\in\{(i-1),i,\widehat{(i)},(i),{\widehat{i}},{\widehat{(i-1)}}\}$ as the dimension of their corresponding bundles (i.e., $* = \dim E^*$). 
By continuity, there exists a constant $C_0>K$ sufficiently large   such that   for any $x$ in the full measure set $\Gamma'$, and any $*$-dimensional linear subspace $V^*(x)$ as the graph of a linear function $G:E^*_x\rightarrow E^{\widehat *}_x$ with slope $\leq \frac{1}{C_0}$, where $*\in\{(i-1),i,\widehat{(i)},(i),{\widehat{i}},{\widehat{(i-1)}}\}$, we have
\begin{equation}\label{local argument 2}
	\mathrm{e}^{-\epsilon N}\leq\frac{\|D\tf_x^{\pm N}(0)|_{V^*(x)}\|}{\|D\tf_x^{\pm N}(0)|_{E^*(x)}\|}\leq \mathrm{e}^{\epsilon N},
\end{equation}
where $N$ is given by Lemma \ref{ABC ergodic lemma}.
We note that $C_0$ is a constant that depends only on $(f,m)$, $N$, and $\epsilon$, but not on $\epsilon_0$, $r_0$, or $\tg_\cdot$.

Now consider the dominated splitting $E^{(i)}\oplus E^{\widehat{(i)}}$.
By the property of dominated splitting, there exist a large constant $C_1>0$ and a constant $\epsilon_0''\leq\epsilon'_0$ such that for any $\epsilon_0\leq\epsilon_0''$ and $r_0$, $\tg_\cdot$ given by $\epsilon_0$ as above, we have for any $x\in \Gamma'$, if 
\begin{center}
    $U$ is the graph of a function $\varphi:E^{(i)}_x\rightarrow E^{\widehat{(i)}}_x$ with slope $\leq\frac{1}{C_1}$, 
\end{center}
then for any $ n\geq 0$,
\begin{center}
    $\tg^n_x(U)\subseteq T_{f^nx}M$ is the graph of a function $\varphi_n:E^{(i)}_{f^nx}\rightarrow E^{\widehat{(i)}}_{f^nx}$ with slope $\leq\frac{1}{C_0}$. 
\end{center}
Similarly, if
\begin{center}
$U$ is the graph of a function $\varphi:E^{\widehat{(i)}}_x\rightarrow E^{(i)}_x$ with slope $\leq\frac{1}{C_1}$, 
\end{center}
then for any $n\leq 0$,
\begin{center}
$\tg^n_x(U)\subseteq T_{f^nx}M$ is the graph of a function $\varphi_n:E^{\widehat{(i)}}_{f^nx}\rightarrow E^{(i)}_{f^nx}$ with slope $\leq\frac{1}{C_0}$.
\end{center} 
See, for instance, \cite{Katok_book} and \cite{HPS}.

The same arguments also hold for the dominated splitting $E^{(i-1)}\oplus E^{{\widehat{(i-1)}}}$.  Keep in mind that the slopes here are all considered under the box norm.

Fix a constant $\gamma_0<\frac{\pi}{10^6\cdot C_0^{100}}$. Following the standard discussions on the graph transformation (see, for instance, \cite{Katok_book,HPS,BW2010,LVY13}), we can choose $\epsilon_0>0$ and consequently $r_0>0$ small enough,   such that   the following lemma holds.

{
	\begin{lemma}[Fake foliations]\label{fake $u$-foliation}
		For each $x\in\Gamma'$ and $*\in\{{\widehat{(i-1)}},(i),\widehat{(i)},(i-1)\}$, there exist unique global fake foliations $\thF^{*}_x$ on $T_xM$ with $C^1$ leaves,   such that   for any $y \in T_xM$:
		\begin{enumerate}
			\item[a)] The unique leaf containing $y$, denoted by $\tW^{*}_x(y)$, is the graph of a $C^1$ function $\varphi:E^{*}_x\rightarrow E^{\hat{*}}_x$ with $|D\varphi|\leq\frac{\gamma_0}{C_1}$. Here when $* = \widehat{(i-1)},(i),\widehat{(i)},(i-1)$, $\hat{*} = (i-1),\widehat{(i)},(i),\widehat{(i-1)}$ respectively. 
			\item[b)] The foliations are invariant under $\tg_{\cdot}$ in the sense that 
			$$
			\tg_x\left(\thF^{*}_x(y)\right) = \thF^{*}_{f(x)}(\tg_x(y)).
			$$
			\item[c)] $\thF^{(i-1)}_x$ subfoliates $\thF^{(i)}_x$, and $\thF^{\widehat{(i)}}_x$ subfoliates $\thF^{\widehat{(i-1)}}_x$.
		\end{enumerate}
	\end{lemma}
}
{Note that Lemma \ref{fake $u$-foliation} does not immediately give the fake $i$-foliation}. Indeed, the fake $i$-foliation is defined by $\tW^i_x(y)=\tW^{{\widehat{(i-1)}}}_x(y)\cap\tW^{(i)}_x(y)$. By the previous lemma, it is invariant under $\tg_\cdot$ and is the graph of a function $\varphi:E^{i}_x\rightarrow E^{{\widehat{i}}}_x$ with slope $\leq \frac{\gamma_0}{C_1}$.

For simplicity, when $y\in \exp_x(|v|\leq r_0)$, we also denote $\tW_x^*(y)  :=  \tW^*_x(\exp_x^{-1}y)$ for $*\in\{(i-1),(i),i,\widehat{(i)},{\widehat{(i-1)}}\}$.

\subsection{Uniformly Hölder holonomy and unstable manifolds}\label{sec 2.2}
In this section, we introduce the uniformly Hölder property of the fake strong unstable holonomy induced by the strong unstable fake foliations within fake unstable manifolds and the geometric properties of unstable manifolds as in \cite[Section 2]{Part1}. 
The fact that the Hölder exponent of the fake strong unstable holonomy within fake unstable manifolds can be arbitrarily close to $1$ is very important for the proof of the entropy formulas in \Cref{section 11}.

First, we give the definition of $i$-transversal.
\begin{definition}[$i$-transversal]\label{c-trans-def}
	For each $x\in\Gamma'$, $U\subseteq T_xM$ is called an $i$-transversal inside $\tW^{(i)}_x(x)\subseteq T_xM$ with slope $\leq \frac{1}{C}$, if it is a one-dimensional submanifold of $\hat{U}=\mathrm{graph}(\varphi)\subseteq \tW^{(i)}_x(x)$ where the function $\varphi: E^i _x\rightarrow E^{\widehat i}_x$ satisfies $|D\varphi|\leq\frac{1}{C}$.  
\end{definition}
It follows that the $(i-1)$-holonomy is uniformly Hölder inside fake $(i)$-leaves.
\begin{lemma}[Uniformly Hölder]\label{uniform Hölder}
For each $x\in\Gamma'$, given two $i$-transversal $U$ and $V$ with slope $\leq\frac{1}{C_1}$ inside $\tW^{(i)}_x(x)\subseteq T_xM$, if the fake $(i-1)$-foliation $\thF ^{(i-1)}_x$ induces an injective holonomy $\tth^{(i-1)}_x:U\rightarrow V$, then the fake $(i-1)$-holonomy is uniformly Hölder. That is, for any $r>0$, $p,q\in U$ satisfying $|p-\tth^{(i-1)}_x(p)|\leq r$ and $|q-\tth^{(i-1)}_x(q)|\leq r$, there exists $C=C(A(x),r)$ such that for $\alpha=1-\frac{7\epsilon}{\lambda_{i-1}-\lambda_{i}}$,
\begin{equation*}
    C^{-1}|p-q|^{\alpha^{-1}}\leq |\tth^{(i-1)}_x(p)-\tth^{(i-1)}_x(q)|\leq C|p-q|^\alpha, \ \forall p,q\in U.
\end{equation*}
\begin{proof}
    The same as \cite[Lemma 2.4]{Part1}.
\end{proof}
\end{lemma}

Characterization of the geometric structure of the local unstable manifolds is also needed. For $$0<\delta\leq \delta_0  :=  \frac{r_0}{100C^N_fC_0^{10}},$$ 
the $i$-th and $(i-1)$-th local unstable manifolds are defined as
$$W^i_{x,2\delta}(y)  := \exp^{-1}_x(\text{component of }W^i(y)\cap \exp_x(R_x(2\delta A(x)^{-1}))\text{ that contains }y)$$
and
$$W^{i-1}_{x,2\delta}(y)  := \exp^{-1}_x(\text{component of }W^{i-1}(y)\cap \exp_x(R_x(2\delta A(x)^{-1}))\text{ that contains }y)$$
where $x,y\in\Gamma'$.
Then, by the same method as in \cite[Lemma 2.5]{Part1}, we can establish the following lemma.
\begin{lemma}
    For any $x\in\Gamma',\ W^i_{x,\delta}(x)=\tW^{(i)}_x(x)\cap R_x(\delta A(x)^{-1})$.\\
    For any $y\in\Gamma'\cap\exp_x(W^i_{x,\delta}(x)),\  W^{i-1}_{x,2\delta}(y)=\tW^{(i-1)}_x(y)\cap R_x(2\delta A(x)^{-1})$.
\end{lemma}
\begin{remark}
The preceding lemma should be understood as a scale statement. The fake leaves are only used inside boxes of size
$\delta A(x)^{-1}$, with $0<\delta\leq\delta_0$. At this size, the relevant
fake unstable leaves are genuine local Pesin unstable manifolds after
restriction. Thus, they have the same local unstable geometry; for example, the coherence of local leaves.
\end{remark}

 Subsequently, for any $0<\delta\leq \delta_0$, we can establish the following lemma, which is parallel to \cite[Lemma 2.6]{Part1}. The proof is almost the same, so we omit it.
\begin{lemma}\label{lemma 10 the manifold property}
    A. For any $x\in\Gamma',\ y\in \Gamma'\cap \exp_x(W^{i}_{x,\delta}(x))$,
    $W^{i-1}_{x,2\delta}(y)$ is the graph of a function $\varphi:R_x^{(i-1)}(2\delta A(x)^{-1})\rightarrow R_x^{{\widehat{(i-1)}}}(2\delta A(x)^{-1})$ with slope $\leq\frac{\gamma_0}{C_1}$.\\
    B. For $m$-a.e$.\ x\in\Gamma'$, if $y\in\Gamma'\cap \exp_x(W^{i}_{x,\delta}(x))$ and $f(y)\in \exp_{f(x)}(W^{i}_{f(x),\delta}(f(x)))$:
    \begin{itemize}
    \item[(1)] $\tf_xW^{i-1}_{x,2\delta}(y)\cap R_{f(x)}(2\delta A(f(x))^{-1})\subseteq W^{i-1}_{f(x),2\delta}(f(y))$.
    \item[(2)] $W^{i}_{x,2\delta}(x)\cap\exp_x^{-1}(W^{i-1}(y))=W^{i-1}_{x,2\delta}(y)$.    
    \end{itemize}
\end{lemma}

For convenience, we introduce the definitions of local charts and fake foliation charts as follows.

\begin{definition}[Local charts and fake foliation charts]\label{fake foliation charts def}
For any $x\in \Gamma'$, 

The {local chart at $x$} is defined as the set $\{v \in T_x M : |v| \leq r_0\}$, where $r_0>0$ is the uniform radius introduced in \Cref{section fake foliation}.

The {$\delta$-fake foliation chart at $x$} is the pair $\{R_x(\delta A(x)^{-1}), \exp_x\}$, where $R_x(\delta A(x)^{-1})$ is a neighborhood in $T_xM$ and $\exp_x: R_x(\delta A(x)^{-1}) \rightarrow M$ is the exponential map.

\end{definition}

\subsection{Straightening maps}\label{straightening map section}
In this section, we define straightening maps that straighten the local unstable manifolds as in \cite[Section 8.3]{LEDRAPPIER_YOUNG_B} and \cite[Section 8.4]{LEDRAPPIER_YOUNG_B}.
We first give a general straightening lemma parallel to \cite[Lemma 8.3.1]{LEDRAPPIER_YOUNG_B}. The proof is straightforward; hence we omit it.
\begin{lemma}[Straightening lemma]\label{straightening lemma} 
    Let $F^{(i-1)}$ be a continuous foliation with $C^1$ leaves on a subset of $R^{(i)}=R^{(i-1)}\oplus R^i$ that contains $R^{(i)}(2\rho)$, where $R^{(i)}$, $R^{(i-1)}$, and $R^i$ denote a $\dim E^{(i)}$-dimensional, $\dim E^{(i-1)}$-dimensional, and $1$-dimensional Euclidean space respectively, and $R^*(\rho)$ denotes the ball centered at the origin with radius $\rho$ under the box norm.
    
       Assume that each leaf of $F^{(i-1)}$ is the graph of a function $\varphi:R^{(i-1)}(2\rho)\rightarrow R^i$ with $|D\varphi|\leq\frac{1}{100}$. Then we define $O:R^{(i)}(\rho)\rightarrow R^{(i)}$ as $O(x_{(i-1)},x_i)=(x_{(i-1)},O_i(x))$, where $x=(x_{(i-1)},x_i)$ is the decomposition of coordinates and $(0,O_i(x))$ is the unique point in $F^{(i-1)}(x)\cap \{0\}\times R^i$. The following properties hold
    \begin{itemize}
    \item[(1)] $O$ is a homeomorphism between $R^{(i)}(\rho)$ and its image.
    \item[(2)] For any $x,y\in R^{(i)}(\rho)$, $O_i(x)=O_i(y)$ if and only if $y\in F^{(i-1)}(x)$.
    \item[(3)] If we further assume the holonomy induced by $F^{(i-1)}$ is uniformly $\alpha$-bi-Hölder as in \Cref{uniform Hölder}, then $O$ is the identity on its (i-1)-component and $\alpha$-bi-Hölder on its i-component. Hence, $O$ is $\alpha$-bi-Hölder.
    \end{itemize}
\end{lemma}

Next, we introduce a method to construct straightening maps on a positive measure set with \Cref{straightening lemma}. First, fix an $A_0>0$   such that   $$\Lambda=\Lambda_{A_0}:=\{x\in\Gamma':A(x)\leq A_0\}$$ with $m(\Lambda)=m(\Lambda_{A_0})>0$, and choose a density point $a_0 \in\Lambda_{A_0}$ for the measure $m|_{\Lambda_{A_0}}$.

Given $\delta_0>0$, we can choose $\tau_0<\delta_0$ small enough depending on $\delta_0$   such that the following holds:  for any $$x,\omega_0\in G_{\tau_0}:=\Lambda\cap\exp_{a_0 }(R_{a_0}(\frac{\tau_0}{100}A_0^{-1})), $$
we have that $$\exp_{\omega_0 }^{-1}\circ\exp_x(W^{i}_{x,\delta_0}(x))\cap R_{\omega_0}(2\tau_0A_0^{-1})$$ 
is the graph of a function $\varphi:R^{(i)}_{\omega_0}(2\tau_0A_0^{-1})\rightarrow R_{\omega_0}^{{\widehat{(i)}}}(2\tau_0A^{-1}_0)$ with slope $\leq\frac{2\gamma_0}{C_1}$ 
, and for any $$y\in \exp_x(W^i_{x,\delta_0}(x))\cap \exp_{\omega_0 }(R_{\omega_0}(\tau_0A_0^{-1})),$$ we have that $$\exp_{\omega_0 }^{-1}\circ\exp_x(W^{i-1}_{x,\delta_0}(y))\cap R_{\omega_0}(2\tau_0A_0^{-1})$$ 
is the graph of a function $\varphi:R_{\omega_0}^{(i-1)}(2\tau_0A_0^{-1})\rightarrow R_{\omega_0}^{{\widehat{(i-1)}}}(2\tau_0A^{-1}_0)$ with slope $\leq\frac{2\gamma_0}{C_1}$.

Let $p_{\omega_0 }:T_{\omega_0 }M\rightarrow E^{(i)}_{\omega_0 }$ denote the projection map along the $E^{\widehat{(i)}}_{\omega_0 }$ direction. 
By our previous assumption that the parameter $\tau_0$ is sufficiently small, the containment relation holds as follows: $$R_{\omega_0}^{(i)}(\tau_0A_0^{-1})\subseteq p_{\omega_0 }\circ\exp_{\omega_0 }^{-1}\circ\exp_x(W^{i}_{x,\delta_0}(x)).$$

\begin{figure}[htbp]
    \centering
    \includegraphics[width=1\textwidth]{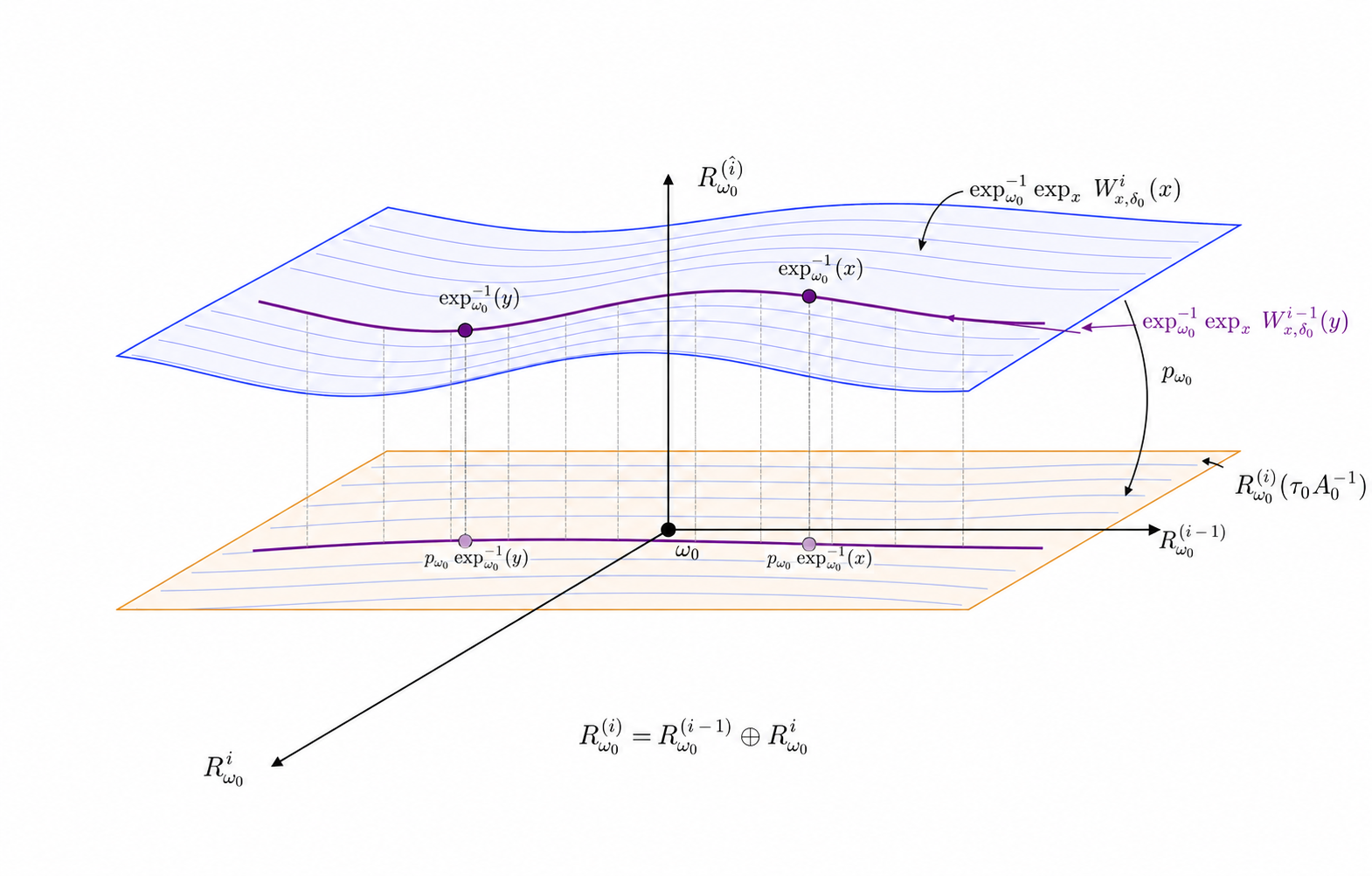}
    \caption{illustration for straightening.}
\end{figure}

Fix an $\omega_0\in G_{\tau_0}$ which is a density point of $m|_{G_{\tau_0}}$. Analogous to \cite[Section 8.4]{LEDRAPPIER_YOUNG_B},
for any $y\in W^i_{x,\delta_0}(x)$, 
$$p_{\omega_0 }\circ\exp_{\omega_0 }^{-1}\circ\exp_x (W^{i-1}_{x,\delta_0}(y))$$ forms a foliation inside $R^{(i)}_{\omega_0}(\tau_0A_0^{-1})$, which is contained in $$p_{\omega_0 }\circ\exp_{\omega_0 }^{-1}\circ\exp_x(W^{i}_{x,\delta_0}(x)).$$ 
Moreover, for any $x\in G_{\tau_0}$, the straightening lemma \Cref{straightening lemma} induces a map $$O_{\omega_0 ,x}:R_{\omega_0}^{(i)}(\tau_0A_0^{-1})\rightarrow R^{(i)}=R^{(i)}_{\omega_0}$$ that straightens the  leaves contained in $$p_{\omega_0 }\circ\exp_{\omega_0 }^{-1}\circ\exp_x (W^{i-1}_{x,\delta_0}(y)).$$

Now we define $$S  :=  \bigcup_{x\in G_{\tau_0}}\left(\exp_x(W^i_{x,\delta_0}(x))\cap\exp_{a_0 }(R_{a_0}(\tau_0A_0^{-1}))\right)$$ as a positive measure lamination. Define the straightening map $\pi:S\rightarrow R^{(i)}$ as $$\pi(y)=O_{\omega_0 ,x}\circ p_{\omega_0 }\circ \exp_{\omega_0 }^{-1}(y),\quad \forall y\in\exp_x(W^i_{x,\delta_0}(x)).$$ Then if $\tau_0$ is small enough, we have 
\begin{itemize}
    
\item[(1)] $S$ is the disjoint union of local unstable manifolds $D_\alpha$ where $D_\alpha\cap G_{\tau_0}\neq \emptyset$, and if $x\in D_\alpha\cap \Lambda_{A_0}$, then $D_\alpha=S\cap \exp_x(W^i_{x,\delta_0}(x))$.
\item[(2)] $\pi$ is continuous. Moreover, for any $\pi|_{D_\alpha}$, we note that since all the $p_{\omega_0 },\exp,\exp^{-1}$ in the definition of $\pi|_{D_\alpha}$ are uniformly bi-Lipschitz, and $O_{\omega_0 ,x}$ is uniformly $\alpha$-bi-Hölder with the constant depending on $A_0$ by \Cref{uniform Hölder} and \Cref{straightening lemma}. Thus, $\pi|_{D_\alpha}$ is $\alpha$-bi-Hölder with the constant depending on $A_0$.
\end{itemize}

\section{Partial entropy}\label{section partial entropy}
The purpose of this section is to define partial entropy for $C^1$ diffeomorphisms with an appropriate dominated splitting. Following the approach of \cite[Section 9]{LEDRAPPIER_YOUNG_B}, we will provide two equivalent characterizations: the first in terms of subordinate partitions and the second using Bowen balls. Since the discussion is standard, several details are omitted.

We first construct the measurable partitions subordinate to $W^i$ and $W^{i-1}$.
Leaving $S=\bigcup D_{\alpha}$ as before, we further define that
\begin{equation}\label{equation 19}
\hat{\xi}_i(x)=\left\{
    \begin{aligned}
        &D_\alpha, & x\in D_\alpha,\\
        &M-S,      & x\notin S,
    \end{aligned}
    \right .
\end{equation}

\begin{equation*}
\hat{\xi}_{i-1}(x)=\left\{
    \begin{aligned}
        &D_\alpha\cap W^{i-1}(x), & x\in D_\alpha,\\
        &M-S,      & x\notin S,
    \end{aligned}
    \right .
\end{equation*}
\begin{equation*}
     \xi_i(x)=\hat\xi_i^+(x)
     :=\bigvee_{j=0}^\infty f^j\hat\xi_i(x),
\end{equation*}
\begin{equation*}
    \xi_{i-1}(x)=\hat\xi_{i-1}^+(x):=\bigvee_{j=0}^\infty f^j\hat\xi_{i-1}(x).
\end{equation*}

As in \cite[Section 3]{Part1}, following \cite{LEDRAPPIER_YOUNG_B} and \cite{Ledrappier_Strelcyn_1982} we can obtain that:
\begin{lemma}
    There exists $\tau_0$ small enough and $\xi_i,\ \xi_{i-1}$ as above, such that $\xi_{i-1}$ is subordinate to $W^{i-1}$, $\xi_i$ is subordinate to $W^i$ and $\xi_{i}\prec\xi_{i-1}$.
\end{lemma}

By the same arguments as in \cite[Section 3]{Part1}, the conditional entropy $H(\xi_{i}|f\xi_{i})$ and $H(\xi_{i-1}|f\xi_{i-1})$ do not depend on the choice of $\xi_i$ and $\xi_{i-1}$. Therefore, the partial entropies are defined as $h_i  :=  H(\xi_{i}|f\xi_{i})$ and $h_{i-1}  :=  H(\xi_{i-1}|f\xi_{i-1})$. By definition, we have $h_u=h^u_m(f)$, where $h^u_m(f)$ is the $u$-entropy defined in  \cite{Part1}.

In the remainder of this section, following the approach in \cite[Section 9]{LEDRAPPIER_YOUNG_B}, we present a definition of partial entropy based on Bowen balls and prove the equivalence between this definition and the previously established one.

Recall \Cref{bowen ball def 1} and \Cref{bowen ball def 2} in \Cref{section 7.2}. Denote
\begin{equation*}
     \uh_i  :=  \lim_{\epsilon\rightarrow 0}\uline h_i(x,\epsilon,\xi_i)
\end{equation*}
and
\begin{equation*}
     \oh_i  :=  \lim_{\epsilon\rightarrow 0}\oh_i(x,\epsilon,\xi_i).
\end{equation*}
The limit exists because $\uline h_i(x,\epsilon,\xi_i)$ and $\oh_i(x,\epsilon,\xi_i)$ increase as $\epsilon\rightarrow 0$.
Moreover, it is easy to verify that the two limits $\uh_i(x,\xi_i)$ and $\oh_i(x,\xi_i)$ are constant along the orbits and independent of the choice of $\xi_i$. Since $m$ is ergodic, for $m$-a.e. $x\in M$, we have $\uh_i(x,\xi_i)=\uline h_i$ and $\oh_i(x,\xi_i)=\oh_i$. In the rest of this section, we prove that $$\uh_i=\oh_i=h_i.$$
Similarly, we can define $\uh_{i-1}$ and $\oh_{i-1}$ and establish the relationship that $$\uh_{i-1}=\oh_{i-1}=h_{i-1}.$$
Recall that
$$V^i(x,n,\epsilon):=\{y\in W^i(x):d^i(f^k(x),f^k(y))<\epsilon,\ \forall 0\leq k<n\},$$
where $d^i$ is the distance induced by the Riemannian submanifold metric on $W^i$.

In this section, $\{m^i_x\}$ denotes the conditional measures of $m$ associated with $\xi_i$. The following \Cref{lemma 15} establishes that the partial entropy defined via Bowen balls is bounded below by the conditional entropy. The proof follows the same steps as in \cite[Section 9.2]{LEDRAPPIER_YOUNG_B} and is included here for completeness.
\begin{lemma}\label{lemma 15}
    $\uh_i\geq h_i$.
\begin{proof}
    For any $\epsilon>0$ and $\delta>0$, we define
    \[
    A_\delta:=\{x:B^i(x,\delta)\subseteq\xi_i(x)\}
    \]
    and
    \[
    U^i(x,n,\delta):=\bigcap_{0\leq j< n, f^jx\in A_\delta} (f^{-j}\xi_i)(x),
    \]
    where $B^i(x,\delta):=\{y\in W^i(x):d^i(x,y)<\delta\}$.
    By the subordinate property of $\xi_i$, one has $m(A_\delta)\to1$ as $\delta\to 0$. Let $g(x):=-\log m^i_x((f^{-1}\xi_i)(x))$ and choose $\delta'$ small enough such that
    \[
    \int_{A_{\delta'}}g(x)\,\td m(x)\geq\int g(x)\,\td m(x)-\epsilon= h_i-\epsilon.
    \]
    Thus, whenever $\delta<\delta'$, by the definition we have $V^i(x,n,\delta)\subseteq U^i(x,n,\delta)$ and $$-\log m^i_x( U^i(x,n,\delta))\geq \sum^{n-1}_{j=0}(\chi_{A_\delta}\cdot g)(f^j(x)).$$ Therefore
    \begin{align*}
        \uh_i(x,\delta,\xi_i)&=\liminf_{n\to\infty}-\frac{1}{n}\log m^i_x(V^i(x,n,\delta))\\
        &\geq \int_{A_\delta} g(x)\,\td m(x)\\
        &\geq h_i-\epsilon
    \end{align*}
    and the proof is complete.
\end{proof}
\end{lemma}
Then it only remains to prove $\oh_i\leq h_i$. First, we recall a Shannon-McMillan-Breiman type theorem for partial entropy from \cite{LEDRAPPIER_YOUNG_B}.
\begin{lemma}\cite[Lemma 9.3.1]{LEDRAPPIER_YOUNG_B}\label{SMB type thm}
    Let $\tP$ be a measurable partition with finite entropy, for $m$-a.e.\ $x$, 
    \begin{equation*}
        \lim_{n\rightarrow\infty}-\frac{1}{n}\log m^i_x(\tP\vee\xi_i)_0^n(x)=h_i
    \end{equation*}
    where $\tP_0^n(x)  :=  (\bigvee_{j=0}^{n-1}f^{-j}\tP) (x)$.
\end{lemma}

Subsequently, we need to construct a partition $\tP$   such that   $\tP_0^n(x)$ can be compared to $V^i(x,n,\epsilon)$. Denote $S'  :=  S\cap \Lambda_{A_0}$, which is a set with a positive measure by definition. The return-time functions $n_+,n_-,n_0: M\rightarrow\mathbb N$ are defined as
\begin{equation*}
    n_+(x)=\inf\{n>0:f^n(x)\in S'\},
\end{equation*}
\begin{equation*}
    n_-(x)=\inf\{n>0:f^{-n}(x)\in S'\},
\end{equation*}
and
\begin{equation*}
    n_0(x)=\inf\{n\geq 0:f^n(x)\in S'\}.
\end{equation*}
Define the function $\psi:M\rightarrow R^{+}$ by:
    \begin{equation*}
	\psi(x)=\left\{
	\begin{aligned}
		&1  , &x&\notin S'\\
            &\frac{\epsilon' }{100C_0^{10}C^N_f} \delta_0A_0^{-1}C^{-2\max\{n_+(x),n_-(x)\}}_f  , & x&\in S'
	\end{aligned}
	\right.
\end{equation*} 
where $\epsilon'$ is a preassigned small number, $C_f$ and $C_0$ are the constants defined in \Cref{section fake foliation}, and $N$ is given by \Cref{ABC ergodic lemma}.
By Kac's lemma, or equivalently by the integrability of the return-time functions to \(S'\), we have \[\int -\log\psi\,\td m<\infty.\]

Since $\int-\log\psi\,\td m<\infty$, by \cite[Lemma 2]{Mañé_1981}, there exists a finite-entropy partition $\tP$  such that   $\tP(x)\subseteq B(x,\psi(x))$, where $B(x,r)$ denotes the ball with center $x$ and radius $r$. Then the following lemma can be established:
\begin{lemma}\label{Lemma 9.3.3}
For any $\epsilon'>0$, there exists a finite-entropy partition $\tP$   such that   for $m$-a.e. $x\in M$,
\begin{equation*}
    \tP^n_0(x)\cap \xi_i(x)\subseteq V^i(x,n,\epsilon'),\ \forall n>n_+(x).
\end{equation*}
\begin{proof}
Consider an arbitrary point $x$ in a set with full $m$ measure such that $f^n(x)\in S'$ holds infinitely many times as $n\rightarrow\pm\infty$. For any $n>n_+(x)$ and $y\in \tP^n_0(x)\cap \xi_i(x)$, we only need to prove $d^i(f^j(x),f^j(y))<\epsilon'$ for $0\leq j\leq n-1$.
We first prove that $d^i(f^j(x),f^j(y))<\epsilon'$ for $0\leq j\leq n_+(x)$ and $$f^{n_+(x)}(y)\in\exp_{f^{n_+(x)}(x)}\left(W^i_{f^{n_+(x)}(x),\delta_0}(f^{n_+(x)}(x))\right).$$

  Let $\ell_0$ be the largest non-positive integer such that   $f^{\ell_0}x\in S'$. Since $\xi_i$ is increasing, $$f^{\ell_0}(y)\in \xi_i(f^{\ell_0}(x))\subseteq \exp_{f^{\ell_0}(x)}\left(W^i_{f^{\ell_0}(x),\delta_0}(f^{\ell_0}(x))\right).$$ Moreover, since $|\ell_0|+n_+(x)=n_-(f^{n_+(x)}x)$ and $n_-(f^{n_+(x)}(x))= n_+(f^{\ell_0}(x))$, we have
\begin{equation}\label{equation 30}
\begin{aligned}
    d^i(f^{\ell_0}(x),f^{\ell_0}(y))&\leq C_f^{|\ell_0|+n_+(x)}\cdot \psi(f^{n_+(x)}(x)) \\
    &\leq \frac{\epsilon' }{100C_0^{10}C^N_f}\delta_0A_0^{-1}C^{-\max\{n_+(f^{n_+(x)}(x)),n_-(f^{n_+(x)}(x))\}}_f\\
    &\leq  \frac{\epsilon' }{100C_0^{10}C^N_f}\delta_0A_0^{-1}C^{-n_+(f^{\ell_0}(x))}_f.\\
\end{aligned}
\end{equation}
Together with $d^i(f^{\ell_0+k}(x),f^{\ell_0+k}(y))\leq d^i(f^{\ell_0}(x),f^{\ell_0}(y))\cdot C_f^k$, we obtain 
\begin{equation}\label{eq 36}
    d^i(f^j(x),f^j(y))<\epsilon',\ \forall\ 0\leq j\leq n_+(x)
\end{equation}
and 
\begin{equation*}
f^{n_+(x)}(y)\in\exp_{f^{n_+(x)}(x)}\left(W^i_{f^{n_+(x)}(x),\delta_0}(f^{n_+(x)}(x))\right).
\end{equation*}

If $n-1=n_+(x)$, then the proof is complete. When $n-1>n_+(x)$, we shall first prove the following claim: 
\begin{claim}
    For any $x\in S'$, if $y\in \tP_0^n(x)\cap \exp_x(W^i_{x,\delta_0}(x))$ for some $n\geq 0$, then $d^i(f^j(x),f^j(y))<\epsilon'$ for all $0\leq j\leq n-1$.
\begin{proof}
    By the same discussion as above with $\ell_0=0$, it can be proved that if $x\in S'$, $y\in \tP(x)\cap \exp_x(W^i_{x,\delta_0}(x))$, then $d^i(f^j(x),f^j(y))<\epsilon'$ for $0\leq j\leq n_+(x)$ and $$f^{n_+(x)}(y)\in\exp_{f^{n_+(x)}(x)}\left(W^i_{f^{n_+(x)}(x),\delta_0}(f^{n_+(x)}(x))\right).$$ If $n-1\leq n_+(x)$, the proof is complete. If $n-1>n_+(x)$, together with $y\in\tP^n_0(x)$, it follows that $f^{n_+(x)}(y)\in\tP(f^{n_+(x)}(x))$ and $$f^{n_+(x)}(y)\in\exp_{f^{n_+(x)}(x)}\left(W^i_{f^{n_+(x)}(x),\delta_0}(f^{n_+(x)}(x))\right).$$ 
    Thus, previous discussion is also available for $f^{n_+(x)}(y)$. 
    By continuing this process, the claim can be proved inductively.
\end{proof} 
\end{claim}

Now we can apply the claim to $f^{n_+(x)}(x)$ and $f^{n_+(x)}(y)$, which implies that $d^i(f^j(x),f^j(y))<\epsilon'$ for $n_+(x)\leq j\leq n-1$. Together with \Cref{eq 36}, the proof is complete.
    
\end{proof}
\end{lemma}

Finally, combining \Cref{SMB type thm} and \Cref{Lemma 9.3.3}, for $m$-a.e. $x\in M$ and any small $\epsilon'>0$, one has
\begin{equation}\label{eq 33}
\begin{aligned}
    \oh_i(x,\epsilon',\xi_i)&=\limsup_{n\rightarrow \infty}-\frac{1}{n}\log m^i_xV^i(x,n,\epsilon')\\
    &\leq \limsup_{n\rightarrow\infty} -\frac{1}{n}\log m^i_x \tP^n_0(x)\\
    &\leq \lim_{n\rightarrow\infty} -\frac{1}{n}\log m^i_x(\xi_i\vee \tP)_0^n(x)\\
    &=h_i.
\end{aligned}
\end{equation}
Thus, $$\oh_i\leq h_i.$$ Together with \Cref{lemma 15}, we obtain $\uh_i=\oh_i=h_i$. Hence, the proof of \Cref{proposition 1} is complete.

\section{Lower bound for transverse entropy}\label{section 10}
In this section, which is analogous to \cite[Section 10]{LEDRAPPIER_YOUNG_B}, we prove the lower bound estimate for the transverse entropy:
$$h_i-h_{i-1}\geq \lambda_i\cdot\max\{\udelta_i-\udelta_{i-1},\odelta_i-\odelta_{i-1}\}.$$ 
The argument in \cite[Section 10]{LEDRAPPIER_YOUNG_B} establishes the estimate only for the upper dimensions. However, in this paper, we require this estimate for both the upper and lower dimensions, since we only assume that the bundle $E^i$ is one-dimensional and thus the induction arguments in \cite{LEDRAPPIER_YOUNG_B} are not available.

Let ${m^i_x}$ and ${m^{i-1}_x}$ be the conditional measures of $m$ associated with the measurable partitions $\xi_i$ and $\xi_{i-1}$, respectively.
We first prove that $$h_i-h_{i-1}\geq \lambda_i\cdot(\udelta_i-\udelta_{i-1}).$$ In preparation, we establish the following lemma:

\begin{lemma}
    There exists a finite-entropy partition $\tP$ and a measurable function $n_5:M\rightarrow \mathbb N^+$ 
     such that   the following properties are satisfied for any $n>n_5(x)$,
    \begin{align*}
    1): &\quad\frac{\log m^i_x(B^i(x,2\e^{-(\lambda_i-3\epsilon)n} )) }{-n(\lambda_i-3\epsilon)}\geq \udelta_i-\epsilon. \\
    2): &\quad-\frac{1}{n}\log m^{i-1}_x\tP^n_0(x)\geq h_{i-1}-\epsilon. \\
    3): &\quad\xi_i(x)\cap \tP^n_0(x)\subseteq B^i(x,\e^{-n(\lambda_i-3\epsilon)}).   \\
    4): &\quad-\frac{1}{n}\log m^i_x\tP^n_0(x)\leq h_i+\epsilon.        
    \end{align*}
\begin{proof}
    By the definition of $\udelta_i$, there exists $n_1:M\rightarrow\mathbb N^+$   such that   when $n>n_1(x)$, property 1) holds.

    By \Cref{proposition 1}, we can choose $\epsilon'>0$ and $n_2:M\rightarrow\mathbb N^+$   such that   for $m$-a.e. $x$ and $n>n_2(x)$, $$-\frac{1}{n}\log m^{i-1}_xV^{i-1}(x,n,\epsilon')\geq h_{i-1}-\epsilon.$$ Furthermore, by \Cref{Lemma 9.3.3} there exists a finite-entropy partition $\tL$ and a function $n_2':M\rightarrow\mathbb N^+$   such that   for $m$-a.e. $x\in M$, $$\tL^n_0(x)\cap \xi_{i-1}(x)\subseteq V^{i-1}(x,n,\epsilon').$$ 
    With the help of the local argument (see \Cref{local arg here}), we can prove the following claim.
    \begin{claim}\label{claim 19}
     There exists a finite-entropy partition $\tR$ and $n_3:M\rightarrow \mathbb N^+$   such that   for $m$-a.e. $x\in M$, $n>n_3(x)$ we have $$\xi_i(x)\cap \tR_0^n(x)\subseteq B^i(x,\e^{-n(\lambda_i-3\epsilon)}).$$
     \begin{proof}
         Let $\tR$ be the partition with finite entropy in \Cref{Lemma 9.3.3} with parameter $\epsilon'$. If $f^n(x)\in S'$, $n\geq 0$, then it follows from \Cref{Lemma 9.3.3} that $$f^n(\xi_i(x)\cap \tR_0^n(x))\subseteq \exp_{f^nx}(W^i_{f^nx,\delta_0}(f^nx)).$$ Fix a point $x$ such that $f^nx\in S'$ holds infinitely often as $n\rightarrow\pm\infty$ and let $n_3(x)=n_0(x)$. When $n>n_0(x)$, take any $y\in\xi_i(x)\cap\tR^n_0(x)$.

         Let $k$ be the largest integer less than  $n$   such that   $f^k(x)\in S'$. Since $f^k(y)\in \tR(f^k(x))$, it follows from the construction of partition $\tR$ that $$f^k(y)\in B(f^k(x),\psi(f^k(x))).$$ Note that $n_+(f^k(x))\geq n-k$. Moreover, since $f^k(y)\in \exp_{f^k(x)}\left(W^i_{f^k(x),\delta_0}(f^k(x))\right)$, by the local argument (see \Cref{local arg here}) we have
         \begin{equation}\label{local argument example}
         \begin{aligned}
             d^i(x,y)&\leq 4|\exp_x^{-1}y|\\
             &\leq 4\cdot |\exp_{f^k(x)}^{-1}f^k(y)|\cdot \e^{-k(\lambda_i-3\epsilon)}C^N_f A(f^k(x))\\
             &\leq 4\cdot 2\psi(f^k(x))\cdot \e^{-k(\lambda_i-3\epsilon)}C^N_f A(f^k(x))\\
             &\leq \e^{-n(\lambda_i-3\epsilon)},
         \end{aligned}
         \end{equation}
         where $C_f\geq 100\max\{|Df^{\pm}|,\mathrm{e}^{|\lambda_1|+100\epsilon},\mathrm{e}^{|\lambda_r|+100\epsilon}\}$ is defined in \Cref{section fake foliation}.
         So $$\xi_i(x)\cap \tR^n_0(x)\subseteq B^i(x,\e^{-n(\lambda_i-3\epsilon)})$$ whenever $n>n_0(x)=n_3(x)$. Thus, the proof of the claim is complete.
     \end{proof}
    \end{claim}
\begin{remark}[The local argument]\label{local arg here}
    This is the first instance where \textbf{the local argument} is applied in this paper. The method is designed to translate properties from the tangent space to a local neighborhood by \Cref{local argument 1} and \Cref{local argument 2}. By extending the hyperbolic-type control provided by \Cref{ABC ergodic lemma}, it allows us to estimate the length growth along fake leaves. More specifically, this argument establishes that:
    $$|\exp_x^{-1}y|\leq |\exp_{f^k(x)}^{-1}f^k(y)|\cdot \e^{-k(\lambda_i-3\epsilon)}C^N_f A(f^k(x)).$$
    For the reader's convenience, we include the details as follows: Recall \Cref{ABC ergodic lemma}, \Cref{local argument 1} and \Cref{local argument 2}. Then we have
    \begin{align*}
         |\exp_x^{-1}y|&=|\exp_x^{-1}y-\exp^{-1}_xx|\\
         &\leq 2\cdot d_{\tW^{i}}(\exp_x^{-1}y,\exp^{-1}_xx)\\
         &\leq 2\cdot\int_{0}^{\exp^{-1}_{f^{k}(x)}f^{k}(y)}  \|D\tg^{-k}(z)|_{T_z\tW^{i}}\|\,\td z\\
         &\leq \frac{1}{2} C^N_f\cdot \int_{0}^{\exp^{-1}_{f^{k}(x)}f^{k}(y)} \|D\tg^{-KN}(z)|_{T_z\tW^{i}}\|\,\td z\quad \text{(where $K:=[\frac{k}{N}]$)}\\
         &\leq \frac{1}{2}C^N_f\cdot \int_{0}^{\exp^{-1}_{f^{k}(x)}f^{k}(y)}\|D\tg^{-KN}(0)|_{E^{(i)}_{f^{k}(x)}}\|\cdot\e^{2KN\epsilon}\,\td z\quad \text{(by \ref{local argument 1} and \ref{local argument 2})}\\
         & \leq \frac{1}{2}C^N_f\cdot \e^{2
         k\epsilon}\cdot A(f^{k}(x))\cdot \e^{-k(\lambda_i-\epsilon)}\cdot d_{\tW^i}(\exp_{f^{k}(x)}^{-1}f^{k}(y),0)\quad\text{(by \Cref{ABC ergodic lemma})}\\
         & \leq |\exp_{f^k(x)}^{-1}f^k(y)|\cdot \e^{-k(\lambda_i-3\epsilon)}C^N_f A(f^k(x)).
    \end{align*}
    The same technique is also used in Part 1 (cf. \cite[Remark 3]{Part1}). We refer to this technique as \textbf{the local argument} and will use it frequently in the rest of this paper. For simplicity, we will omit the details and directly state the result whenever the same argument is used again. 
\end{remark}
    
Now, let $\tP=\tR\vee\tL$. It is a finite-entropy partition because both $\tR$ and $\tL$ are finite entropy. Since $\tP^n_0(x)\supseteq (\xi_i\vee\tP)_0^n(x)$, by \Cref{SMB type thm} there exists $n_4:M\rightarrow \mathbb N$   such that   for $m$-a.e. $x\in M$ and $n>n_4(x)$, property 4) holds. Finally, we choose $\tP$ and $n_5  := \max\{n_1,n_2,n_2',n_3,n_4\}$, and it is easy to verify that when $n>n_5(x)$, properties 1)-4) are satisfied.
\end{proof}
\end{lemma}

Choose $\tP$ and $n_5$ as above. Then we claim that there exists $N_0>0$ such that the set $\Gamma_{N_0}=\{x\in M:n_5(x)\leq N_0\}$ has a positive measure, and for $m$-a.e. $x\in\Gamma_{N_0}$, there exists $n(x)\geq N_0$  such that   the following properties 5)-8) are satisfied for $n=n(x)$.
\begin{align*}
5):&\quad L  :=  B^{i-1}(x,\e^{-n(\lambda_i-3\epsilon)})\subseteq \xi_{i-1}(x). \\    
6):&\quad \frac{m^{i-1}_x(L\cap \Gamma_{N_0})}{m^{i-1}_x(L)}\geq\frac{1}{2}.   \\
7): &\quad \frac{\log m^{i-1}_x(B^{i-1}(x,\e^{-n(\lambda_{i}-3\epsilon)}))}{-n(\lambda_i-3\epsilon)}\leq \udelta_{i-1}+\epsilon.  \\
8): &\quad\frac{\log 2}{n}\leq\epsilon.   
\end{align*}
Property 5) holds for any $n$ large enough (depending on $x$) because $\xi_{i-1}$ is subordinate to $W^{i-1}$. Thus, for $m$-a.e. $x\in M$, one has that $\xi_{i-1}(x)$ contains an open neighborhood of $x$ in $W^{i-1}(x)$. 
Property 6) holds for any $n$ large enough because almost every $x$ is a density point of $m^{i-1}_x$.
Property 7) holds for some $n$ large enough by the definition of $\udelta_{i-1}$ and Property 8) holds for any $n$ large enough.

Then, by the same arguments as in \cite[Section 10.2]{LEDRAPPIER_YOUNG_B}, we complete the proof of this section. For completeness, the details are presented here.

First, by 6) and 7),  $$m^{i-1}_x(L\cap\Gamma_{N_0})\geq\frac{1}{2}m^{i-1}_x(L)\geq\frac{1}{2}\e^{-n(\lambda_i-3\epsilon)(\udelta_{i-1}+\epsilon)}.$$ For $y\in L\cap\Gamma_{N_0}$, by 2) we have $$m^{i-1}_x(\tP^n_0(y))\leq\e^{-(h_{i-1}-\epsilon)n}.$$ Thus, the number of atoms of $\tP^n_0$ that have a non-empty intersection with $L\cap\Gamma_{N_0}$ is greater than or equal to $$ \frac{1}{2}\e^{-n(\lambda_i-3\epsilon)(\udelta_{i-1}+\epsilon)+(h_{i-1}-\epsilon)n}.$$
Note that the number of atoms of $\tP^n_0$ which have a non-empty intersection with $L\cap\Gamma_{N_0}$ is equal to the number of atoms of $\tP^n_0\vee \xi_{i}$ that have a non-empty intersection with $L\cap\Gamma_{N_0}$ since $L\subseteq\xi_{i-1}(x)\subseteq \xi_i(x)$. By 3), when $y\in\Gamma_{N_0}$, we have $$\xi_i(y)\cap\tP^n_0(y)\subseteq B^i(y,\e^{-n(\lambda_i-3\epsilon)}).$$ Hence, the atoms of $\tP^n_0\vee \xi_{i}$ that have a non-empty intersection with    $L\cap\Gamma_{N_0}$ are contained in $B^i(x,2\e^{-n(\lambda_i-3\epsilon)})$. Since \(L\subseteq \xi_i(x)\), for every \(y\in L\cap\Gamma_{N_0}\) we have
\(\xi_i(y)=\xi_i(x)\) and hence \(m_y^i=m_x^i\). Together with 4), it follows that
\begin{equation*}
\begin{aligned}
    m^i_x(B^i(x,2\e^{-n(\lambda_i-3\epsilon)}))&\geq \#\{\text{distinct atoms of $\tP^n_0\vee \xi_{i}$ intersecting $L\cap\Gamma_{N_0}$}\}\\
    &\ \ \times \{\text{minimal measure of such atoms}\}\\
    &\geq \frac{1}{2}\e^{-n(\lambda_i-3\epsilon)(\udelta_{i-1}+\epsilon)+(h_{i-1}-\epsilon)n}\cdot\e^{-n(h_i+\epsilon)}.
\end{aligned}   
\end{equation*}
Together with 1),
\begin{equation*}
    (\udelta_i-\epsilon)(\lambda_i-3\epsilon)\leq (\udelta_{i-1}+\epsilon)(\lambda_i-3\epsilon)+\frac{\log 2}{n}+h_i-h_{i-1}+2\epsilon.
\end{equation*}
Simplifying, one has
\begin{equation*}
    \udelta_i-\udelta_{i-1}-2\epsilon\leq \frac{h_i-h_{i-1}+3\epsilon}{\lambda_i-3\epsilon}.
\end{equation*}
Letting $\epsilon\rightarrow 0$, we obtain $h_i-h_{i-1}\geq\lambda_i(\udelta_i-\udelta_{i-1})$.

The proof for the corresponding inequality $h_i-h_{i-1}\geq\lambda_i(\odelta_i-\odelta_{i-1})$ is analogous and requires only minor adjustments. 
One simply needs to interchange the positions of properties 1) and 7) in the proof for $\odelta_i$ and $\odelta_{i-1}$. 
The argument remains valid as long as one of these properties holds for all sufficiently large~$n$, while the other holds for some sufficiently large~$n$.


\section{Entropy formulas}\label{section 11}
In this section, we finish the proof of the entropy formulas in \Cref{section 7.2}. In \Cref{sec 5.1}, some general results in dimension theory are introduced for preparations. In \Cref{upper estimate}, we prove the upper bound estimate for transverse entropy. Together with the lower bound obtained in \Cref{section 10}, we complete the proof of the entropy formulas in \Cref{sec 5.3}.

\subsection{Lemmas on dimension estimation}\label{sec 5.1}
Before proving the main theorems, we begin with preparations concerning some general results in dimension theory. We prepare four lemmas in measure dimension theory. Two more lemmas are required with respect to \cite[Section 11]{LEDRAPPIER_YOUNG_B}, since we need lemmas to deal with the upper dimension.

\begin{lemma}\label{upper transverse dimension lemma}
Let $m$ be a probability measure on $R^p\times R^q$, where $R^n$ denotes the $n$-dimensional Euclidean space. $\pi:R^p\times R^q\rightarrow R^p$ is the projection onto $R^p$ and $\{m_s\}_{s\in R^p}$ is the disintegration of $m$ associated with $\pi$. Define
$$\gamma(s)  :=  \liminf_{\rho\rightarrow0}\frac{\log m\circ \pi^{-1}(B^p(s,\rho))}{\log \rho},\ s\in R^p$$
and assume that there exists $\delta\geq 0$ such that for $m$-a.e. $(s,t)\in R^p\times R^q$, 
$$\limsup_{\rho\rightarrow0}\frac{\log m(B((s,t),\rho))}{\log\rho}\leq\delta.$$
Then for $m$-a.e. $(s,t)\in R^p\times R^q$,
$$\limsup_{\rho\rightarrow 0}\frac{\log m_s(B^q(t,\rho))}{\log\rho}\leq\delta-\gamma(s).$$
\begin{proof}
Let
\[
\nu:=m\circ \pi^{-1}.
\]
We fix the disintegration
$$
m=\int m_s\,\td\nu(s),
$$
and note that for $m$-a.e. $(s,t)$,
$$
0\leq \gamma(s)\leq \delta.
$$

Fix $h>0$ and $\sigma>0$. For each $j\geq0$, set
\[
a_j:=jh,
\qquad
E_j:=\{s\in\mathbb R^p: a_j\leq \gamma(s)<a_j+h\}.
\]
The sets $E_j$ form a countable partition of the set where $\gamma$ is finite. Since
$\gamma(s)\leq\delta$ for $m$-a.e. $(s,t)$, it is enough to prove the desired estimate
on each $E_j\times\mathbb R^q$.

For $n\geq1$, write $r_n=\e^{-n}$ and define
\[
B_n(\sigma):=
\left\{
(s,t):
m(B((s,t),r_n))\geq \e^{-n(\delta+\sigma)}
\right\},
\]
\[
C_{n,j}(\sigma):=
\left\{
(s,t):
\nu(B^p(s,r_n))\leq \e^{-n(a_j-\sigma)}
\right\},
\]
and
\[
A_{n,j}(\sigma):=
\left\{
(s,t):s\in E_j,\ 
m_s(B^q(t,2r_n))\leq \e^{-n(\delta-a_j+3\sigma)}
\right\}.
\]
We claim that
\[
m\left(\limsup_{n\to\infty} A_{n,j}(\sigma)\right)=0
\quad
\text{on }E_j\times\mathbb R^q .
\]

Let
\[
D_{n,j}:=A_{n,j}(\sigma)\cap B_n(\sigma)\cap C_{n,j}(\sigma).
\]
We first estimate $m(D_{n,j})$. Fix $z=(s,t)\in D_{n,j}$. By disintegration,
\[
m(A_{n,j}(\sigma)\cap B(z,r_n))
=
\int_{B^p(s,r_n)}
m_{s'}\bigl(
\{u\in\mathbb R^q:(s',u)\in A_{n,j}(\sigma)\cap B(z,r_n)\}
\bigr)\,\td\nu(s').
\]
If the integrand is nonzero for some $s'\in B^p(s,r_n)$, then there exists
$t_0\in\mathbb R^q$ such that
\[
(s',t_0)\in A_{n,j}(\sigma)\cap B(z,r_n).
\]
For every $u$ such that $(s',u)\in A_{n,j}(\sigma)\cap B(z,r_n)$, we have
\[
u\in B^q(t_0,2r_n).
\]
Since $(s',t_0)\in A_{n,j}(\sigma)$, it follows that
\[
m_{s'}(B^q(t_0,2r_n))
\leq
\e^{-n(\delta-a_j+3\sigma)}.
\]
Consequently,
\[
m_{s'}\bigl(
\{u:(s',u)\in A_{n,j}(\sigma)\cap B(z,r_n)\}
\bigr)
\leq
\e^{-n(\delta-a_j+3\sigma)}.
\]
Integrating over $B^p(s,r_n)$ gives
\[
m(A_{n,j}(\sigma)\cap B(z,r_n))
\leq
\e^{-n(\delta-a_j+3\sigma)}
\nu(B^p(s,r_n)).
\]
Since $z\in C_{n,j}(\sigma)$, we get
\[
m(A_{n,j}(\sigma)\cap B(z,r_n))
\leq
\e^{-n(\delta-a_j+3\sigma)}\e^{-n(a_j-\sigma)}
=
\e^{-n(\delta+2\sigma)}.
\]
On the other hand, $z\in B_n(\sigma)$ implies
\[
m(B(z,r_n))\geq \e^{-n(\delta+\sigma)}.
\]
Therefore,
\[
m(A_{n,j}(\sigma)\cap B(z,r_n))
\leq
\e^{-n\sigma}m(B(z,r_n)).
\]

The family $\{B(z,r_n):z\in D_{n,j}\}$ covers $D_{n,j}$. By the Besicovitch
covering lemma, there exists a countable subcover $\mathcal B_{n,j}$ with
multiplicity bounded by a constant $C=C(p+q)$. Thus,
\[
\begin{aligned}
m(D_{n,j})
&\leq
\sum_{B\in\mathcal B_{n,j}} m(D_{n,j}\cap B)  \\
&\leq
\sum_{B\in\mathcal B_{n,j}} m(A_{n,j}(\sigma)\cap B)  \\
&\leq
\e^{-n\sigma}\sum_{B\in\mathcal B_{n,j}}m(B) \\
&\leq
C \e^{-n\sigma}.
\end{aligned}
\]
Hence
\[
\sum_{n=1}^\infty m(D_{n,j})<\infty.
\]
By the Borel-Cantelli lemma,
\[
m\left(\limsup_{n\to\infty}D_{n,j}\right)=0.
\]

We now remove the auxiliary sets $B_n(\sigma)$ and $C_{n,j}(\sigma)$. By the
assumption of the lemma, for $m$-a.e. $(s,t)$,
\[
(s,t)\in B_n(\sigma)
\]
for all sufficiently large $n$. Moreover, if $s\in E_j$, then $\gamma(s)\geq a_j$;
hence, by the definition of $\gamma$, we have
\[
\nu(B^p(s,r_n))\leq \e^{-n(a_j-\sigma)}
\]
for all sufficiently large $n$, that is,
\[
(s,t)\in C_{n,j}(\sigma)
\]
eventually. Therefore, on $E_j\times\mathbb R^q$,
\[
m\left(\limsup_{n\to\infty} A_{n,j}(\sigma)\right)=0.
\]

Since $j$ ranges over a countable set, we conclude that for $m$-a.e. $(s,t)$ there
exists $j\geq0$ with $s\in E_j$ such that, for all sufficiently large $n$,
\[
(s,t)\notin A_{n,j}(\sigma).
\]
Equivalently,
\[
m_s(B^q(t,2\e^{-n}))
>
\e^{-n(\delta-a_j+3\sigma)}
\]
for all sufficiently large $n$.

We now pass to arbitrary $\rho>0$. If
\[
2\e^{-(n+1)}\leq \rho\leq 2\e^{-n},
\]
then
\[
B^q(t,\rho)\supseteq B^q(t,2\e^{-(n+1)}).
\]
Hence, for all sufficiently large $n$,
\[
m_s(B^q(t,\rho))
\geq
m_s(B^q(t,2\e^{-(n+1)}))
>
\e^{-(n+1)(\delta-a_j+3\sigma)}.
\]
Since $\log\rho\leq \log 2-n<0$, it follows that
\[
\frac{\log m_s(B^q(t,\rho))}{\log\rho}
\leq
\frac{(n+1)(\delta-a_j+3\sigma)}{n-\log 2}.
\]
Letting $\rho\to0$, equivalently $n\to\infty$, gives
\[
\limsup_{\rho\to0}
\frac{\log m_s(B^q(t,\rho))}{\log\rho}
\leq
\delta-a_j+3\sigma.
\]
Since $s\in E_j$, we have
\[
a_j\leq \gamma(s)<a_j+h,
\]
and therefore
\[
\delta-a_j+3\sigma
\leq
\delta-\gamma(s)+h+3\sigma.
\]
Thus, for $m$-a.e. $(s,t)$,
\[
\limsup_{\rho\to0}
\frac{\log m_s(B^q(t,\rho))}{\log\rho}
\leq
\delta-\gamma(s)+h+3\sigma.
\]
Finally, taking $h\to0$ and $\sigma\to0$ along countable sequences, we obtain
\[
\limsup_{\rho\to0}
\frac{\log m_s(B^q(t,\rho))}{\log\rho}
\leq
\delta-\gamma(s)
\]
for $m$-a.e. $(s,t)\in\mathbb R^p\times\mathbb R^q$. This proves the lemma.
\end{proof}

\end{lemma}

A proof similar to the above ones gives the following lemma. See also \cite[Lemma 12.1.2]{LEDRAPPIER_YOUNG_B} for the ideas of the proof.
\begin{lemma}\label{upper dimension lemma}
    Let $(\Omega,\mu)$ be an arbitrary probability space and $m$ a probability on $\Omega\times R^q$ with $m(\td \omega,\td s)=\int m_\omega(\td s)\mu(\td\omega)$ where $\omega\in\Omega,s\in R^q$. Suppose $\delta\geq0$ such that
    $$\limsup_{\rho\rightarrow0}\frac{\log m(\Omega\times B^q(s,\rho))}{\log\rho}\leq\delta,\ \ m-a.e.\  (\omega,s)\in \Omega\times R^q$$
    then
    $$\limsup_{\rho\rightarrow0}\frac{\log m_\omega(B^q(s,\rho))}{\log\rho}\leq\delta,\ \ m-a.e.\ (\omega,s)\in \Omega\times R^q$$
\end{lemma}

Both \Cref{upper transverse dimension lemma} and \Cref{upper dimension lemma} have a version for the lower dimensions. See \cite[Lemma 11.3.1]{LEDRAPPIER_YOUNG_B} and \cite[Lemma 11.3.2]{LEDRAPPIER_YOUNG_B}.

\begin{lemma}\cite[Lemma 11.3.1]{LEDRAPPIER_YOUNG_B}\label{lower transverse dimension lemma}
Let $m$ be a probability measure on $R^p\times R^q$, where $R^n$ denotes the $n$-dimensional Euclidean space. $\pi:R^p\times R^q\rightarrow R^p$ is the projection onto $R^p$ and $\{m_s\}_{s\in R^p}$ is the disintegration of $m$ associated with $\pi$. Define
$$\gamma(s)  :=  \liminf_{\rho\rightarrow0}\frac{\log m\circ \pi^{-1}(B^p(s,\rho))}{\log \rho},\ s\in R^p$$
and assume that there exists $\delta\geq 0$ such that for $m$-a.e. $(s,t)\in R^p\times R^q$,  
$$\liminf_{\rho\rightarrow 0}\frac{\log m_s(B^q(t,\rho))}{\log\rho}\geq\delta.$$
Then for $m$-a.e. $(s,t)\in R^p\times R^q$,
$$\liminf_{\rho\rightarrow0}\frac{\log m(B((s,t),\rho))}{\log\rho}\geq\delta+\gamma(s).$$
\end{lemma}

\begin{lemma}\cite[Lemma 11.3.2]{LEDRAPPIER_YOUNG_B}\label{lower dimension lemma}
    Let $(\Omega,\mu)$ be an arbitrary probability space and $m$ a probability on $\Omega\times R^q$ with  $m(\td \omega,\td s)=\int m_\omega(\td s)\mu(\td\omega)$ where $\omega\in\Omega,s\in R^q$. Suppose $\delta\geq0$ such that
   $$\liminf_{\rho\rightarrow0}\frac{\log m_\omega(B^q(s,\rho))}{\log\rho}\geq\delta,\ \ m-a.e.\ (\omega,s)\in \Omega\times R^q$$
    then
     $$\liminf_{\rho\rightarrow0}\frac{\log m(\Omega\times B^q(s,\rho))}{\log\rho}\geq\delta,\ \ m-a.e.\ (\omega,s)\in \Omega\times R^q$$
\end{lemma}

\subsection{Upper bound for transverse entropy}\label{upper estimate}
In this section, we prove the upper bound estimate for the transverse entropy: $$\lambda_i\min\{\udelta_i-\udelta_{i-1},\odelta_i-\odelta_{i-1}\}\geq h_i-h_{i-1}$$ using ideas similar to  \cite{Part1} and \cite[Section 11]{LEDRAPPIER_YOUNG_B}. We first prove that  $$\lambda_i(\odelta_i-\odelta_{i-1})\geq h_i-h_{i-1}.$$ 

\subsubsection{Discussions parallel to  \cite{Part1}}

We carry out discussions that are parallel to  \cite{Part1}.
The only difference in this subsection is to substitute $(i-1)$ for $u$, $i$ for $c$, and $\widehat{(i)}$ for $s$.

Consider $A(x),\ \xi_i,\ \xi_{i-1},\ S$ and $S'$ as before. Let $\tP$ be a finite-entropy partition adapted to $\{A(x),\delta_{\text{1}}\}$ as in \cite[Section 3.2]{Part1}. Choose $E\subseteq S'$ to be a positive measure set of diameter sufficiently small and depending only on $A_0$ as in \cite[Section 5.1]{Part1}. We require that $\tP$ refines $\{E,M-E\},\{S,M-S\}$. For $0<\delta<\frac{\delta_\text{1}}{100}$, we define $\eta_i:=\xi_i\vee \tP^+$ and $\eta_{i-1}  := \xi_{i-1}\vee \tP^+$. Then the following lemmas can be obtained as in \cite[Section 3]{Part1}.

\begin{lemma}
The following properties hold:
\begin{itemize}
    \item[(1)] $\eta_i$ and $\eta_{i-1}$ are increasing, which means that $f\eta_i\prec\eta_i$ and $f\eta_{i-1}\prec\eta_{i-1}$.
    \item[(2)] $\eta_i\prec\eta_{i-1}$.
    \item[(3)] $\eta_{i-1}(x)\subseteq\exp_x(W^{i-1}_{x,2\delta}(x))$, $\eta_i(x)\subseteq\exp_x(W^i_{x,2\delta}(x))$ for $m$-a.e$.\ x$.
    \item[(4)] $h_m(f,\eta_i)=h_i$ and $h_m(f,\eta_{i-1})=h_{i-1}$.
\end{itemize}
\begin{proof}
    The same as \cite[Lemma 3.3]{Part1}, parallel to \cite[Lemma 11.1.1]{LEDRAPPIER_YOUNG_B}.
\end{proof}
\end{lemma}

\begin{lemma}
    For $m$-a.e$.\ x$ and for any $y\in\Gamma'\cap\eta_i(x)$, $$\exp_x(W^{i-1}_{x,2\delta}(y))\cap \eta_i(x)=\eta_{i-1}(y).$$
\begin{proof}
    The same as \cite[Lemma 3.4]{Part1}, parallel to \cite[Lemma 11.1.2]{LEDRAPPIER_YOUNG_B}.
\end{proof}
\end{lemma}

\begin{lemma}
For $m$-a.e$.\ x$ and for any $ y\in\Gamma'\cap \eta_i(x)$, $$f^{-1}(\eta_{i-1}(y))=\eta_{i-1}(f^{-1}y)\cap f^{-1}(\eta_i(x)).$$
\begin{proof}
    The same as \cite[Lemma 3.5]{Part1}, parallel to \cite[Lemma 11.1.3]{LEDRAPPIER_YOUNG_B}.
\end{proof}
\end{lemma}

So, there is a nice quotient structure for $\eta_{i}(x)/\eta_{i-1}$ as in  \cite{Part1} or \cite[Section 11]{LEDRAPPIER_YOUNG_B}. We now define the transverse metric on $\eta_i(x)/\eta_{i-1}$ for $m$-a.e. $x$. The first step is to define the projection map $\tpi:\bigcup_{n\geq 0}f^nE\rightarrow R^{(i)}$, where $R^{(i)}$ is a $\dim E^{(i)}$-dimensional Euclidean space. When $x\in E$, $\tpi(x)  :=  \pi(x)$, where $\pi$ is the straightening map defined in \Cref{straightening map section}. When $x\notin E$, we define $\tpi(x):=\pi(f^{\ell_0}(x))$, where $\ell_0$ is the largest negative number   such that   $f^{\ell_0}x\in E$. From the discussion in \Cref{straightening map section} and the previous discussion, for $m$-a.e. $x\in M$, $\tpi|_{\eta_i(x)}$ maps $\eta_{i-1}(y)$ into distinct $R^{(i-1)}$ planes and $\tpi|_{\eta_i(x)}$ is $\alpha$-bi-Hölder.

For $x\in \bigcup _{n\geq0}f^nE$ and $y,y'\in\eta_i(x)$, the transverse metric is defined by $$\tilde d^i(y,y')  :=  |\tpi^iy-\tpi^iy'|,$$ where $\tpi^i$ is the $i$-th component of $\tpi$. Let $\{\tm^i_x\}$ be the conditional measures of $m$ associated with $\eta_i$. Let $\tilde B^i(x,\rho)  := \{y\in\eta_i(x):\tilde d^i(x,y)\leq\rho\} $. We define
\begin{equation}\label{eq 59}
    \tilde\gamma_i(x)  :=  \liminf_{\rho\rightarrow 0}\frac{\log\tm^i_x(\tilde B^i(x,\rho))}{\log\rho}
\end{equation}
as the transverse dimension function for $\eta_i/\eta_{i-1}$. Recall our assumption that $\dim E^i=1$. This implies that the quotient space ${R^{(i)}}/{R^{(i-1)}}$ is isomorphic to $\mathbb R$. Therefore, since the quotient measure $\tm^i_x\circ (\tpi^i)^{-1}$ is supported on this one-dimensional space, an application of  \cite[Lemma 4.4]{Part1} yields the bound $\tgamma_i(x)\leq 1$.

By the same arguments as in \cite[Section 5]{Part1}, we can prove that the following proposition analogous to \cite[Equation (13)]{Part1} holds.

\begin{lemma}\label{lem 33}
    For any $\epsilon>0$, there exist measurable partitions $\eta_i,\ \eta_{i-1}$ 
    as above   such that   for $m$-a.e. $x\in M$, 
    \begin{equation*}
    \begin{aligned}
        \frac{\lambda_i+3\epsilon}{\alpha}\cdot\tgamma_i(x)&\geq -(1-\epsilon)\cdot(h(f,\eta_{i-1})+2\epsilon)+(1-\epsilon)\cdot(h(f,\eta_i)-\epsilon)\\
        &=(1-\epsilon)\cdot(h_i-h_{i-1}-3\epsilon)
    \end{aligned}
    \end{equation*}
    where $\tgamma_i(x)$ is the transverse dimension function for $\eta_i/\eta_{i-1}$ defined by \Cref{eq 59}.
\end{lemma}

In particular, one has
\begin{equation*}
    \frac{\lambda_i+3\epsilon}{\alpha}\geq (1-\epsilon)\cdot(h_i-h_{i-1}-3\epsilon).
\end{equation*}
As $\epsilon\rightarrow 0$, it follows that $\alpha=1-\frac{7\epsilon}{\lambda_{i-1}-\lambda_{i}}\rightarrow 1$, which proves that 
\begin{equation}\label{eq 55}
    \lambda_i\geq h_i-h_{i-1}.
\end{equation}

\subsubsection{Substitute $\odelta_i$ and $\odelta_{i-1}$ for $\tgamma_i(x)$}
To finish the proof, we still need to substitute $\odelta_i$ and $\odelta_{i-1}$ for $\tgamma_i(x)$. For this purpose, another transverse dimension function $\gamma_i(x)$ is introduced as a transition. Consider $\epsilon>0$, $A_0$, $S$ and $\pi$ given as before, and $E$, $\eta_i$, $\eta_{i-1}$ as above. $\tgamma_i$ is the function of the transverse dimension of $\eta_i/\eta_{i-1}$ defined by \Cref{eq 59}. $\hat\xi_i$ is defined in \Cref{equation 19} and let $\{\hat m^i_x\}$ be the conditional measures of $m$ associated with $\hat\xi_i$.

Fix $x\in S$, let $\hat\pi  :=  \pi|_{\hat\xi_i(x)}:\hat\xi_i(x)\rightarrow R^{(i)}$ be the straightening map on $\hat\xi_i(x)$. Since $m$ is ergodic, for $m$-a.e. $x\in S$, $\hat m^i_x(\bigcup_{n\geq 0}f^nE)=1$. We define the transverse dimension function $\gamma_i:R^{(i)}\rightarrow \mathbb R$ as 
\begin{equation*}
    \gamma_i(y)=\liminf_{\rho\rightarrow0}\frac{\log \hat m^i_x\circ\hat\pi^{-1}\{z\in R^{(i)}:|z_i-y_i|\leq\rho\}}{\log\rho}.
\end{equation*}

We remark that the maps $\hat\pi$ and $\tpi$ are distinct. The map $\hat\pi$ is a straightening map defined for points in $S$, constructed leafwise from the straightening maps on each element $\hat\xi_i(x)$, while the definition of $\tpi$ depends on the location of $x$. For $x \in E$, it coincides with $\hat\pi$. For $x \notin E$, its definition requires first iterating $x$ backward via $f^{-1}$ finite many times until its first return to $E$, and then composing this iterate with the appropriate straightening map  $\hat\pi|_E$. 
Since $f^{-1}$ is $C^1$ and the straightening maps are $\alpha$-bi-Hölder, it follows that both $\hat\pi$ and $\tpi$ are $\alpha$-bi-Hölder continuous.

Then, based on the arguments in the first half of \cite[Section 11.4]{LEDRAPPIER_YOUNG_B}, with some adjustments, we can prove the following equation 
\begin{equation}\label{eq 57}
    \gamma_i(y)\geq\frac{\alpha^3(1-\epsilon)}{\lambda_i+3\epsilon}\cdot(h_i-h_{i-1}-3\epsilon),\ \hat m^i_x\circ\hat\pi^{-1}-a.e.\ y\in R^{(i)}.
\end{equation}
Let us show the proof. Recall that given any $\eta_i(p)\subseteq\hat\xi_i(x)$, the definition for $\tgamma_i(y)$ is
\begin{equation*}
    \tgamma_i(y)=\liminf_{\rho\rightarrow0}\frac{\log \tilde m^i_p\circ\tpi^{-1}\{z\in R^{(i)}:|z_i-y_i|\leq\rho\}}{\log\rho}.
\end{equation*}
By \Cref{lem 33}, for $\tilde m^i_p\circ\tpi^{-1}-a.e.\ y\in R^{(i)}$,
\begin{equation*}
    \liminf_{\rho\rightarrow0}\frac{\log \tilde m^i_p\circ\tpi^{-1}\{z\in R^{(i)}:|z_i-y_i|\leq\rho\}}{\log\rho}\geq\frac{\alpha(1-\epsilon)}{\lambda_i+3\epsilon}\cdot(h_i-h_{i-1}-3\epsilon).
\end{equation*}
Since both $\tpi$ and $\hat{\pi}$ are $\alpha$-bi-Hölder continuous when restricted to any fixed atom of partition $\eta_i$, it follows that
for $\tilde m^i_p\circ\hat\pi^{-1}-a.e.\ y\in R^{(i)}$,
\begin{equation*}
    \liminf_{\rho\rightarrow0}\frac{\log \tilde m^i_p\circ\hat\pi^{-1}\{z\in R^{(i)}:|z_i-y_i|\leq\rho\}}{\log\rho}\geq\frac{\alpha^3(1-\epsilon)}{\lambda_i+3\epsilon}\cdot(h_i-h_{i-1}-3\epsilon).
\end{equation*}
Indeed, on each atom of \(\eta_i\), the transition map
\(\hat\pi\circ\tpi^{-1}\) between the \(\tpi\)-coordinates and the
\(\hat\pi\)-coordinates is \(\alpha^2\)-Hölder, since both \(\tpi^{-1}\)
and \(\hat\pi\) are \(\alpha\)-Hölder. Therefore, a lower pointwise
dimension lower bound is multiplied by \(\alpha^2\). Combined with the
factor \(\alpha\) already present in \Cref{lem 33}, this gives the
factor \(\alpha^3\).
Then, applying \Cref{lower dimension lemma} with $\Omega=\hat\xi_i(x)/\eta_i$ and $q=1$ as in \cite{LEDRAPPIER_YOUNG_B}, we can finish the proof of \Cref{eq 57}.

Consider now the partition of $\hat\pi(\hat\xi_i(x))\subseteq R^{(i)}$ into planes of $\{z_i=\text{constant}\}$. Since $\hat\pi$ is $\alpha$-bi-Hölder, we may assume that for $\hat m^i_x\circ\hat\pi^{-1}-a.e.\ y\in R^{(i)}$, 
\begin{equation*}
    \limsup_{\rho\rightarrow0}\frac{\log \hat m^i_x\circ\hat\pi^{-1}\{z\in R^{(i)}:|z-y|<\rho\}}{\log\rho}\leq \alpha^{-1}\cdot \odelta_i
\end{equation*}
and
\begin{equation*}
    \limsup_{\rho\rightarrow0}\frac{\log \hat m^{i-1}_{\hat\pi^{-1}y}\circ\hat\pi^{-1}\{z\in R^{(i)}:z_i=y_i,|z-y|<\rho\}}{\log\rho}\geq \alpha\cdot \odelta_{i-1}.
\end{equation*}
Hence using \Cref{upper transverse dimension lemma}, we obtain that
\begin{equation*}
    \alpha^{-1}\odelta_i-\gamma_i(y)\geq \alpha \odelta_{i-1}.
\end{equation*}
Together with \Cref{eq 57}, it follows that
\begin{equation*}
    \alpha^{-1}\odelta_i-\alpha\odelta_{i-1}\geq\gamma_i(y)\geq \frac{\alpha^3(1-\epsilon)}{\lambda_i+3\epsilon}\cdot(h_i-h_{i-1}-3\epsilon).
\end{equation*}
Once again letting $\epsilon\rightarrow0$ and $\alpha\rightarrow1$, we have
\begin{equation*}
    h_i-h_{i-1}\leq \lambda_i\cdot(\odelta_i-\odelta_{i-1}).
\end{equation*}

The proof of $ h_i-h_{i-1}\leq \lambda_i\cdot(\udelta_i-\udelta_{i-1})$ is nearly the same, as it only requires substituting the lemmas for the lower dimension (\Cref{lower transverse dimension lemma} and \Cref{lower dimension lemma}) for those lemmas for the upper dimension (\Cref{upper transverse dimension lemma} and \Cref{upper dimension lemma}). See also \cite[Section 11]{LEDRAPPIER_YOUNG_B} for the ideas of the proof. We leave the proof to the reader. Thus, the proof of the upper bound of transverse entropy
\begin{equation*}
    \lambda_i\cdot\min\{\udelta_i-\udelta_{i-1},\odelta_i-\odelta_{i-1}\}\geq h_i-h_{i-1}
\end{equation*}
is complete.

\subsection{Entropy formulas}\label{sec 5.3}
In this section, we complete the proof of three entropy formulas mentioned in \Cref{section 7.2}.

Recall that in the last section, we have proved that
\begin{equation*}
     \lambda_i\min\{\udelta_i-\udelta_{i-1},\odelta_i-\odelta_{i-1}\}\geq h_i-h_{i-1}\geq \lambda_i\max\{\udelta_i-\udelta_{i-1},\odelta_i-\odelta_{i-1}\}.
\end{equation*}
It follows that $\udelta_i-\udelta_{i-1}=\odelta_i-\odelta_{i-1}  :=  \gamma_i$ and $h_i-h_{i-1}=\lambda_i\gamma_i$. Together with \Cref{eq 55} we have $\gamma_i\leq 1$. Hence, the proof of the transverse entropy formula \Cref{transverse entropy formula} is complete.

Subsequently, we further assume that the assumption for the partial entropy formula \Cref{partial entropy formula} holds, i.e. $f\in \Diff(M)$ preserves an ergodic Borel probability $m$ and the splitting $$E^{1}\oplus E^{2}\oplus\cdots\oplus E^{i}\oplus E^{\widehat{(i)}}$$ is dominated on the support of $m$ with $\dim E^{2}=\cdots=\dim E^{i}=1$.

To prove the partial entropy formula, the key is to prove that $\udelta_1=\odelta_1=\delta_1\leq\dim E^{u,1}$ and $h_1=\lambda_1\delta_1$ following the ideas of \cite[Section 10.1]{LEDRAPPIER_YOUNG_B}. We first prove that $h_1\leq \lambda_1\udelta_1$.
Fix an $\epsilon>0$ small enough. Choose a measurable partition $\xi_1$ subordinate to $W^1$ and let $\{m^1_x\}$ be the conditional measures of $m$ associated with $\xi_1$. By \Cref{lemma 15}, it follows that for $m$-a.e. $x\in M$, there exists $\delta>0$ small enough such that   
\begin{equation*}
    \uh_1(x,\delta,\xi_1)\geq h_1-\epsilon.
\end{equation*}
Consider $y\in B^1(x,\frac{1}{100C^N_f}\cdot\e^{-n(\lambda_1+3\epsilon)}\delta A(x)^{-1})$. Let $z  :=  \exp_x^{-1}y$, then $z\in W^1_x(x)$ and $$|z|\leq \frac{1}{50C^N_f}\cdot\e^{-n(\lambda_1+3\epsilon)}\delta A(x)^{-1}.$$  Together with the local argument (see \Cref{local arg here}), we can prove that  
\begin{equation*}
\begin{aligned}
    |\tf_x^k(z)|\leq |z|\cdot C^N_f\mathrm{e}^{(\lambda_1+3\epsilon)k}A(x)
    \leq \frac{\delta}{50},\quad \forall 0\leq k\leq n,
\end{aligned}
\end{equation*}
and $\tf^k_xz\in W^1_{f^k(x)}(f^k(x))$. Following this, one has $$B^1(x,\frac{1}{100C^N_f}\cdot\e^{-n(\lambda_1+3\epsilon)}\delta A(x)^{-1})\subseteq V^1(x,n,\delta).$$ This implies that for $m$-a.e. $x\in M$,
\begin{equation*}
\begin{aligned}
    \udelta_1(x)&=\liminf_{\rho\rightarrow0}\frac{\log m^1_x(B^1(x,\rho))}{\log\rho}\\
    &=\liminf_{n\rightarrow\infty}\frac{\log m^1_x(B^1(x,\frac{1}{100C^N_f}\cdot\e^{-n(\lambda_1+3\epsilon)}\delta A(x)^{-1}))}{-n(\lambda_1+3\epsilon)}\\
    &\geq\frac{1}{\lambda_1+3\epsilon}\cdot\liminf_{n\rightarrow\infty}\frac{1}{-n}\log m^1_x(V^1(x,n,\delta))\\
    &= \frac{1}{\lambda_1+3\epsilon}\cdot \uh_1(x,\delta,\xi_1)\\
    &\geq\frac{h_1-\epsilon}{\lambda_1+3\epsilon}.
\end{aligned}
\end{equation*}
Since $\epsilon$ can be arbitrarily small, as $\epsilon\rightarrow0$, it follows that $h_1\leq \lambda_1\udelta_1$.

We continue to prove $h_1\geq \lambda_1\odelta_1$. Let $\tR$ and $n_3$ be the measurable partition and the function in \Cref{claim 19} with $i=1$. Then we have
\begin{equation*}
\begin{aligned}
    \odelta_1(x)&=\limsup_{\rho\rightarrow0}\frac{\log m^1_x(B^1(x,\rho))}{\log\rho} \\
    &=\limsup_{n\rightarrow\infty}\frac{\log m^1_x(B^1(x,\e^{-n(\lambda_1-3\epsilon)}))}{-n(\lambda_1-3\epsilon)}\\
    &\leq\frac{1}{\lambda_1-3\epsilon}\limsup_{n\rightarrow\infty}-\frac{1}{n}\log m^1_x(\tR^n_0(x))\\
    &\leq\frac{h_1}{\lambda_1-3\epsilon},
\end{aligned}
\end{equation*}
where the last inequality follows from the definition of $\tR$ and \Cref{eq 33}. 
Letting $\epsilon\rightarrow0$, we obtain $h_1\geq \lambda_1\odelta_1$. Thus, the proof for $\udelta_1=\odelta_1  := \delta_1$ and $h_1=\lambda_1\delta_1$ is complete. $\delta_1\leq\dim E^{u,1}$ follows from \cite[Lemma 4.4]{Part1}.

To complete the proof, it only remains to establish the following claim by induction on $j$ for $2 \le j \le i$:
\begin{itemize}
    \item[(a)] The dimension $\delta_j$ is exact (i.e., $\odelta_j = \udelta_j$).
    \item[(b)] The bound $\delta_j \leq \dim E^1 + j-1$ holds.
    \item[(c)] The entropy formula $h_j-h_{j-1}=\lambda_j(\delta_j-\delta_{j-1})$ is satisfied.
\end{itemize}

\noindent\textbf{Inductive Step:} Assume the claim holds for an index $j-1$. By the transverse entropy formula \Cref{transverse entropy formula}, the $j$-th transverse dimension is exact and bounded by 1, yielding
\[ \udelta_j - \delta_{j-1} = \odelta_j - \delta_{j-1} \le 1, \]
and the entropy relation is given by $h_j-h_{j-1}=(\udelta_j-\delta_{j-1})\cdot\lambda_j$.
Since $\delta_{j-1}$ is exact by the inductive hypothesis, the exactness of the transverse dimension immediately implies that $\delta_j$ is also exact. We may therefore set $\odelta_j=\udelta_j:=\delta_j$. This proves (a). 
The bound (b) follows from the inductive hypothesis: $\delta_j \leq \delta_{j-1} + 1 \leq (\dim E^1 + j-2) + 1 = \dim E^1 + j-1$.
Finally, substituting $\delta_j$ and $\delta_{j-1}$ into the entropy relation proves (c). This completes the induction.

Hence, the proof of the partial entropy formula \Cref{partial entropy formula} is complete.

Finally, we assume that the assumption for the metric entropy formula \Cref{metric entropy formula} holds, i.e. the splitting $E^{u,1}\oplus E^{u,2}\oplus\cdots\oplus E^{u,u}\oplus E^{c}\oplus E^s$ is dominated on the support of $m$ and $\dim E^{u,2}=\cdots=\dim E^{u,u}=1$, $\dim E^c\leq 1$. Then to obtain the metric entropy formula \Cref{metric entropy formula}, we only need to combine the partial entropy formula \Cref{partial entropy formula} and the theorem \cite[Theorem 1.1]{Part1}.

\section{Proof of \Cref{C1 dimension estimate 1}}\label{Appendix}

In this section, we prove \Cref{C1 dimension estimate 1} following the ideas in \cite[Section 12]{LEDRAPPIER_YOUNG_B}. The main difference is the proof of \Cref{lemma 35}. We use fake foliations and the local argument (see \Cref{local arg here}) instead of Lyapunov charts.

In \Cref{sec 6.1}, we construct suitable partitions and introduce their general properties, following the approach in \cite[Section 12.1]{LEDRAPPIER_YOUNG_B}.
In \Cref{sec 6.2}, \Cref{lemma 35} is proved, which concerns the behavior of the center direction.
Finally, in \Cref{sec 6.3}, we complete the proof using an argument analogous to that in \cite[Section 12.3]{LEDRAPPIER_YOUNG_B}.

\subsection{General properties}\label{sec 6.1}
Throughout this section, let $m$ be an ergodic Borel probability measure preserved by $f\in\Diff(M)$, and assume that $m$ is simply dominated throughout this section.
In the case where $m$ is hyperbolic, we refer the reader to the proof of \Cref{C^1 dimension estimate 2} in \Cref{section erconjecture}, although the proof presented in \Cref{Appendix} can be adapted to cover it with several adjustments.
Therefore, in \Cref{Appendix} we will focus on the case where $\dim E^c=1$.

Let $\lambda^u$ be the smallest positive exponent, $\lambda^s$ the largest negative exponent. Fix any $\epsilon>0$ small enough. Let $A(x)$ be the function in \cite[Lemma 2.2]{Part1} (which is parallel to \Cref{ABC ergodic lemma} with $u,c,s$).
Let $\xi^u$ be a measurable partition subordinate to $W^u$ for $f$, and let $\xi^s$ be a measurable partition subordinate to $W^s$ for $f^{-1}$. 
Let $\{m^u_x\}$ and $\{m^s_x\}$ be the conditional measures of $m$ associated with $\xi^u$ and $\xi^s$ respectively. $d^u$ and $d^s$ are defined as in \Cref{section 7.3}. We denote $h:=h_m(f)$.

First, by \cite[Main Theorem]{Part1}, the unstable entropy is equal to the metric entropy when $m$ is simply dominated, that is,
\begin{equation}\label{a1}
    h_u:=H(\xi^u|f\xi^u)=h,
\end{equation}
and
\[
h_s:=H(\xi^s|f^{-1}\xi^s)=h.
\]

Recall that in \Cref{section 7.3}, for $m$-a.e. $x\in M$,

\begin{equation}\label{a2}
    \lim_{\rho\rightarrow0}\frac{\log m^u_x(B^u(x,\rho))}{\log\rho}= d^u,
\end{equation}
\begin{equation}\label{a3}
       \lim_{\rho\rightarrow0}\frac{\log m^s_x(B^s(x,\rho))}{\log\rho}= d^s.
\end{equation}
Moreover, for any finite-entropy partition $\tP$ and $a,b>0$, by the Shannon-McMillan-Breiman theorem we have for
$m$-a.e. $x\in M$,
\begin{equation}\label{a4}
    \lim_{n\rightarrow\infty}-\frac{1}{n}\log m(\tP^{na}_{-nb}(x))\leq (a+b)h
\end{equation}
where $\tP^q_p(x)  := (\bigvee_{i=p}^{q-1}f^{-i}\tP)(x)$.

Let $\eta=\xi^u\vee \tP^0_{-\infty}$ and $\{m_x\}$ be the conditional measures of $m$ associated with $\eta$. We introduce a general fact in measure theory, see \cite[Section 12.4]{LEDRAPPIER_YOUNG_B}.
\begin{lemma}\cite[Lemma 12.4.1]{LEDRAPPIER_YOUNG_B}
For $m$-a.e. $x\in M$,
    \begin{equation}\label{a5}
        \lim_{n\rightarrow\infty} -\frac{1}{n}\log m_x(\tP^n_0(x))=h_m(f,\tP)
    \end{equation}
\end{lemma}
Using \Cref{a2} and \Cref{upper dimension lemma}, we can prove that for $m$-a.e. $x\in M$,
\begin{equation}\label{a6}
    \limsup_{\rho\rightarrow 0}\frac{\log m_x(B^u(x,\rho))}{\log\rho}\leq d^u.
\end{equation}
By the previous discussion in \Cref{section partial entropy}, more specifically, by the first inequality of \Cref{eq 33} applying to $f^{-1}$ and $\liminf$, it follows that for any small $\epsilon'>0$, there exists a finite-entropy partition $\tP_1$ such that for $m$-a.e. $x\in M$,
\[
\liminf_{n\rightarrow\infty}-\frac{1}{n}\log m^s_x\left((\tP_1)^0_{-n}(x)\right)\geq \uh_s(x,\epsilon',\xi^s).
\]
One can choose $\epsilon'$ small enough such that 
\[
\uh_s(x,\epsilon',\xi^s)\geq h_s-\epsilon=h-\epsilon.
\]
Thus, for $m$-a.e. $x\in M$, one has 
\begin{equation}\label{a7}
    \liminf_{n\rightarrow\infty}-\frac{1}{n}\log m^s_x\left((\tP_1)^0_{-n}(x)\right)\geq h-\epsilon.
\end{equation}
We also assume that $h_m(f,\tP_1)\geq h-\epsilon$. Together with \Cref{a5}, one has
\begin{equation}\label{a8}
    \lim_{n\rightarrow\infty}-\frac{1}{n}\log m_x((\tP_1)^n_0(x))\geq h-\epsilon.
\end{equation}

\subsection{Cutting the center direction}\label{sec 6.2}
We establish the following lemma, which is analogous to \cite[Lemma 12.2.1]{LEDRAPPIER_YOUNG_B}. 
\begin{lemma}\label{lemma 35}
    Let $a=\frac{1}{\lambda^u-3\epsilon}$, $b=\frac{1}{-\lambda^{s}-3\epsilon}$. For any $\epsilon'>0$, there exists a set $\Lambda_1$ with $m(\Lambda_1)>1-\epsilon'$, an integer $N_0\in\mathbb N$, a constant $C>0$, and a finite-entropy partition $\tP$ that refines $\tP_1$ with the following properties: for any $n\geq N_0$, there exists a partition $\Xi_n\succ \tP^{na}_{-nb}$   such that for any $x\in\Lambda_1$, 
    \begin{itemize}
    \item[(1)] $\mathrm{diam}(\Xi_n(x))\leq 2\e^{-n}$ and
    \item[(2)] $m(\Xi_n(x))\geq C^{-1}\e^{-n(1+\epsilon)}m(\tP^{na}_{-nb}(x))$.
    \end{itemize}
\begin{proof}
    Take $A>0$ large enough so that for $$\Lambda_2:=\{x\in \Gamma':A(x)\leq A\},$$ one has   $m(\Lambda_2)>1-0.1\epsilon'$. Let $$\delta< \frac{r_0}{10^6C_0^{10}C_f^N} $$ be small enough that for any $d(x,x')\leq\delta$ and $*,**\in\{c,u,s,cu,cs,su\}$ satisfying $T_xM=E^*_x\oplus E^{**}_x$, 
    $$\exp_{x'}^{-1}\circ\exp_x\{v_x^*\in T_xM:v_x^*=\mathrm{constant}\}$$ is the graph of a function $\varphi:E^{**}\rightarrow E^{*}$ with slope $\leq\frac{1}{10^6}$.
    
    Moreover, the decomposition of $T_xM=E^u_x\oplus E^c_x\oplus E^s_x$ for any $v\in T_xM$ is denoted as 
    $v=v_u+v_c+v_s$ or $v=v^u+v^c+v^ s$.
    Recall that $|\cdot|$ denotes the Euclidean norm on the tangent space induced by the Riemannian metric and $|\cdot|'$ is the box norm defined in \Cref{section fake foliation}.  
    Let $\Xi_0$ be a finite partition of $\Lambda_2$ into sets of diameter smaller than $\delta$ and choose a point $z(q)\in q$ for any $q\in\Xi_0$. Denote $\bar x  :=  \exp^{-1}_{z(\Xi_0(x))}x$. For any $y\in\Xi_0(x)$ and $\bar y  :=  \exp^{-1}_{z(\Xi_0(x))}y$, since $\delta$ is small enough, it follows that $$\frac{1}{2}d(x,y)\leq|\bar x-\bar y|\leq 2d(x,y).$$

    \begin{claim}\label{claim 6.3}
        For any $x\in\Lambda_2$, define $$V(x,n)=\{y\in\Xi_0(x):d(f^i(x),f^i(y))\leq \frac{\delta}{2}A(f^i(x))^{-1},\forall -nb\leq i\leq na\}.$$
        For any $y_1,y_2\in V(x,n)$ such that $|\bar y^c_1-\bar y^c_2|\leq \frac{\e^{-n}}{100C_0}$, $|\bar y_1-\bar y_2|\leq\e^{-n}$.
    \begin{proof}
        If $$|\bar y^c_1-\bar y^c_2|\geq\frac{1}{100}\max \{|\bar y_1^u-\bar y^u_2|,|\bar y_1^s-\bar y^s_2|\},$$ then the claim is clearly true since $$C_0^{-1}|\cdot|\leq |\cdot|'\leq C_0|\cdot|,$$ where $|\cdot|'$ denotes the box norm and $|\cdot|$ denotes the Euclidean norm.    
        Now assume $$|\bar y^c_1-\bar y^c_2|\leq\frac{1}{100}\max \{|\bar y_1^u-\bar y^u_2|,|\bar y_1^s-\bar y^s_2|\}.$$
        Since $\delta>0$ is chosen small enough, it follows that 
        \begin{equation}\label{rewirte1}
        |(\exp_x^{-1} y_1)_c-(\exp_x^{-1} y_2)_c|\leq\max \{|(\exp_x^{-1} y_1)_u-(\exp_x^{-1} y_2)_u|,|(\exp_x^{-1} y_1)_s-(\exp_x^{-1} y_2)_s|\}.
        \end{equation}
        Without loss of generality, we assume that 
        \begin{equation}\label{rewirte2}
            |(\exp_x^{-1} y_1)_u-(\exp_x^{-1} y_2)_u|\geq |(\exp_x^{-1} y_1)_s-(\exp_x^{-1} y_2)_s|.
        \end{equation}

        
        Consider two points $y_1$ and $y_2$ such that $\exp_x^{-1}y_1$ and $\exp_x^{-1}y_2$ lie inside the local chart at $x$ (see \Cref{fake foliation charts def}). By definition, the first $[ na ]$ steps of their forward orbits remain within the local charts at the corresponding points on the forward orbit of $x$. Consequently, the maps $\tilde{f}$ and $\tilde{g}$ defined in \Cref{section fake foliation} coincide along these orbit segments.

For $0\leq i\leq [ na ]$, denote $u_0:=\tW_x^{cs}(y_1)\cap\tW^u_x(y_2)$ and $u_i := \tilde{g}_x^i(u_0)$. By the invariance of the fake foliations in \Cref{fake $u$-foliation} b) (similar discussions are presented in \cite{LVY13}),
\[
    u_i=\tW^{cs}_{f^i(x)}(f^i(y_1))\cap\tW^u_{f^i(x)}(f^i(y_2)).
\]
By the graph properties of the fake foliations in \Cref{fake $u$-foliation} a), it follows that $u_i$ stays inside the local chart at $f^i(x)$.
        

Next, we derive estimates involving the unstable distance.

Recall that the constant $C_0$, introduced in \Cref{section fake foliation}, is chosen to be sufficiently large to dominate any geometric distortion factors that arise from the uniform angles away from zero between the center-stable and unstable directions. 
By the graph property of the fake foliations, after increasing \(C_0\) if
necessary, the angles between the fake center-stable leaves and the fake
unstable leaves are uniformly controlled by \(C_0\). The reader can keep in mind that the fake leaves are flat enough.

After enlarging $C_0$ if necessary for the steps that follow, we obtain the following estimates, where $d^u$ denotes the distance measured along the fake unstable leaves.

        \begin{align*}
        d^u(\exp_{x}^{-1}y_2,u_0)&\geq |(\exp_x^{-1}y_2)_u-(u_0)_u|\\ 
        &\geq |(\exp_x^{-1}y_2)_u-(\exp_x^{-1}y_1)_u|-|(\exp_x^{-1}y_1)_u-(u_0)_u|\\
        &\quad\text{(By the triangle inequality)}\\
        &\geq |(\exp_x^{-1}y_2)_u-(\exp_x^{-1}y_1)_u|-\frac{1}{100C_0^2}|(\exp_x^{-1}y_1)_{cs}-(u_0)_{cs}|\\
        &\quad\text{(Because $u_0$ and $\exp_x^{-1}y_1$ are in the same fake center-stable leaf)}\\
        &\geq |(\exp_x^{-1}y_2)_u-(\exp_x^{-1}y_1)_u|-\frac{1}{100C_0^2}|\exp_x^{-1}y_1-u_0|\\
        &\geq |(\exp_x^{-1}y_2)_u-(\exp_x^{-1}y_1)_u|-\frac{1}{10C_0}|\exp_x^{-1}y_1-\exp_x^{-1}y_2|\\
        &\quad\text{(Since $\tW_x^{cs}(y_1)$ and $\tW^u_x(y_2)$ have uniform angle,} \\
        &\quad\quad\text{and $C_0$ is chosen large enough depending on this uniform angle)}\\
        &\geq |(\exp_x^{-1}y_2)_u-(\exp_x^{-1}y_1)_u|-\frac{1}{10}|\exp_x^{-1}y_1-\exp_x^{-1}y_2|'\\
        &\quad\text{(Since we have $C_0^{-1}|\cdot|\leq |\cdot|'\leq C_0|\cdot|$ in previous discussion)}\\
        &\geq |(\exp_x^{-1}y_2)_u-(\exp_x^{-1}y_1)_u|-\frac{1}{10}|(\exp_x^{-1}y_1)_u-(\exp_x^{-1}y_2)_u|\\
        &\quad\text{(By \Cref{rewirte1} and \Cref{rewirte2})}\\
        &\geq \frac{1}{2} |(\exp_x^{-1}y_2)_u-(\exp_x^{-1}y_1)_u|.
        \end{align*}
        In conclusion, we have
        \begin{equation}\label{longeq}
            d^u(\exp_{x}^{-1}y_2,u_0)\geq \frac{1}{2} |(\exp_x^{-1}y_2)_u-(\exp_x^{-1}y_1)_u|.
        \end{equation}
        \begin{figure}[htbp]
    \centering
    \includegraphics[width=0.7\textwidth]{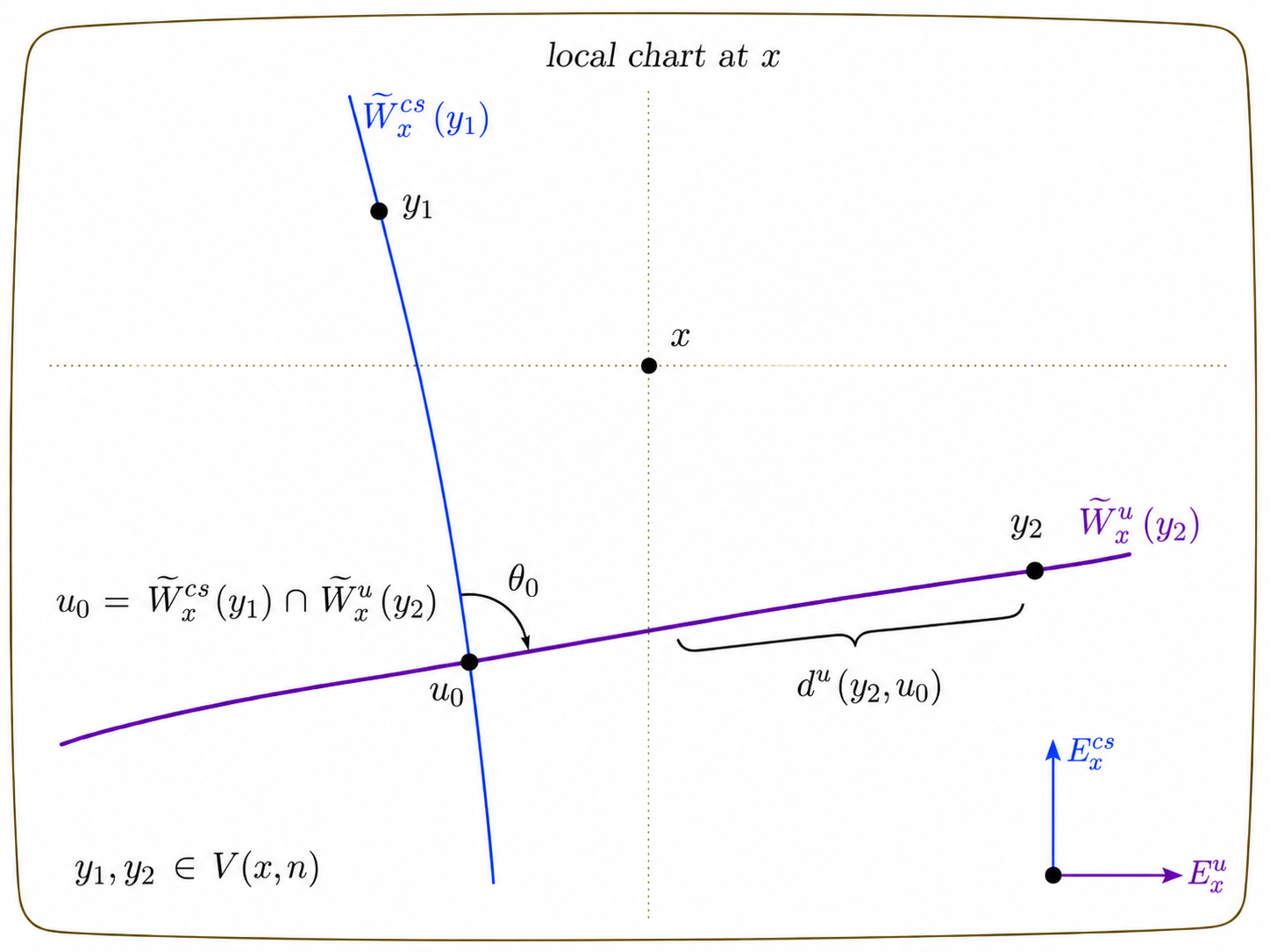}
    \caption{illustration for \Cref{claim 6.3}.}
\end{figure}
     
        Secondly, since it follows from $y_1,y_2\in V(x,n)$ that $$|\exp_{f^{[na]}(x)}^{-1}f^{[na]}(y_1)-\exp_{f^{[na]}(x)}^{-1}f^{[na]}(y_2)|\leq \delta A(f^{[na]}(x))^{-1},$$ then one can calculate that
        \begin{equation*}
        \begin{aligned}
        &\quad d^u(\exp_{f^{[na]}(x)}^{-1}f^{[na]}(y_2),u_{[na]})\\
        &\leq 2\cdot|(\exp_{f^{[na]}(x)}^{-1}f^{[na]}(y_2))_u-(u_{[na]})_u|\\
        &\quad\text{(Since $\exp_{f^{[na]}(x)}^{-1}f^{[na]}(y_2)$ and $u_{[na]}$ are in the same fake unstable leaf)}\\
        &\leq 20C_0\cdot|\exp_{f^{[na]}(x)}^{-1}f^{[na]}(y_1)-\exp_{f^{[na]}(x)}^{-1}f^{[na]}(y_2)|\\
        &\quad\text{(Since $\tW_{f^{[na]}x}^{cs}(f^{[na]}(y_1))$ and $\tW^u_{f^{[an]}x}(f^{[na]}(y_2))$ have uniform angle,} \\
        &\quad\quad\quad\text{and $C_0$ is chosen large enough depending on this uniform angle)}\\
        &\leq 20C_0\delta \cdot A(f^{[na]}(x))^{-1}.
        \end{aligned}
        \end{equation*}
        Moreover, recall that $u_i$ stays inside the local chart at $f^i(x)$, by the local argument (see \Cref{local arg here}), we can calculate that 
        \begin{equation*}
        \begin{aligned}
            d^u(\exp_{x}^{-1}y_2,u_0) &\leq A(f^{[na]}(x))\cdot C^N_f\e^{-[na](\lambda^u-3\epsilon)}\cdot d^u(\exp_{f^{[na]}(x)}^{-1}f^{[na]}(y_2),u_{[na]})\\
            &\leq 20C_0\cdot C^N_f\cdot\delta \e^{-[na](\lambda^u-3\epsilon)}.
        \end{aligned}
        \end{equation*}
        Together with \Cref{longeq}, it follows that 
        \begin{equation*}
            |(\exp_x^{-1}y_2)_u-(\exp_x^{-1}y_1)_u|\leq 100C_0\cdot C^N_f\cdot \e^{-n}\delta\leq \frac{1}{4}C_0^{-1}\e^{-n}.
        \end{equation*}
        Then it follows that $$|\exp_x^{-1}y_2-\exp_x^{-1}y_1|\leq \frac{1}{4}\e^{-n}$$ and thus $|\bar y_1-\bar y_2|\leq\e^{-n}$.
    \end{proof}
    \end{claim}
The proof of \Cref{lemma 35} requires a specific partition $\tP$ and a constant $N_0$, which we define below. Let $\tP_1$ be the partition defined previously. We then define the partition $\tP_2$ of $M$ by
\[ \tP_2 = \{ M - \Lambda_2 \} \cup \Xi_0, \]
where $\Xi_0$ is the measurable partition of $\Lambda_2$ at the beginning of this proof. Using \cite[Section 2.4]{Part1}, there exist a finite-entropy partition \(\tP_3\) and a measurable function \(n_0:M\to\mathbb N\) such that, if \(n\geq n_0(x)\), then
\[
    (\tP_3)^{na}_{-nb}(x)\subseteq V(x,n).
\]
We choose
\[
    \tP=\tP_1\vee\tP_2\vee\tP_3
\]
and \(N_0\in\mathbb N\) such that
\[
    m\{x:n_0(x)\leq N_0\}\geq 1-0.1\epsilon'.
\]

Then, for any \(n\geq N_0\), we define a measurable partition \(\Xi_n\) of \(M\) by refining the atoms of \(\tP^{na}_{-nb}\) that intersect
$
    \Lambda_2\cap\{x:n_0(x)\leq N_0\}.
$
More precisely, for each atom \(P\) of \(\tP^{na}_{-nb}\) satisfying
\[
    P\cap \Lambda_2\cap\{x:n_0(x)\leq N_0\}\neq\varnothing,
\]
we subdivide the whole atom \(P\) into finitely many measurable pieces according to the center coordinate in the chart determined by \(\Xi_0\), in such a way that whenever \(y_1,y_2\) belong to the same element of this subdivision, 
one has
\[
    |\bar y^c_1-\bar y^c_2|\leq \frac{\e^{-n}}{100C_0}.
\]
For atoms \(P\) of \(\tP^{na}_{-nb}\) with
\[
    P\cap \Lambda_2\cap\{x:n_0(x)\leq N_0\}=\varnothing,
\]
we leave \(P\) unchanged. Thus \(\Xi_n\) is a measurable partition of \(M\) and satisfies
\[
    \Xi_n\succ \tP^{na}_{-nb}.
\]
    Note that since the center is one-dimensional, there exists a constant $C_9>0$ such that the cardinality of $\Xi_n|_{\tP^{na}_{-nb}(x)}$ is bounded by $C_9\cdot\e^{n}$.
    The claim above implies that these elements have a diameter less than $2\e^{-n}$.

    Let 
    $$A_n  :=  \{x:m(\Xi_n(x)\cap \tP^{na}_{-nb}(x))\leq\frac{\e^{-n}}{C_9}0.1\epsilon'\cdot\e^{-n\epsilon}(1-\e^{-\epsilon})\cdot m(\tP^{na}_{-nb}(x))\}.$$
    Then it follows that $$m(A_n)\leq 0.1\epsilon'\cdot\e^{-n\epsilon}(1-\e^{-\epsilon}).$$ Let $$\Lambda_1:=(\Lambda_2\cap\{x:n_0(x)\leq N_0\})-\bigcup_{n\in \mathbb N}A_n,$$ 
    then we have $m(\Lambda_1)>1-\epsilon'$, and \Cref{lemma 35} is proved with $C=\frac{10C_9}{\epsilon' (1-\e^{-\epsilon})}$.
\end{proof}
\end{lemma}
\subsection{Counting arguments}\label{sec 6.3}
With \Cref{lemma 35} in hand, \Cref{C1 dimension estimate 1} can be proved by adapting the counting arguments directly from \cite[Section 12.3]{LEDRAPPIER_YOUNG_B}, and we provide the full details below for completeness.

Fix $a,b$, $\Lambda_1$, $N_0$, and $\tP$ as in \Cref{lemma 35}. By \Cref{a4}, \Cref{a7}, and \Cref{a8}, there exists a set $\Lambda_3\subseteq \Lambda_1$ with $m(\Lambda_3)>1-2\epsilon'$ and $N_1>N_0$   such that   for any $x\in\Lambda_3$ and $n\geq N_1$, 
\begin{equation}\label{ea}
    m^s_x(\tP^0_{-nb}(x))\leq \e^{-nb(h-2\epsilon)},
\end{equation}
\begin{equation}\label{eb}
    m_x(\tP^{na}_0(x))\leq \e^{-na(h-2\epsilon)},
\end{equation}
\begin{equation}\label{ec}
    m(\tP^{na}_{-nb}(x))\geq \e^{-n(a+b)(h+2\epsilon)},
\end{equation}
\begin{equation}\label{ed}
    \mathrm{diam}(\Xi_n(x))\leq 2\e^{-n},
\end{equation}
and 
\begin{equation}\label{ee}
    m(\Xi_n(x))\geq C^{-1}\e^{-n(1+\epsilon)}\e^{-n(a+b)(h+2\epsilon)}.
\end{equation}

By \Cref{a6} and the density point theorem, we can choose $\Lambda_4\subseteq\Lambda_3$ with $m(\Lambda_4)>1-3\epsilon'$ and $N_2>N_1$   such that for any $y\in\Lambda_4$ and $n>N_2$,
\begin{equation}\label{ef}
    m_y(\Lambda_3\cap B^u(y,\e^{-n}))\geq\frac{1}{2}\e^{-n(d^u+\epsilon)}.
\end{equation}
Again, by \Cref{a3} and the density point argument, we can choose $\Lambda\subseteq \Lambda_4$ with $m(\Lambda)>1-4\epsilon'$ and $N_3>N_2$   such that for any $x\in\Lambda$ and $n\geq N_3$,
\begin{equation}\label{eg}
    m^s_x(\Lambda_4\cap B^s(x,\e^{-n}))\geq\frac{1}{2}\e^{-n(d^s+\epsilon)}.
\end{equation}

Now, fix any $x\in\Lambda$ and denote $$\bar d:=\limsup_{n\to\infty}-\frac{1}{n}\log m(B(x,4\e^{-n})).$$ There exist infinitely many $n$   such that  
\begin{equation}\label{eh}
    m\left(B(x,4\e^{-n})\right)\leq\e^{-n(\bar d-\epsilon)}.
\end{equation}
For $n$ large enough such that $n>N_3$ and $4C\leq\e^{n\epsilon}$, we consider
$$N=\#\{\text{atoms of }\Xi_n\text{ intersecting }\Lambda_3\cap B(x,2\e^{-n})\}$$
From \Cref{ed}, \Cref{ee}, and \Cref{eh}, it follows that
\begin{equation}\label{upper bound for N}
    N\leq C\e^{n(1+\epsilon)}\e^{n(a+b)(h+2\epsilon)}\e^{-n(\bar d-\epsilon)}.
\end{equation}

Then we estimate the lower bound for $N$. First, we count the atoms of $\tP^{na}_{-nb}$ that intersect with $\Lambda_3\cap B(x,4\e^{-n})$. Since $x\in\Lambda$, for any $y\in\xi^s(x)\cap\Lambda_4\cap B(x,\e^{-n})$, by \Cref{ea} one has
\begin{equation*}
    m^s_x(\tP^0_{-nb}(y))=m^s_y(\tP^0_{-nb}(y))\leq\e^{-nb(h-2\epsilon)}.
\end{equation*}
Thus, by \Cref{eg},
\begin{equation*}
    \#\{\text{atoms of }\tP^0_{-nb}\text{ intersecting }\Lambda_4\cap B(x,\e^{-n})\}\geq\frac{1}{2}\e^{-n(d^s+\epsilon)}\e^{nb(h-2\epsilon)}.
\end{equation*}
Let us fix one of these atoms as $p_u$ and choose $y\in p_u\cap\Lambda_4\cap B(x,\e^{-n})$. For any $z\in\eta(y)\cap \Lambda_3\cap B(y,\e^{-n})$, by \Cref{eb} we have
\begin{equation*}
    m_y(\tP_0^{na}(z))=m_z(\tP^{na}_0(z))\leq\e^{-na(h-2\epsilon)}.
\end{equation*}
Thus, if $n(X):=\#\{\text{atoms of }\tP^{na}_0\text{ intersecting }X\cap \Lambda_3\cap B(y,\e^{-n})\}$, by \Cref{ef} we obtain
\begin{equation*}
    n(\eta(y))\geq\frac{1}{2}\e^{-n(d^u+\epsilon)}\e^{na(h-2\epsilon)}.
\end{equation*}
Moreover, since $$\eta(y)\subseteq \tP^0_{-\infty}(y)\subseteq p_u,$$ it follows that $n(p_u)\geq n(\eta(y))$. Let $p_s$ be one of the $\text{atoms of }\tP^{na}_0\text{ intersecting }p_u\cap \Lambda_3\cap B(y,\e^{-n})$, then $p_u\cap p_s$ is an atom of $\tP^{na}_{-nb}$ that intersects $\Lambda_3\cap B(y,\e^{-n})$ for some $y$ with $d(y,x)\leq\e^{-n}$.
Since $\Xi_n$
refines 
$\tP^{na}_{-nb}$
, 
each atom of $\tP^{na}_{-nb}$ intersecting the relevant set contains at least one atom of $\Xi_n$ intersecting the same set.
Hence,
\begin{equation}\label{lower bound for N}
\begin{aligned}
    N&\geq\sum_{\{p_u:p_u\cap\Lambda_4\cap B(x,\e^{-n})\neq\emptyset\}}n(p_u)\\
    &\geq \frac{1}{4}\e^{-n(d^u+d^s+2\epsilon)}\e^{n(a+b)(h-2\epsilon)}.
\end{aligned}
\end{equation}
Comparing \Cref{upper bound for N} and \Cref{lower bound for N} and letting $n\rightarrow\infty$, $\epsilon\rightarrow 0$, it follows that $$\bar d\leq d^u+d^s+1.$$ Thus, the proof of \Cref{C1 dimension estimate 1} is complete.

\begin{remark}
    If we only admit weaker assumption that $E^u\oplus E^c\oplus E^s$ is dominated on the support of $m$ and $\dim E^c\leq 1$ instead of simply dominated, then $d^u$ and $d^s$ are not known to be exact. However, if we define 
    \begin{equation*}
    \bar d^u:=\limsup_{\epsilon\rightarrow 0}\frac{\log m^u_x(B^u(x,\epsilon))}{\log\epsilon},\ m-a.e.\ x\in M,
\end{equation*}\begin{equation*}
    \bar d^s:=\limsup_{\epsilon\rightarrow 0}\frac{\log m^s_x(B^s(x,\epsilon))}{\log\epsilon},\ m-a.e.\ x\in M.
\end{equation*}
    Using the same arguments, we can obtain the following upper bound estimate
\begin{equation*}
    \bar d\leq\bar d^u+\bar d^s+1.
\end{equation*}
\end{remark}

\section{Eckmann-Ruelle conjecture and \Cref{C^1 dimension estimate 2}}\label{section erconjecture}

\subsection{Eckmann-Ruelle conjecture}
Let $M$ be a $C^{\infty}$ compact Riemannian manifold without boundary, and $f: M \to M$ a diffeomorphism. We consider an $f$-invariant, ergodic Borel measure $m$, which is called \textit{hyperbolic} if it has no zero Lyapunov exponents.

A central problem in dimension theory, often referred to as the Eckmann-Ruelle conjecture, is whether the dimension of such a measure $m$ is exact. Following several breakthroughs in the last century, notably by Young \cite{Young_1982} and Ledrappier and Young \cite{LEDRAPPIER_YOUNG_A, LEDRAPPIER_YOUNG_B}, this conjecture was resolved for $C^{1+\alpha}$ diffeomorphisms by Barreira, Pesin, and Schmeling in their remarkable work \cite{Barreira_Pesin_Schmeling}. Their proof established a new and non-trivial property for hyperbolic ergodic measures, called the asymptotically almost local product structure. This property is key to their proof.

The result was subsequently extended to random systems by Liu and Xie \cite{liu-xie06}, who employed the method of local entropy, and to endomorphisms by Shu \cite{Shu10}. A common feature of these works is the requirement of $C^{1+\alpha}$ regularity. More recently, working in the $C^1$ setting with a dominated splitting, Wang and Cao \cite{WangCao16} obtained a special version of upper and lower estimates for the dimension of ergodic hyperbolic measures.

In this section, we establish a $C^1$ version of the Eckmann-Ruelle conjecture by proving the exactness of the dimension; see \Cref{main thmsecond}. Specifically, we prove that if the Oseledets splitting into stable and unstable subbundles is dominated on the support of $m$, and if the stable and unstable dimensions are themselves exact, then the dimension of $m$ is exact and equals the sum of its stable and unstable dimensions. As a corollary, we obtain \Cref{C^1 dimension estimate 2}, which confirms that the Eckmann-Ruelle conjecture holds in the $C^1$ setting for simply dominated measures.

Our technical approach is based on that of Barreira, Pesin, and Schmeling \cite{Barreira_Pesin_Schmeling}. We construct the required partitions and compute the dimension by making essential use of fake foliation charts and the dominated splitting.

Now, we turn to the technical details and definitions. Denote the regular points with respect to $m$ by $\Gamma_{\mathrm{reg}}$. Let $E^u$ be the sub-bundle of all positive Lyapunov exponents, and let $E^s$ be the sub-bundle of all negative Lyapunov exponents. We assume that $E^u\oplus E^s$ is a dominated splitting on the support of $m$. 

For any $x\in \Gamma_{\mathrm{reg}}\cap \mathrm{supp}m$, the (global) stable and unstable manifolds are defined as $$W^s(x)=\{ y\in M: \limsup_{n\rightarrow \infty}\frac{1}{n}\log (d(f^nx,f^ny))<0 \}$$
and
$$W^u(x)=\{ y\in M: \limsup_{n\rightarrow \infty}\frac{1}{n}\log (d(f^{-n}x,f^{-n}y))<0 \},$$
where $d$ is the distance induced by the Riemannian metric on $M$. When $E^u\oplus E^s$ is dominated on the support of $m$, Abdenur, Bonatti, and Crovisier proved in \cite[Section 8]{ABC} that $W^u(x)$  (resp. $W^s(x)$) is an injectively immersed $C^1$ manifold with dimension $\dim E^u(x)$ (resp. $\dim E^s(x)$).

Recall the definitions of the stable and unstable dimension in \Cref{section 7.3}. 
Choosing a measurable partition $\xi^u$ subordinate to $W^u$ and $\xi^s$ subordinate to $W^s$, we define the conditional measures of $m$ associated with $\xi^u$ and $\xi^s$ as $\{m_x^u\}$ and $\{m^s_x\}$. For $*\in\{s,u\}$, let $B^*(x,\epsilon)$ denote the ball in $W^*(x)$ centered at $x$ of radius $\epsilon>0$. We define 
\begin{equation*}
    \ud^*(x,\xi^*)=\liminf_{\epsilon\rightarrow 0}\frac{\log m^*_x(B^*(x,\epsilon))}{\log\epsilon}
\end{equation*}
and
\begin{equation*}
    \od^*(x,\xi^*)=\limsup_{\epsilon\rightarrow 0}\frac{\log m^*_x(B^*(x,\epsilon))}{\log\epsilon}.
\end{equation*}
One can easily check that $\ud^*(x,\xi^*)$ and $\od^*(x,\xi^*)$ are independent of the choice of the subordinate partitions and are constants for $m$ almost every point, denoted by $\ud^*$ and $\od^*$. 
For $* \in \{s, u\}$, we say that the {$*$-dimension of $m$ is exact} if $\ud^* = \od^*$. In this case, we denote this common value by $d^* := \ud^* = \od^*$.

It was shown by Ledrappier and Young in \cite{LEDRAPPIER_YOUNG_B} that for $C^2$ diffeomorphisms, both the stable dimension $d^s$ and the unstable dimension $d^u$ are exact. This regularity assumption was later weakened to $C^{1+\alpha}$ by Brown in \cite{AaronBrown2022}. In the present work, under the assumption of $C^1$ regularity, we have proved in \Cref{section 7.3} that $d^s$ and $d^u$ are exact for any simply dominated measure.

Based on the theory of Ledrappier and Young, Barreira, Pesin, and Schmeling proved in \cite{Barreira_Pesin_Schmeling} that for a $C^{1+\alpha}$ diffeomorphism, the dimension of any hyperbolic measure $m$ is exact and satisfies the formula $\dim m = d^s + d^u$.

Following their ideas, we prove the following \Cref{main thmsecond}, which is a $C^1$ version of their result. 
\begin{theorem}\label{main thmsecond}
     Let $f \in \Diff(M)$ and $m$ be an ergodic hyperbolic measure of $f$. If
     $d^s$ and $d^u$ are exact and $E^u\oplus E^s$ is dominated on $\mathrm{supp}m$, then for $m$-a.e. $x\in M$,
    \begin{equation*}
        \limsup_{\epsilon\rightarrow0}\frac{\log m(B(x,\epsilon))}{\log\epsilon}
        =\liminf_{\epsilon\rightarrow0}\frac{\log m(B(x,\epsilon))}{\log\epsilon}
        =d^u+d^s.
    \end{equation*}
\end{theorem}
Note that $d^u$ and $d^s$ are dynamically defined dimensions and $\dim m$ is defined measure-theoretically; hence the result shows the coincidence of the dynamical dimensions and the measure-theoretic dimension.

We remark that \Cref{main thmsecond} also has a non-ergodic version. The technique is a standard reduction using the Ergodic Decomposition Theorem. The details are omitted in this paper, and we refer the reader to \cite[Section 7]{Barreira_Pesin_Schmeling} for the idea.
\subsection{Construction of partitions}
In this section, we use local charts and fake foliation (see \Cref{fake foliation charts def}) and the local argument (see \Cref{local arg here}) to construct special partitions and subordinate partitions in preparation for the proof of the main theorem in our $C^1$ setting. 
 The partition $\tP$ constructed here is guaranteed by the dominated splitting, which is different from  \cite{Barreira_Pesin_Schmeling}.

\subsubsection{Fake foliations}
Let $f\in \Diff(M)$ preserve an ergodic Borel probability $m$, and the corresponding splitting $E^u\oplus E^s$ is dominated on the support of $m$. The following construction of fake foliation is similar to \cite[Section 2]{Part1} and \Cref{section fake foliation}.


 By the arguments in \cite[Section 2]{Part1} or \Cref{section fake foliation}, one has:
\begin{lemma}\label{ABC ergodic lemmassecond}
Given any $\epsilon>0$, there exist a full measure set $\Gamma'$, a measurable function A: $\Gamma'\rightarrow[1,\infty)$, and a constant $N>0$  such that   $A(f^{\pm}x)\leq \mathrm{e}^{\epsilon}A(x)$ and for any $k\in \mathbb N $,
\begin{equation*}
    \prod_{\ell=0}^{k-1}\parallel Df^N|_{E^s(f^{\ell N}(x))}\parallel \leq A(x)\cdot \mathrm{e}^{kN(\lambda^s+\epsilon)} 
\end{equation*}
\begin{equation*}
    \prod_{\ell=0}^{k-1}\parallel Df^{-N}|_{E^u(f^{-\ell N}(x))}\parallel \leq A(x)\cdot \mathrm{e}^{kN(-\lambda^u+\epsilon)} 
\end{equation*}
\end{lemma}
Fix a constant $C_f\geq 100\max\{|Df^\pm|,\mathrm{e}^{|\lambda_1|+100\epsilon},\mathrm{e}^{|\lambda_r|+100\epsilon}\}$. There exists a constant $\Delta>0$   such that for any $x\in M$, $\exp_x:\{|v|\leq \Delta\}\rightarrow M$ is a diffeomorphism to the image with $|D\exp^{\pm}|\leq 2$.
We define $\tf_x,\ r_0$ and $\tg_x$ as in \Cref{section fake foliation}. 
Similarly, there exists a small constant $\epsilon_0'>0$ such that for any $\epsilon_0\leq\epsilon_0'$ and the $\tg_x$, $N$ defined as above, 
\begin{equation*}
    \mathrm{e}^{-\epsilon N}\leq \frac{|D\tf^{\pm N}_x(0)v|}{|D\tg^{\pm N}_x(y)v|}\leq \mathrm{e}^{\epsilon N},\ \ \forall x\in M,\ \ v ,y\in T_xM .   
\end{equation*}

Then, we introduce the box norm $|\cdot|'$ on each $T_xM,x\in\Gamma'$ respect to the splitting $T_xM=E^u_x\oplus E^s_x$ as in \Cref{section fake foliation}. 
Let $R_x(\rho)$ and $R_x  ^*(\rho)$ be the ball centered at $0$ with radius $\rho$ under the box norm in $T_xM$ and $E^*_x$, where $*\in\{u,s\}$. 
There exists a constant $K>1$  such that   $K^{-1}|\cdot|\leq |\cdot|'\leq K|\cdot|$. 

One can choose a constant $C_0>K$ large enough such that for any $x\in\Gamma'$ and a $*$-dimensional linear subspace $V^*(x)$ as the graph of a linear function: $E^*_x\rightarrow E^{\hat*}_x$ with slope $\leq \frac{1}{C_0}$, $*,\hat*\in\{u,s\}$ and $T_xM=E^*_x\oplus E^{\hat*}_x$, one has
\begin{equation*}
     \mathrm{e}^{-\epsilon N}\leq\frac{\|D\tf_x^{\pm N}(0)|_{V^*(x)}\|}{\|D\tf_x^{\pm N}(0)|_{E^*(x)}\|}\leq \mathrm{e}^{\epsilon N} 
\end{equation*}

Fix a constant $\gamma_0<\frac{\pi}{10^6\cdot C_0^{100}}$. Following the standard discussions on the graph transformation, see \cite{Katok_book,HPS,BW2010, LVY13} for instance, we can choose $\epsilon_0>0$ and $r_0>0$ small enough,   such that   the following lemma holds, where the constant $C_1$ is also given as in \Cref{section fake foliation}.

\begin{lemma}[Fake foliations]\label{fake $u$-foliationsecond}
For each $x\in\Gamma'$ and $*\in\{s,u\}$ there exist unique global fake foliations $\thF^{*}_x$ on $T_xM$ with $C^1$ leaves,  such that   for any $y \in T_xM$:
		\begin{enumerate}
			\item[a)] The unique leaf containing $y$, denoted by $\tW^{*}_x(y)$, is the graph of a $C^1$ function $\varphi:E^{*}_x\rightarrow E^{\hat *}_x$ with $|D\varphi|\leq\frac{\gamma_0}{C_1}$. 
			\item[b)] The foliations are invariant under $\tg_{\cdot}$ in the sense that 
			$$
			\tg_x\left(\thF^{*}_x(y)\right) = \thF^{*}_{f(x)}(\tg_x(y)).
			$$
		\end{enumerate}
\end{lemma}


	\begin{remark}\label{r.foliation.size}
		
        It is important to note that the fake foliations are defined on the entire tangent bundle $T_\cdot M$ and are invariant under the global map $\tg_\cdot$. However, within the local charts, that is, for vectors $v \in T_\cdot M$ where $|v|\leq r_0$, the map $\tg_\cdot$ coincides with the real dynamics $\tf_\cdot$. Consequently, the foliations are locally invariant under $\tf_\cdot$.
        
        This local invariance is the crucial property that allows the foliations to capture the dynamics of $f$ within this region, and we refer to it as the \textbf{local invariance property}.
	\end{remark}

For simplicity, when $y\in \exp_x\{|v|\leq r_0\}$, we also denote $\tW_x^*(y) :=  \tW^*_x(\exp_x^{-1}y)$ for $*\in\{s,u\}$.
 For $0<\delta\leq \delta_0 :=  \frac{r_0}{100C^N_fC_0^{10}}$, the local unstable and stable manifolds are defined by
 $$W^u_{x,2\delta}(x) := \exp^{-1}_x(\text{component of }W^u(x)\cap \exp_x(R(2\delta A(x)^{-1}))\text{ that contains }x),$$
 $$W^s_{x,2\delta}(x) := \exp^{-1}_x(\text{component of }W^s(x)\cap \exp_x(R(2\delta A(x)^{-1}))\text{ that contains }x).$$ 
 As in \cite[Lemma 2.5]{Part1}, it follows that
\begin{lemma}\label{lemma 8 the sizesecond}
For any $x\in\Gamma'$,
    $$W^u_{x,2\delta}(x)=\tW^u_x(x)\cap R_x(2\delta A(x)^{-1}),$$ $$W^s_{x,2\delta}(x)=\tW^s_x(x)\cap R_x(2\delta A(x)^{-1}).$$
\end{lemma}

\subsubsection{Special partitions}
In this section, we construct the special partitions $\tP$, $\xi^u$, and $\xi^s$. Our approach is based on the methods of \cite{Ledrappier_Strelcyn_1982}, which have since been adapted to other settings, such as partially hyperbolic systems in \cite{YANG_expanding_entropy}. The first step in this construction is to define the small boundary condition.
\begin{definition}[Small boundary condition]
    A measurable partition $\tA$ satisfies the small boundary condition if for any $ \lambda>0$,  $$\sum_{n=1}^{\infty}m(B(\partial\tA,\e^{-n\lambda}))<\infty,$$ where $\partial\tA$ denotes the boundary of partition $\tA$ and $B(\partial\tA,r):=\{x\in M:d(x,\partial\tA)\leq r\}$. A Borel set $A$ satisfies the small boundary condition if the partition $\{A,M-A\}$ satisfies the small boundary condition.
\end{definition}

By \cite[Proposition 3.2]{Ledrappier_Strelcyn_1982}, it follows that for any $x\in M$, $B(x,r)$ satisfies the small boundary condition for $\mathrm{Leb}$-a.e. $r>0$ small enough. This property is crucial for the construction of subordinate partitions.

Recall that $\Lambda_{A_0}:=\{x\in\Gamma':A(x)\leq A_0\}$.
Given $0<\epsilon<1$, we fix $A_0>0$ such that $m(\Lambda_{A_0})>1-0.1\epsilon$. 
By \Cref{lemma 8 the sizesecond}, this choice ensures that local stable and unstable manifolds have a uniform size of $\delta_0A_0^{-1}$. 
We then choose two positive numbers $r_{A_0}>r_{A_0}'>0$, sufficiently small compared to $\delta_0A_0^{-1}$, such that for any two points $x,y\in\Lambda_{A_0}$ with $d(x,y)\leq r_{A_0}'$, the local unstable manifold $B^u(x,r_{A_0})$ and the local stable manifold $B^s(y,r_{A_0})$ intersect at exactly one point and have uniform angles away from zero. 
Moreover, the slopes of the graphs $\exp_x^{-1}(B^s(y,r_{A_0}))$ and $\exp_y^{-1}(B^u(x,r_{A_0}))$ are bounded by $\frac{1}{10^5}$.

By the compactness of \(M\) and the preceding discussion, there exist finitely many balls
\(\{B_i\}_{i=1}^{m}\) that cover \(M\), such that each \(B_i\) satisfies the small boundary condition
and has radius smaller than \(\frac{r_{A_0}'}{100}\).
Then we define the finite partition $\tP$ as
\begin{equation*}
    \tP:=\{B_1,B_2\setminus B_1,\cdots,B_m\setminus(\bigcup_{i=1}^{m-1}B_i)\}.
\end{equation*}
It follows that $\tP$ satisfies the small boundary condition.

For $x\in\Lambda_{A_0}$, we define the local stable manifold $V^s_{\mathrm{loc}}(x):=\exp_x(W^s_{x,2\delta_0}(x))$ and the local unstable manifold $V^u_{\mathrm{loc}}(x):=\exp_x(W^u_{x,2\delta_0}(x))$. Then $\hat\xi^s$ and $\hat\xi^u$ can be defined as follows, which is analogous to \cite{LEDRAPPIER_YOUNG_A} and \cite{YANG_expanding_entropy}. 
\begin{equation*}
\hat\xi^s(x)=\left\{
    \begin{aligned}
        &V^s_{\mathrm{loc}}(y)\cap\tP(x) &\forall\, x\in V^s_{\mathrm{loc}}(y) \text{ for some } y\in\tP(x)\cap\Lambda_{A_0}, \\
        &\tP(x)\setminus\left(\bigcup_{y\in\tP(x)\cap\Lambda_{A_0}}V^s_{\mathrm{loc}}(y)\right)      & \forall\, x\notin V^s_{\mathrm{loc}}(y) \text{ for any } y\in\tP(x)\cap\Lambda_{A_0}.\ \  
    \end{aligned}
    \right .
\end{equation*}
\begin{equation*}
\hat\xi^u(x)=\left\{
    \begin{aligned}
        &V^u_{\mathrm{loc}}(y)\cap\tP(x) &\forall\, x\in V^u_{\mathrm{loc}}(y) \text{ for some } y\in\tP(x)\cap\Lambda_{A_0}, \\
        &\tP(x)\setminus\left( \bigcup_{y\in\tP(x)\cap\Lambda_{A_0}}V^u_{\mathrm{loc}}(y) \right)      &\forall\, x\notin V^u_{\mathrm{loc}}(y) \text{ for any } y\in\tP(x)\cap\Lambda_{A_0}.\ \  
    \end{aligned}
    \right .
\end{equation*}
Define $\xi^u:=\bigvee_{n=0}^\infty f^{n}\hat\xi^u$ and $\xi^s:=\bigvee_{n=0}^\infty f^{-n}\hat\xi^s$. For the same reasons as in \cite{Ledrappier_Strelcyn_1982},  $\xi^u$ is subordinate to $W^u$ and $\xi^s$ is subordinate to $W^s$.

\subsubsection{Fine properties}
In this section, we prove that the $\tP$, $\xi^u$, and $\xi^s$ constructed above satisfy fine properties analogous to \cite[Section 4]{Barreira_Pesin_Schmeling} with several adjustments. For $k,l\geq0$, from now on, we denote the notation that for any measurable partition $\eta$, $$\eta_k^l:=\bigvee_{n=-k}^lf^{-n}\eta.$$ Let $h:=h_m(f)$ be the metric entropy of $f$. Without loss of generality, we can assume that $h_m(f,\tP)>h-\frac{\epsilon}{10}$, since the diameter of the small boundary balls (and thus $\tP$) can be arbitrarily small. Recall that the local charts and the fake foliation charts are defined in \Cref{fake foliation charts def}.

\begin{lemma}\label{lemma 6second}
There exist a set $\Gamma\subseteq \Lambda_{A_0}$ of measure $m(\Gamma)>1-\frac{\epsilon}{2}$, an integer $n_0\geq 1$, and a constant $C>1$   such that   for any $x\in\Gamma$ and any $n\geq n_0$, the following hold:\\
  A. For any integers $k,l\geq 1$,
\begin{equation}\label{eq5second}
    C^{-1}\e^{-(l+k)(h+\epsilon)}\leq m(\tP_k^l(x))\leq C\e^{-(l+k)(h-\epsilon)},
\end{equation}
\begin{equation}\label{eq6second}
    C^{-1}\e^{-kh-k\epsilon}\leq m^s_x(\tP^0_k(x))\leq C\e^{-kh+k\epsilon},
\end{equation}
\begin{equation}\label{eq7second}
    C^{-1}\e^{-lh-l\epsilon}\leq m^u_x(\tP^l_0(x))\leq C\e^{-lh+l\epsilon}.
\end{equation}
  B.
\begin{equation}\label{eq8second}
    \xi^s(x)\cap(\bigcap_{n\geq0}\tP^n_0(x))\supseteq B^s(x,\e^{-n_0}),
\end{equation}
\begin{equation}\label{eq9second}
    \xi^u(x)\cap(\bigcap_{n\geq0}\tP^0_n(x))\supseteq B^u(x,\e^{-n_0}).
\end{equation}
  C.
\begin{equation}\label{eq10second}
    \e^{-d^sn-n\epsilon}\leq m^s_x(B^s(x,\e^{-n}))\leq \e^{-d^sn+n\epsilon},
\end{equation}
\begin{equation}\label{eq11second}
    \e^{-d^un-n\epsilon}\leq m^u_x(B^u(x,\e^{-n}))\leq \e^{-d^un+n\epsilon}.
\end{equation}

\begin{proof}
We examine the properties one by one and show that they can be satisfied by appropriately choosing  $\Gamma,n_0$ and $C$:
\begin{itemize}
\item[A]: Since $\xi^u$ is subordinate to $W^u$, we first claim that

\begin{claim}
For any $x\in\Lambda_{A_0}$ and $l\in\mathbb N$, one has
$$\tP^l_0(x)\cap\xi^u(x)=\bigvee_{i=0}^lf^{-i}\hat\xi^u(x)\cap\xi^u(x).$$
\begin{proof}
Since by definition one has $\tP\prec\hat\xi^u$, which implies that 
$$\tP^l_0(x)\cap\xi^u(x)\supseteq\bigvee_{i=0}^lf^{-i}\hat\xi^u(x)\cap\xi^u(x).$$
So we only need to prove that for $0\leq i\leq l$,
$$f^i\left( \tP^l_0(x)\cap\xi^u(x)\right) \subseteq \hat\xi^u(f^ix).$$

Before giving the proof, we record a basic property of the fake leaves.
Since for any $V^u_{\mathrm{loc}}(y),\ y\in\Lambda_{A_0}$ and any point $z\in V^u_{\mathrm{loc}}(y)$, the local argument implies that $\mathrm{diam} (f^{-j}V^u_{\mathrm{loc}}(y))\leq r_0$. Thus, $f^{-j}V^u_{\mathrm{loc}}(y)$ always stays in the local chart of $f^{-j}z$. By the construction of fake leaves, it implies $$V^u_{\mathrm{loc}}(y)\subseteq \exp_z(\tW^u_{z}(z)).$$

Now, we start to prove the claim. First, consider any $x_0\in  \tP^l_0(x)\cap\xi^u(x)$. Since $x\in\Lambda_{A_0}$, by the definition of $\hat\xi^u$ and the previous property of fake leaves, one has $$\tP^l_0(x)\cap\xi^u(x)\subseteq \exp_{x_0}\tW^u_{x_0}(x_0).$$ Note that each element of $\tP$ has diameter much smaller than $r_0$, which is the uniform size of the local charts. Thus, the local invariance property in \Cref{r.foliation.size} holds for $0\leq i\leq l$. So
$$f^i\left( \tP^l_0(x)\cap\xi^u(x)\right)\subseteq\exp_{f^i(x_0)} \left(\tW^u_{f^i(x_0)}(f^i(x_0))\right).$$
By the graph properties of fake leaves, it follows that the submanifold distance $\tilde d^u$,
which is induced by the manifold $\exp_{f^i(x_0)}\left(\tW^u_{f^i(x_0)}(f^i(x_0))\right)$ on $f^i\left( \tP^l_0(x)\cap\xi^u(x)\right)\subseteq W^u(f^i(x_0))$, is equivalent to the Riemannian distance $d$ on $M$, that is
$$\frac{1}{10}\tilde d^u(\cdot,\cdot)\leq d(\cdot,\cdot)\leq \tilde d^u(\cdot,\cdot).$$

Together with the previous property of the fake leaves, it follows that if $f^i(x_0)\in V^u_{\mathrm{loc}}(y)$ for some $y\in\Lambda_{A_0}$, then $$V^u_{\mathrm{loc}}(y)\subseteq\exp_{f^i(x_0)}\left(\tW^u_{f^i(x_0)}(f^i(x_0))\right)$$ and thus the submanifold distance $d^u$ on $V^u_{\mathrm{loc}}(y)$ is equivalent to the Riemannian distance $d$ on $M$, i.e.
$$\frac{1}{10}d^u(\cdot,\cdot)\leq d(\cdot,\cdot)\leq d^u(\cdot,\cdot).$$

Since each element of $\tP$ has diameter much smaller than $\delta_0A_0^{-1}$, which is the radius of the local unstable manifold $V_{\mathrm{loc}}^u(y)$ when $y\in\Lambda_{A_0}$, the previous discussion implies  that whenever $f^i\left( \tP^l_0(x)\cap\xi^u(x)\right)$ has a non-empty intersection with the union of the local unstable manifolds in the definition of $\hat\xi^u$, it is contained in one of the local unstable manifolds. Thus, by checking the definition of $\hat\xi^u$, we have for any $0\leq i\leq l$,
$$f^i\left( \tP^l_0(x)\cap\xi^u(x)\right) \subseteq \hat\xi^u(f^ix).$$
Hence, the proof of the claim is complete.
\end{proof}
\end{claim}

Then, by the definition of $\tP$ and $\xi^u$, one has
\begin{equation*}
\begin{aligned}
    m^u_x(\tP^l_0(x))&=m^u_x(\tP^l_0(x)\cap\xi^u(x))\\
     &=m^u_x(\bigvee_{i=0}^lf^{-i}\hat\xi^u(x)\cap\xi^u(x))\\
    &=m^u_x(\bigvee_{i=0}^lf^{-i}\xi^u(x))\\
    &=m^u_x\left((\xi^u)_0^l(x)\right).
\end{aligned}
\end{equation*}
and similarly,
\begin{equation*}
    m^s_x(\tP^0_k(x))=m^s_x\left((\xi^s)^0_k(x)\right).
\end{equation*}
By \cite[Theorem 1.1]{Part1}, under the present dominated splitting assumptions,
the stable and unstable entropies coincide with the metric entropy. Then the properties in A follow from the Shannon-McMillan-Breiman theorem and its partial version, see \Cref{SMB type thm} or \cite[Lemma 9.3.1]{LEDRAPPIER_YOUNG_B}.
\item[B]: By the definition of $\xi^u$ and $\xi^s$, one has $\tP\prec\hat\xi^u$ and $\tP\prec\hat\xi^s$. Then we have $$\xi^s(x)\subseteq\bigcap_{n\geq0}\tP^n_0(x)$$ 
and thus
\begin{equation}\label{review eq1}
\xi^s(x)\cap\left(\bigcap_{n\geq0}\tP^n_0(x)\right)=\xi^s(x).
\end{equation}
Similarly, it follows that 
\begin{equation}\label{review eq2}
\xi^u(x)\cap\left(\bigcap_{n\geq0}\tP^0_n(x)\right)=\xi^u(x).
\end{equation}
Thus, the properties in B can be obtained from the subordinate properties of $\xi^u$ and $\xi^s$, more specifically, property (a) of \Cref{def subordiante partition}.
\item[C]: The properties in C directly follow from the definition of $d^u$ and $d^s$.
\end{itemize}
\end{proof}
\end{lemma}

Now, we fix a constant integer $a\in\mathbb N$ that is larger than $\max\{\frac{1}{\lambda^u-100\epsilon},\frac{1}{-\lambda^s-100\epsilon}\}$ throughout the section.

\begin{lemma}\label{lemma 7second}
For any $x\in \Gamma$ and $n\geq n_0$, one has 
\\
D.
\begin{equation*}
    \tP^{an}_{0}(x)\cap\xi^u(x)\subseteq B^u(x,\e^{-n})\subseteq \tP(x)\cap\xi^u(x),
\end{equation*}
\begin{equation*}
    \tP^{0}_{an}(x)\cap\xi^s(x)\subseteq B^s(x,\e^{-n})\subseteq \tP(x)\cap\xi^s(x),
\end{equation*}
\begin{equation*}
    \tP^{an}_{an}(x)\subseteq B(x,\e^{-n}),
\end{equation*}
\begin{proof}
The right hand side follows directly from \Cref{eq8second} and \Cref{eq9second}. We only prove the left hand side.

We first prove that $$ \tP^{an}_{0}(x)\cap\xi^u(x)\subseteq B^u(x,\e^{-n}).$$ For $y\in\tP^{an}_{0}(x)\cap\xi^u(x)$, we have $f^k(y)\in\tP(f^k(x))$ for $0\leq k\leq an$; therefore, $f^k(y)$ stays in the local chart of $f^k(x)$. Consider the fake foliation charts along the orbit of $x$. By the invariant property, we have $\exp_{f^k(x)}^{-1}f^k(y)\in\tW^u_{f^k(x)}(f^k(x))$. By the local argument (see \Cref{local arg here}), 
\begin{equation*}
\begin{aligned}
    |\exp^{-1}_xy-0|&\leq A(f^{an}(x))\cdot C_f^N\e^{-(\lambda^u-3\epsilon)an}\cdot|\exp_{f^{an}(x)}^{-1}f^{an}(y)-0|\\
    &\leq A_0C^N_f\e^{-(\lambda^u-4\epsilon)an}\cdot 2r_{A_0}\\
    &\leq \frac{1}{100}\e^{-n}.
\end{aligned}
\end{equation*}
Hence, $$\tP^{an}_{0}(x)\cap\xi^u(x)\subseteq B^u(x,\e^{-n}).$$ 
Similarly, one has $$\tP^{0}_{an}(x)\cap\xi^s(x)\subseteq B^s(x,\e^{-n}).$$

It remains to prove that $\tP^{an}_{an}(x)\subseteq B(x,\e^{-n})$. If $y\in\tP^{an}_{an}(x)$, then $f^k(y)\in\tP(f^k(x))$ for $-an\leq k\leq an$. It follows that $f^k(y)$ stays in the local chart of $f^k(x)$. Denote $$z_k:=\tW^u_{f^k(x)}(f^k(x))\cap\tW^s_{f^k(x)}(f^k(y)).$$ By the local product structure of the fake foliations, $z_k$ stays in the local chart of $f^k(x)$ for any $-na \leq k\leq na$. By the local invariant property, we have $\tf_{f^k(x)}z_k=z_{k+1}$. For the same reason, the local argument (see \Cref{local arg here}) implies that $|\exp_x^{-1}y-z_0|\leq\frac{1}{100}\e^{-n}$ and $|z_0-0|\leq\frac{1}{100}\e^{-n}$. Then we have $|\exp_x^{-1}y|\leq\frac{1}{50}\e^{-n}$, and thus $d(x,y)\leq \e^{-n}$. Hence, we have proved  that $$\tP^{an}_{an}(x)\subseteq B(x,\e^{-n}).$$  
\end{proof}
\end{lemma}

Using the local argument as above, and together with the domination condition and fake foliations, the following tube-covering lemma can be obtained, which controls the shape of $\tP^{an}_0(x)$ and $\tP_{an}^0(x)$ for any $x\in\Gamma$.
\begin{lemma}\label{tube lem}
    If $x\in\Gamma$, $n\geq n_0$, then
    \begin{equation*}
        \tP^{an}_0(x)\subseteq U(B^s(x,r_{A_0}),\e^{-n}),
    \end{equation*}
    \begin{equation*}
        \tP^{0}_{an}(x)\subseteq U(B^u(x,r_{A_0}),\e^{-n}),
    \end{equation*}
    where $U(X,r):=\{x\in M:d(x,X)<r\}$ denotes the $r$-neighborhood of $X$.
    \begin{proof}
        Since the diameter of $\tP$ is much smaller than $r_{A_0}$, $\tP(x)$ is contained in the fake foliation chart of $x$. Thus, for any $y\in\tP^{an}_0(x)$, there exists a point $$p:=B^s(x,r_{A_0})\cap\exp_x(\tW^u_{x,2\delta}(y)).$$ 
        As in the previous lemma, we can use the local argument to show that
        $$|\exp_x^{-1}y-\exp_x^{-1}p|\leq\frac{\e^{-n}}{100}.$$
        Thus, it follows that $d(y,p)<\e^{-n}$. Since $p\in B^s(x,r_{A_0})$, the proof of the first statement is complete. The second statement follows in the same manner. 
        
    \end{proof}
\end{lemma}
\begin{lemma}
    For any $x\in\Gamma$ and $n\geq n_0$, 
    \begin{equation}\label{markov ssecond}
        \tP^{an}_{an}(x)\cap\xi^s(x)=\tP^0_{an}(x)\cap\xi^s(x),
    \end{equation}
    \begin{equation}\label{markov usecond}
        \tP^{an}_{an}(x)\cap\xi^u(x)=\tP_0^{an}(x)\cap\xi^u(x).
    \end{equation}
    \begin{proof}
        It follows directly from \Cref{review eq1} and \Cref{review eq2}.
    \end{proof}
\end{lemma}

Using the Borel density point arguments as in \cite[Section 4]{Barreira_Pesin_Schmeling}, we can enlarge $n_0$ if necessary and choose a set $\hat\Gamma\subseteq \Gamma$ with $m(\hat\Gamma)>1-\epsilon$   such that   for any  $n\geq n_0$ and $x\in\hat\Gamma$, one has
\begin{equation}\label{eq19second}
    m(B(x,\e^{-n})\cap\Gamma)\geq \frac{1}{2}m(B(x,\e^{-n})),
\end{equation}
\begin{equation}\label{eq20second}
    m^s_x(B^s(x,\e^{-n})\cap\Gamma)\geq \frac{1}{2}m^s_x(B^s(x,\e^{-n})),
\end{equation}
\begin{equation}\label{eq21second}
    m^u_x(B^u(x,\e^{-n})\cap\Gamma)\geq\frac{1}{2}m^u_x(B^u(x,\e^{-n})).
\end{equation}
Together with \Cref{lemma 6second} and \Cref{lemma 7second}, it follows that
\begin{lemma}\label{lemma 41}
    There exists a positive constant $D>0$   such that   for any $k\geq1$ and every $x\in\hat\Gamma$,
    \begin{equation*}
        m^s_x(\tP^k_0(x)\cap\Gamma)\geq D,
    \end{equation*}
    \begin{equation*}
        m^u_x(\tP^0_k(x)\cap\Gamma)\geq D.
    \end{equation*}
    \begin{proof}
        The same as \cite[Proposition 4]{Barreira_Pesin_Schmeling}. 
        By \Cref{eq8second}, for any $k\geq 1$ and $x\in\Gamma$, one has
        $$\tP^k_0(x)\cap\Gamma\supseteq B^s(x,\e^{-n_0})\cap\Gamma.$$
        Thus, by \Cref{eq20second} and \Cref{eq10second},
        $$m^s_x\left(\tP^k_0(x)\cap\Gamma\right)\geq \frac{1}{2}m^s_x\left(B^s(x,\e^{-n_0})\right)\geq\frac{1}{2}\e^{-(d^s+\epsilon)n_0}:=D.$$
        The other inequality follows similarly. 
    \end{proof}
\end{lemma}

For any given $x\in\Gamma$ and $n\geq n_0$, we define two classes of elements of partition $\tP_{an}^{an}$ denoted as $\tR(n)$ and $\tF(n)$ as in \cite{Barreira_Pesin_Schmeling}:
$$\tR(n):=\{\tP^{an}_{an}(y)\subseteq \tP(x):\tP^{an}_{an}(y)\cap\Gamma\neq\emptyset\},$$
$$\tF(n):=\{\tP^{an}_{an}(y)\subseteq\tP(x):\tP^0_{na}(y)\cap\hat\Gamma\neq\emptyset,\tP^{an}_0(y)\cap\hat\Gamma\neq\emptyset\}.$$
Here \(\tR(n)\) denotes the collection of \(n\)-scale rectangles that actually meet the good set \(\Gamma\).  The family \(\tF(n)\) imposes a different product-type condition: both the stable strip \(\tP^0_{na}(y)\) and the unstable strip \(\tP^{an}_0(y)\) have to meet the smaller density good set \(\hat\Gamma\).  Thus \(\tF(n)\) admits a ``local product structure''.  This distinction will be useful in the counting argument below: rectangles in \(\tR(n)\) have good estimates from \Cref{lemma 6second}, which are used to count rectangles intersecting \(\Gamma\), while \(\tF(n)\) is used to recover a controlled local product structure.

Now we can define the set $\tQ_n(x)$, which is, roughly speaking, the union of elements in $\tF(n)$ in the ``$\e^{-n}$-neighborhood'' of $x$, and which has a nice local product structure. 
\begin{lemma}\label{lemma 10second}
    There exist an integer $n_1\gg n_0$ and a constant $C_2>1$ large enough such that  : for any $x\in\Gamma$ and $n\geq n_1$,
    define $\tQ_n(x)$ by $$\tQ_n(x)=\bigcup\tP^{an}_{an}(y),$$ 
    where the union is taken over all $y\in M$ such that $\tP^{an}_{an}(y)\in \tF(n)$ and which satisfy the following properties:
    \begin{itemize}
    \item[1.] There exists a point $q\in\tP^{an}_0(y)\cap\hat\Gamma$   such that   $\exp_x^{-1}(B^s(q,\e^{-n_0}))\subseteq T_xM$ contains the graph of a function $\varphi_s:R^s_x(C^{30}_2\e^{-n})\to E^u_x$ with slope less than $\frac{1}{10^5}$ and $|\varphi_s(0)|\leq C_2\e^{-n}$.
     \item[2.] There exists a point $d\in\tP^{0}_{an}(y)\cap\hat\Gamma$   such that   $\exp_x^{-1}(B^u(d,\e^{-n_0}))\subseteq T_xM$ contains the graph of a function $\varphi_u:R^u_x(C^{30}_2\e^{-n})\to E^s_x$ with slope less than $\frac{1}{10^5}$ and $|\varphi_u(0)|\leq C_2\e^{-n}$.   
     \end{itemize}
    Then we have
    \begin{equation}\label{eq16second}
        B(x,\e^{-n})\cap\hat\Gamma\subseteq\tQ_n(x)\subseteq B(x,C_2^2\e^{-n}),
    \end{equation}  
    \begin{equation}\label{eq17second}
        B^s(x,\e^{-n})\cap\hat\Gamma\subseteq\tQ_n(x)\cap\xi^s(x)\subseteq B^s(x,C_2^2\e^{-n}),
    \end{equation} 
    \begin{equation}\label{eq18second}
        B^u(x,\e^{-n})\cap\hat\Gamma\subseteq\tQ_n(x)\cap\xi^u(x)\subseteq B^u(x,C_2^2\e^{-n}),
    \end{equation} 
    and for each $y\in\tQ_n(x)$, $\tP^{an}_{an}(y)\subseteq\tQ_n(x)$.
    
\begin{figure}[htbp]
    \centering
    \includegraphics[width=0.9\textwidth]{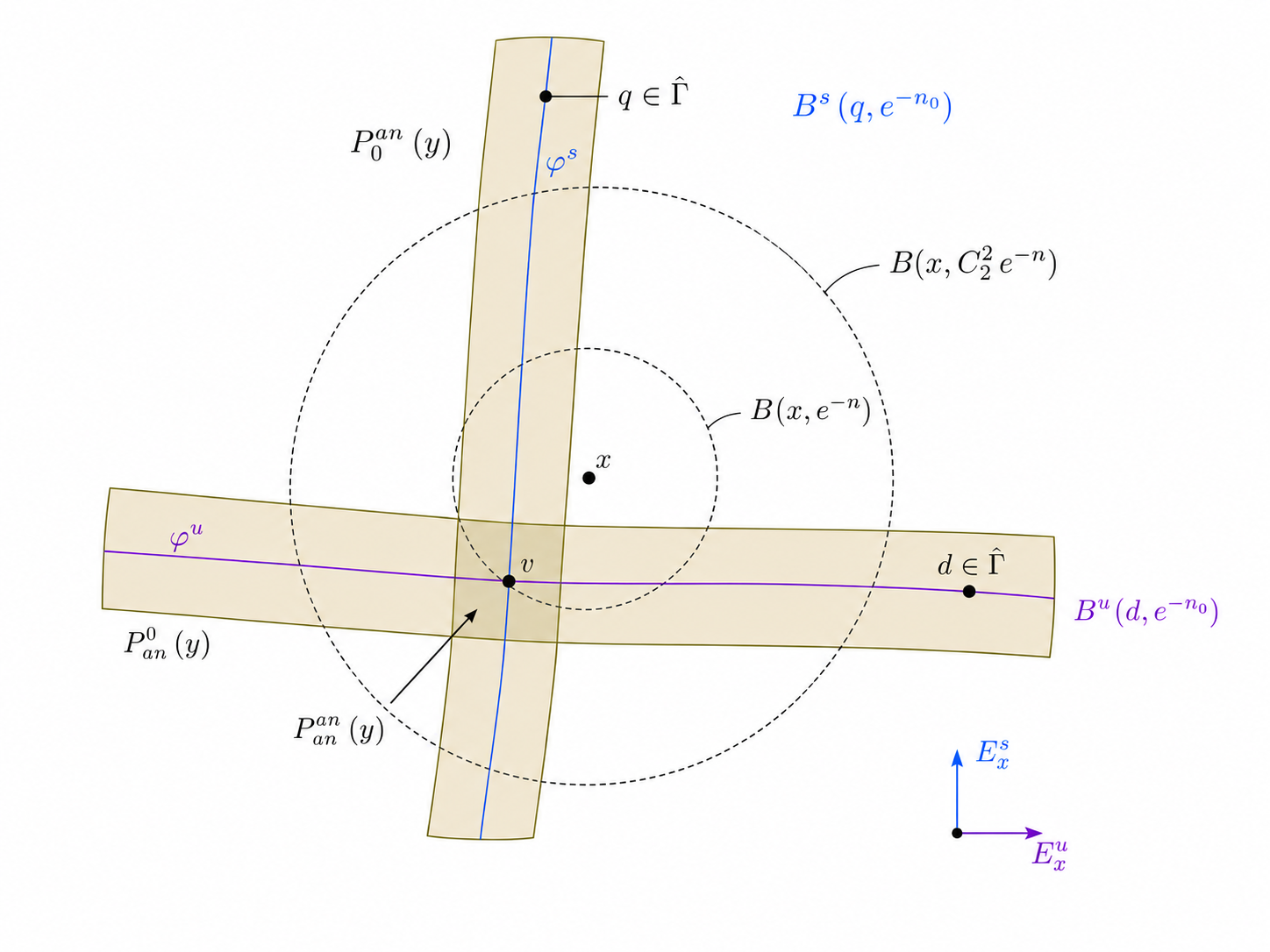}
    \caption{illustration for \Cref{lemma 10second}.}
\end{figure}

    \begin{proof}
    By the definition of $\tQ_n$, it follows that for each $y\in\tQ_n(x)$, $\tP^{an}_{an}(y)\subseteq\tQ_n(x)$.

    We first prove the left hand side of \Cref{eq16second} that 
    $$B(x,\e^{-n})\cap\hat\Gamma\subseteq\tQ_n(x).$$ 
    If $y\in B(x,\e^{-n})\cap\hat\Gamma$,
    clearly we have $\tP_{an}^{an}(y)\in \tF(n)$ since $y\in \tP^{an}_0(y)\cap\tP_{an}^0(y)$.
    Moreover, $y$ has local manifolds $B^u(y,\e^{-n_0}),B^s(y,\e^{-n_0})$ with uniform size $\e^{-n_0}$.
    Since $n_0\ll n_1\leq n$, one can choose $n_1=n_1(C_2)$ large enough such that $\exp_x^{-1}(B^s(y,\e^{-n_0}))$ contains the graph of a function  $\varphi_s:R^s_x(C^{30}_2\e^{-n})\to E^u_x$ and
    $\exp_x^{-1}(B^u(y,\e^{-n_0}))$ contains the graph of a function  $\varphi_u:R^u_x(C^{30}_2\e^{-n})\to E^s_x$. Recall that by the definition of $\tP$, those graphs have slope less than $\frac{1}{10^5}$. Since $d(x,y)\leq\e^{-n}$ and the local stable and unstable manifolds have uniform angles away from zero, there exists a constant $C_2>1$ such that $|\varphi_s(0)|\leq C_2\e^{-n}$ and $|\varphi_u(0)|\leq C_2\e^{-n}$.
    So the left hand side of \Cref{eq16second} is proved.

    The second step of the proof is to prove the right hand side of \Cref{eq16second} that $$\tQ_n(x)\subseteq B(x,C_2^2\e^{-n}).$$ Suppose $y\in\tQ_n(x)$, we only need to prove that $\tP^{an}_{an}(y)\subseteq B(x,C_2^2\e^{-n})$. 
   Let $q$ and $d$ be chosen as in property 1. and 2., together with \Cref{tube lem} one has $$\tP^{an}_{an}(y)=\tP^{an}_0(q)\cap\tP^0_{an}(d)\subseteq U(B^s(q,r_{A_0}),\e^{-n})\cap U(B^u(d,r_{A_0}),\e^{-n}).$$
    Let $v:=B^s(q,r_{A_0})\cap B^u(d,r_{A_0})$. By the uniform angle condition between stable and unstable manifolds, we can choose the constant $C_2$ large enough such that
\[ \tP^{an}_{an}(y)\subseteq B(v,C_2\e^{-n}). \]

Thus, it remains to prove that $v\in B(x,C_2\cdot(C_2-1)\e^{-n})$. 
To this end, recall that in the chart at $x$, the manifolds $B^s(q,r_{A_0})$ and $B^u(d,r_{A_0})$ contain the graphs of functions $\varphi_s$ and $\varphi_u$, respectively. Due to the properties of these functions, their graphs intersect at a unique point, which is $\exp_x^{-1}v$ and must satisfy the estimate:
\[ |\exp_x^{-1}v-0|\leq \frac{C_2\cdot(C_2-1)}{2}\cdot\e^{-n}. \]
Therefore $v\in B(x,C_2\cdot(C_2-1)\e^{-n})$ and the proof of \Cref{eq16second} is complete.

    Finally, it is not difficult to use \Cref{eq16second} to deduce \Cref{eq17second} and \Cref{eq18second}. Note that $\xi^s(x)$ and $\xi^u(x)$ are contained in the local stable and unstable manifold, and by the graph properties it follows that the submanifold distances $d^s$ on $V^s_{\mathrm{loc}}(x)$ and $d^u$ on $V^u_{\mathrm{loc}}(x)$ are equivalent to the Riemannian distance $d$ on $M$, i.e. $$\frac{1}{10}d^s(\cdot,\cdot)\leq d(\cdot,\cdot)\leq d^s(\cdot,\cdot)$$ and $$\frac{1}{10}d^u(\cdot,\cdot)\leq d(\cdot,\cdot)\leq d^u(\cdot,\cdot).$$ Thus by slightly enlarging $C_2$, \Cref{eq17second} and \Cref{eq18second} follows from \Cref{eq16second} directly, noting that one has $\xi^u(x)\supseteq B^u(x,\e^{-n_0})$ and $\xi^s(x)\supseteq B^s(x,\e^{-n_0})$ whenever $x\in\Gamma$ by \Cref{eq8second} and \Cref{eq9second}.
    \end{proof}
\end{lemma}

\begin{remark}\label{remark:Qn_difference}
The definition of $\tQ_n(x)$ has been modified accordingly from that in \cite{Barreira_Pesin_Schmeling} due to the intrinsic difficulties of the problem in the present setting. Geometric conditions 1. and 2. in the definition of $\tQ_n(x)$ are essential to ensure the local product structure in the setting of $C^1$ plus domination.
\end{remark}

By the graph properties in the definition of $\tQ_n(x)$ and the proof of the previous lemma, \Cref{lemma 7.11} follows directly (keep in mind that $n_1\gg n_0$). We omit the proof.
\begin{lemma}\label{lemma 7.11}
    For any $x\in\Gamma$, $n\geq n_1$, and $\tP_{an}^{an}(y)\subseteq \tQ_n(x)$: \begin{itemize}
        \item[1.] there exists a point $q\in\tP^{an}_0(y)\cap\hat\Gamma$ such that for any $\zeta\in B(x,C_2^2\e^{-n})\cap\Gamma$, $B^s(q,\e^{-n_0})\cap B^u(\zeta,\e^{-n_0})$ intersects exactly at one point.
        \item[2.] there exists a point $d\in\tP^{0}_{an}(y)\cap\hat\Gamma$ such that for any $\zeta\in B(x,C_2^2\e^{-n})\cap\Gamma$, $B^u(d,\e^{-n_0})\cap B^s(\zeta,\e^{-n_0})$ intersects exactly at one point.
    \end{itemize}
\end{lemma}

The following two lemmas, \Cref{lem 7.13} and \Cref{lemma **second}, can be proved similarly, which introduce $\tQ_n^*(x)$ and $\tQ_n^{**}(x)$ as larger sets that are analogous to $\tQ_n(x)$. Each enlargement uses a smaller graph domain but allows a larger displacement and a slightly larger slope,  which are chosen to absorb the coordinate changes used in \Cref{lem includingsecond}. They will both be used in the last section. The inclusion relationships between $\tQ_n(x)$, $\tQ^{*}_n(x)$, and $\tQ_n^{**}(x)$ are shown in \Cref{lem includingsecond}, which will be used in the proof of \Cref{lemma lowboundsecond}.
\begin{lemma}\label{lem 7.13}
        After increasing \(n_1\) if necessary, the following holds for every
\(x\in\Gamma\) and every \(n\ge n_1\).  Define $\tQ^*_n(x)$ by $$\tQ^*_n(x)=\bigcup\tP^{an}_{an}(y)$$
      where the union is taken over all $y\in M$ such that $\tP^{an}_{an}(y)\in \tF(n)$ and which satisfy the following properties:
    \begin{itemize}
    \item[1.] There exists a point $q\in\tP^{an}_0(y)\cap\hat\Gamma$   such that   $\exp_x^{-1}(B^s(q,\e^{-n_0}))\subseteq T_xM$ contains the graph of a function $\varphi_s:R^s_x(C^{20}_2\e^{-n})\to E^u_x$ with slope less than $\frac{2}{10^5}$ and $|\varphi_s(0)|\leq C^4_2\e^{-n}$.
     \item[2.] There exists a point $d\in\tP^{0}_{an}(y)\cap\hat\Gamma$   such that   $\exp_x^{-1}(B^u(d,\e^{-n_0}))\subseteq T_xM$ contains the graph of a function $\varphi_u:R^u_x(C^{20}_2\e^{-n})\to E^s_x$ with slope less than $\frac{2}{10^5}$ and $|\varphi_u(0)|\leq C^4_2\e^{-n}$.   
     \end{itemize}
    Then we have
    \begin{equation}\label{eq16asecond}
        B(x,C_2^3\e^{-n})\cap\hat\Gamma\subseteq\tQ^*_n(x)\subseteq B(x,C_2^5\e^{-n}),
    \end{equation}  
    \begin{equation}\label{eq17asecond}
        B^s(x,C_2^3\e^{-n})\cap\hat\Gamma\subseteq\tQ^*_n(x)\cap\xi^s(x)\subseteq B^s(x,C_2^5\e^{-n}),
    \end{equation} 
    \begin{equation}\label{eq18asecond}
        B^u(x,C^3_2\e^{-n})\cap\hat\Gamma\subseteq\tQ^*_n(x)\cap\xi^u(x)\subseteq B^u(x,C_2^5\e^{-n}),
    \end{equation} 
    and for each $y\in\tQ^*_n(x)$, $\tP^{an}_{an}(y)\subseteq\tQ^*_n(x)$.
\end{lemma}

\begin{lemma}\label{lemma **second}
        After increasing \(n_1\) if necessary, the following holds for every
\(x\in\Gamma\) and every \(n\ge n_1\).  Define $\tQ^{**}_n(x)$ by $$\tQ^{**}_n(x)=\bigcup\tP^{an}_{an}(y)$$
    where the union is taken over all $y\in M$ such that $\tP^{an}_{an}(y)\in \tF(n)$ and which satisfy the following properties:
   \begin{itemize}
    \item[1.] There exists a point $q\in\tP^{an}_0(y)\cap\hat\Gamma$   such that   $\exp_x^{-1}(B^s(q,\e^{-n_0}))\subseteq T_xM$ contains the graph of a function $\varphi_s:R^s_x(C^{10}_2\e^{-n})\to E^u_x$ with slope less than $\frac{3}{10^5}$ and $|\varphi_s(0)|\leq C^7_2\e^{-n}$.
     \item[2.] There exists a point $d\in\tP^{0}_{an}(y)\cap\hat\Gamma$   such that   $\exp_x^{-1}(B^u(d,\e^{-n_0}))\subseteq T_xM$ contains the graph of a function $\varphi_u:R^u_x(C^{10}_2\e^{-n})\to E^s_x$ with slope less than $\frac{3}{10^5}$ and $|\varphi_u(0)|\leq C^7_2\e^{-n}$.   
     \end{itemize}    
    Then we have
    \begin{equation}\label{eq16bsecond}
        B(x,C_2^6\e^{-n})\cap\hat\Gamma\subseteq\tQ^{**}_n(x)\subseteq B(x,C_2^8\e^{-n}),
    \end{equation}  
    \begin{equation}\label{eq17bsecond}
        B^s(x,C_2^6\e^{-n})\cap\hat\Gamma\subseteq\tQ^{**}_n(x)\cap\xi^s(x)\subseteq B^s(x,C_2^8\e^{-n}),
    \end{equation} 
    \begin{equation}\label{eq18bsecond}
        B^u(x,C^6_2\e^{-n})\cap\hat\Gamma\subseteq\tQ^{**}_n(x)\cap\xi^u(x)\subseteq B^u(x,C_2^8\e^{-n}),
    \end{equation} 
    and for each $y\in\tQ^{**}_n(x)$, $\tP^{an}_{an}(y)\subseteq\tQ^{**}_n(x)$.
\end{lemma}

The next lemma records a simple but useful stability property of the sets \(\tQ_n(x)\), \(\tQ_n^*(x)\), and \(\tQ_n^{**}(x)\) under a change of base point.
\begin{lemma}\label{lem includingsecond}
For any $x,y\in\Gamma$ and $n\geq n_1$, one has $\tQ_n(x)\subseteq\tQ_n^*(y)$ whenever $y\in\tQ_n(x)\cap\Gamma$ and $\tQ_n^*(x)\subseteq\tQ_n^{**}(y)$ whenever $y\in\tQ_n^*(x)\cap\Gamma$.
\begin{proof}
We prove the first inclusion; the second one follows by the same argument.

Let
\[
    P:=\tP^{an}_{an}(z)
\]
be an atom contained in \(\tQ_n(x)\).  By the definition of \(\tQ_n(x)\), we have
\(P\in \tF(n)\), \(P\subseteq \tP(x)\), and there exist points
\[
    q\in \tP^{an}_0(z)\cap\hat\Gamma,
    \qquad
    d\in \tP^0_{an}(z)\cap\hat\Gamma
\]
such that, in the chart at \(x\),
\[\exp_x^{-1}(B^s(q,\e^{-n_0}))\] contains a graph over
\(R_x^s(C_2^{30}\e^{-n})\) with slope at most \(10^{-5}\) and displacement at
the origin at most \(C_2\e^{-n}\), and
\[\exp_x^{-1}(B^u(d,\e^{-n_0}))\] contains a graph over
\(R_x^u(C_2^{30}\e^{-n})\) with the same bounds.

Since \(y\in \tQ_n(x)\cap\Gamma\), \Cref{eq16second} gives
\[
    d(x,y)\leq C_2^2\e^{-n}.
\]
Also \(\tQ_n(x)\subseteq \tP(x)\), so \(y\in\tP(x)\), and hence
\[
    \tP(y)=\tP(x).
\]
Therefore \(P\subseteq \tP(y)\). 

It remains to check the graph conditions with \(y\) as the base point.  By the
uniform \(C^1\)-continuity of the exponential charts, of the bundles \(E^s,E^u\),
and of the local stable and unstable manifolds, after choosing \(C_2\) large
enough and then increasing \(n_1\) if necessary, changing the base point from
\(x\) to \(y\) with \(d(x,y)\leq C_2^2\e^{-n}\) transforms the above graphs into
graphs over the smaller domains
\[
    R_y^s(C_2^{20}\e^{-n})
    \quad\text{and}\quad
    R_y^u(C_2^{20}\e^{-n}),
\]
with slopes at most \(2\cdot10^{-5}\) and displacements at the origin at most
\(C_2^4\e^{-n}\).  Indeed, the loss in slope is absorbed by the margin from
\(10^{-5}\) to \(2\cdot10^{-5}\), and the displacement error, of order
\(d(x,y)\), is absorbed by the larger bound \(C_2^4\e^{-n}\); the reduction of
the domains from \(C_2^{30}\e^{-n}\) to \(C_2^{20}\e^{-n}\) absorbs the change
of coordinates.

Thus \(P\) satisfies the defining conditions of \(\tQ_n^*(y)\), and hence
\[
    P\subseteq \tQ_n^*(y).
\]
Taking the union over all atoms \(P\subseteq \tQ_n(x)\), we obtain
\[
    \tQ_n(x)\subseteq \tQ_n^*(y).
\]

Now assume \(y\in\tQ_n^*(x)\cap\Gamma\).  Then, by \Cref{eq16asecond},
\[
    d(x,y)\leq C_2^5\e^{-n}.
\]
If \(P=\tP^{an}_{an}(z)\subseteq\tQ_n^*(x)\), then the same argument gives
\(P\subseteq\tP(y)\) and \(P\in\tF(n)\) inside \(\tP(y)\).  Moreover, the graph
conditions defining \(\tQ_n^*(x)\), after changing the base point from \(x\) to
\(y\), remain valid on the smaller domains
\[
    R_y^s(C_2^{10}\e^{-n})
    \quad\text{and}\quad
    R_y^u(C_2^{10}\e^{-n}),
\]
with slopes at most \(3\cdot10^{-5}\) and displacements at the origin at most
\(C_2^7\e^{-n}\).  Hence \(P\) satisfies the defining conditions of
\(\tQ_n^{**}(y)\), so
\[
    P\subseteq \tQ_n^{**}(y).
\]
Taking the union over all atoms \(P\subseteq \tQ_n^*(x)\), we conclude that
\[
    \tQ_n^*(x)\subseteq \tQ_n^{**}(y).
\]
The proof is complete.
\end{proof}
\end{lemma}

\subsection{Preparations for the proof}
In this section, we carry out preparations for the proof of \Cref{main thmsecond} analogously to \cite[Section 5]{Barreira_Pesin_Schmeling}. 

Fix an $x\in\hat\Gamma$ and let $n>n_1$. For each set $A\subseteq \tP(x)$ and $\flat\in\{s,u\}$, define
\begin{equation*}
    N(n,A)=\#\{R\in\tR(n):R\cap A\neq \emptyset\}
\end{equation*}
\begin{equation*}
    N^\flat(n,y,A)=\#\{R\in\tR(n):R\cap\xi^\flat(y)\cap\Gamma\cap A\neq \emptyset\}
\end{equation*}
\begin{equation*}
    \hat N^\flat(n,y,A)=\#\{R\in\tF(n):R\cap\xi^\flat(y)\cap A\neq \emptyset\}
\end{equation*}
Then, analogously to the ideas in \cite[Section 5]{Barreira_Pesin_Schmeling}, we can prove a series of lemmas in sequence as follows.
\begin{lemma}\label{lem1second}
    For each $y\in \tP(x)\cap\Gamma$ and integer $n\geq n_1$ we have for $\flat\in\{s,u\}$
    \begin{equation*}
        N^\flat(n,y,\tQ_n(y))\leq m^\flat_y\left(B^\flat(y,C_2^2\e^{-n})\right)\cdot C\e^{anh+an\epsilon},
    \end{equation*}
    \begin{equation*}
        N^\flat(n,y,\tQ^*_n(y))\leq m^\flat_y\left(B^\flat(y,C_2^5\e^{-n})\right)\cdot C\e^{anh+an\epsilon},
    \end{equation*}
    \begin{equation*}
        N^\flat(n,y,\tQ_n^{**}(y))\leq m^\flat_y\left(B^\flat(y,C_2^8\e^{-n})\right)\cdot C\e^{anh+an\epsilon}.
    \end{equation*}
\begin{proof}
    By \Cref{eq17second} and the definitions, it follows that
    \begin{equation*}
    \begin{aligned}
        &m^s_y\left(B^s(y,C_2^2\e^{-n})\right)\geq m^s_y(\tQ_n(y))\\
        &\geq N^s\left(n,y,\tQ_n(y)\right)\cdot\min\{m^s_y(R):R\in \tR(n)\text{ and } R\cap\xi^s(y)\cap\tQ_n(y)\cap\Gamma\neq\emptyset\}. 
    \end{aligned}
    \end{equation*}
    Then the first inequality for $s$ follows from \Cref{markov ssecond} and \Cref{eq6second}. The first inequality for $u$ can be proved similarly. The proofs for $\tQ_n^*$ and $\tQ_n^{**}$ are the same.
\end{proof}
\end{lemma}

\begin{lemma}\label{lem2second}
    There exists a measurable function $n_2:M\rightarrow \mathbb N$   such that   $n_2(y)\geq n_1$ and for each $y\in\tP(x)\cap\hat\Gamma$ and $n\geq n_2(y)$, 
    \begin{equation*}
        m(B(y,\e^{-n}))\leq N(n,\tQ_n(y))\cdot 2C\e^{-2anh+2an\epsilon},
    \end{equation*}
    \begin{equation*}
        m(B(y,C_2^3\e^{-n}))\leq N(n,\tQ^*_n(y))\cdot 2C\e^{-2anh+2an\epsilon},
    \end{equation*}
    \begin{equation*}
        m(B(y,C^6_2\e^{-n}))\leq N(n,\tQ^{**}_n(y))\cdot 2C\e^{-2anh+2an\epsilon}.
    \end{equation*}
\begin{proof}
    By the Borel density point arguments, for $m$-a.e. $y\in\hat\Gamma$, there is an integer $n_2(y)\geq n_1$   such that   for all $n\geq n_2(y)$ one has $m(B(y,\e^{-n})\cap\hat\Gamma)\geq\frac{1}{2}m(B(y,\e^{-n}))$. Together with \Cref{eq16second}, it follows that
    \begin{equation*}
    \begin{aligned}
        \frac{1}{2}m(B(y,\e^{-n}))&\leq m(B(y,\e^{-n})\cap\hat\Gamma)\leq m(\tQ_n(y)\cap\Gamma)\\
        &\leq N(n,\tQ_n(y))\cdot\max\{m(R):R\in \tR(n)\text{ and }R\cap\tQ_n(y)\neq\emptyset\}.
    \end{aligned}
    \end{equation*}
    The desired inequality follows from \Cref{eq5second}.
    The proofs for $\tQ_n^*$ and $\tQ_n^{**}$ are the same.
\end{proof}
\end{lemma}

\begin{lemma}\label{lem3second}
    There exists $\Delta\in\mathbb N$   such that   for $m$-a.e. $y\in\tP(x)\cap\hat\Gamma$ and $n\geq n_2(y)$,
    \begin{equation*}
        N(n+\Delta,\tQ_{n+\Delta}(y))\leq\hat N^s(n,y,\tQ_n(y))\cdot\hat N^u(n,y,\tQ_n(y))\cdot 2C^2\e^{2\Delta\cdot a(h+\epsilon)+4an\epsilon},
    \end{equation*}
      \begin{equation*}
        N(n+\Delta,\tQ^*_{n+\Delta}(y))\leq\hat N^s(n,y,\tQ^*_n(y))\cdot\hat N^u(n,y,\tQ^*_n(y))\cdot 2C^2\e^{2\Delta\cdot a(h+\epsilon)+4an\epsilon},
    \end{equation*}
      \begin{equation*}
        N(n+\Delta,\tQ^{**}_{n+\Delta}(y))\leq\hat N^s(n,y,\tQ^{**}_n(y))\cdot\hat N^u(n,y,\tQ^{**}_n(y))\cdot 2C^2\e^{2\Delta\cdot a(h+\epsilon)+4an\epsilon}.
    \end{equation*}
\begin{proof}
We choose $\Delta=\Delta(C_2)$   such that   for all $n\geq n_2(y)$,
\begin{equation*}
\begin{aligned}
    2m(\tQ_n(y)\cap\hat\Gamma)&\geq 2m(B(y,\e^{-n})\cap\hat\Gamma)\geq m(B(y,\e^{-n}))\\
    &\geq m(B(y,C^2_2\e^{-n-\Delta}))\geq m(\tQ_{n+\Delta}(y)).
\end{aligned}
\end{equation*}
Moreover, for $n\geq n_2(y)$ we have
\begin{equation*}
\begin{aligned}
    m(\tQ_{n+\Delta}(y))&=\sum_{\tP^{a({n+\Delta})}_{a({n+\Delta})}(z)\subseteq\tQ_{n+\Delta}(y)}m(\tP^{a({n+\Delta})}_{a({n+\Delta})}(z))\\&\geq N({n+\Delta},\tQ_{n+\Delta}(y))\cdot C^{-1}\e^{-2a({n+\Delta})(h+\epsilon)}
\end{aligned}
\end{equation*}
and
\begin{equation*}
    m(\tQ_n(y)\cap\hat\Gamma)=\sum_{\tP^{an}_{an}(z)\subseteq\tQ_n(y)}m(\tP^{an}_{an}(z)\cap\hat\Gamma)\leq N_n(y)C\e^{-2an(h-\epsilon)},
\end{equation*}
where $$N_n(y):=\#\{\tP_{an}^{an}(z)\in\tR(n):\tP^{an}_{an}(z)\cap\hat\Gamma\neq\emptyset\text{ and }\tP^{an}_{an}(z)\subseteq\tQ_n(y)\}.$$ Together, it follows that
\begin{equation}\label{equation 54second}
    N(n+\Delta,\tQ_{n+\Delta}(y))\leq 2N_n(y)\cdot C^2\e^{2a\Delta(h+\epsilon)+4an\epsilon}.
\end{equation}
On the other hand, for any $v\in \hat\Gamma$ and $\tP^{an}_{an}(v)\subseteq\tQ_n(y)$, since $$B^u(y,\e^{-n_0})\subseteq \xi^u(y)\ ,B^s(y,\e^{-n_0})\subseteq \xi^s(y),$$      by \Cref{lemma 7.11}, the intersections $\tP^0_{an}(v)\cap\tP^{an}_0(y)$ and $\tP^0_{an}(y)\cap\tP^{an}_0(v)$ are non-empty. Moreover,
since $\tP^{an}_{an}(v)\subseteq\tQ_n(y)$ and $\tP^{an}_{an}(y)\subseteq\tQ_n(y)$ imply
\[
\tP^0_{an}(v)\cap\tP^{an}_0(y)\subseteq\tQ_n(y),\ \tP^0_{an}(y)\cap\tP^{an}_0(v)\subseteq \tQ_n(y),
\] 
it follows that
$$\left(
\tP^0_{an}(v)\cap\tP^{an}_0(y),\tP^0_{an}(y)\cap\tP^{an}_0(v)
\right)$$ 
belongs to 
\begin{equation*}
    \{R\in\tF(n):R\cap\xi^s(y)\cap\tQ_n(y)\neq \emptyset\}\times\{R\in\tF(n):R\cap\xi^u(y)\cap\tQ_n(y)\neq\emptyset\}.
\end{equation*}

By definition, this correspondence from
\[
\{\tP_{an}^{an}(z)\in\tR(n):\tP^{an}_{an}(z)\cap\hat\Gamma\neq\emptyset\text{ and }\tP^{an}_{an}(z)\subseteq\tQ_n(y)\}
\]
to
\[
 \{R\in\tF(n):R\cap\xi^s(y)\cap\tQ_n(y)\neq \emptyset\}\times\{R\in\tF(n):R\cap\xi^u(y)\cap\tQ_n(y)\neq\emptyset\}
\]
is injective. 
Then it follows that
\begin{equation*}
    N_n(y)\leq\hat N^s(n,y,\tQ_n(y))\cdot\hat N^u(n,y,\tQ_n(y))
\end{equation*}
and together with \Cref{equation 54second}, the proof is complete.
  The proofs for $\tQ_n^*$ and $\tQ_n^{**}$ are the same.
\end{proof}
\end{lemma}

\begin{lemma}\label{lem4second}
    For each $x\in\hat\Gamma$ and $n\geq n_1$, 
    \begin{equation*}
        \hat N^s(n,x,\tP(x))\leq D^{-1}C^2\e^{anh+3an\epsilon},
    \end{equation*}
    \begin{equation*}
        \hat N^u(n,x,\tP(x))\leq D^{-1}C^2\e^{anh+3an\epsilon}.
    \end{equation*}
    \begin{proof}
        Since the partition $\tP$ is finite, there exist points $y_i\in M, i\in \mathbb N$   such that   $\bigcup \tP^{an}_0(y_i)=\tP(x)$ and they are mutually disjoint. Assume that $y_i\in\hat\Gamma$ whenever $\tP^{an}_0(y_i)\cap\hat\Gamma\neq\emptyset$. Then we have $$N(n,\tP(x))\geq\sum_{\tP^{an}_0(y_i)\cap\hat\Gamma\neq\emptyset}N^s(n,y_i,\tP^{an}_0(y_i)).$$ By \Cref{lemma 41} and \Cref{lemma 6second} A, for any $y_i\in\hat\Gamma$,
        \begin{equation*}
        \begin{aligned}
            N^s(n,y_i,\tP^{an}_0(y_i))&\geq\frac{m^s_{y_i}(\tP^{an}_0(y_i)\cap\Gamma)}{\max\{m^s_{y_i}(\tP^{an}_{an}(z)):z\in\xi^s(y_i)\cap\tP(x)\cap\Gamma\}}\\
            &\geq \frac{D}{\max\{m^s_z(\tP^{0}_{an}(z)):z\in\xi^s(y_i)\cap\tP(x)\cap\Gamma\}}\\
            &\geq DC^{-1}\e^{anh-an\epsilon}.
        \end{aligned}
        \end{equation*}
    and similarly it follows that
    \begin{equation*}
        N(n,\tP(x))\leq\frac{m(\tP(x))}{\min\{m(\tP^{an}_{an}(z)):z\in\tP(x)\cap\Gamma\}}\leq C\e^{2anh+2an\epsilon}.
    \end{equation*}
    By the construction of $\xi^u$, if $y\in\xi^u(x)$, then $y\in\tP^0_{an}(x)$ and thus $$\tP^{an}_{an}(y)=\tP^0_{an}(x)\cap\tP^{an}_0(y).$$ Hence, we have
    \begin{equation*}
        \hat N^u(n,x,\tP(x))\leq\#\{i:\tP^{an}_0(y_i)\cap\hat\Gamma\neq\emptyset\}.
    \end{equation*}
    Together, we conclude that
    \begin{equation*}
    \begin{aligned}
        C\e^{2anh+2an\epsilon}&\geq N(n,\tP(x))\geq \sum_{i:\tP^{an}_0(y_i)\cap\hat\Gamma\neq\emptyset}
        N^s(n,y_i,\tP^{an}_0(y_i))\\
        &\geq \#\{i:\tP^{an}_0(y_i)\cap\hat\Gamma\neq\emptyset\}\cdot\min_{y_i\in\hat\Gamma}         N^s(n,y_i,\tP^{an}_0(y_i))\\
        &\geq\hat N^u(n,x,\tP(x))\cdot DC^{-1}\e^{anh-an\epsilon}.
    \end{aligned}
    \end{equation*}
    Thus $\hat N^u(n,x,\tP(x))\leq D^{-1}C^2\e^{anh+3an\epsilon}$. The other inequality can be proved similarly.
    \end{proof}
\end{lemma}

Finally, we establish the key \Cref{lem5second} in this section as follows, which is parallel to \cite[Lemma 5]{Barreira_Pesin_Schmeling}.
\begin{lemma}\label{lem5second}
    For $m$-a.e. $y\in\tP(x)\cap\hat\Gamma$, we have for $\flat\in\{s,u\}$,
    \begin{equation}\label{eq 181}
        \limsup_{n\rightarrow\infty}\frac{\hat N^\flat(n,y,\tQ_n(y))}{N^\flat(n,y,\tQ_n(y))}\cdot\e^{-7an\epsilon}<1,
    \end{equation}
      \begin{equation*}
        \limsup_{n\rightarrow\infty}\frac{\hat N^\flat(n,y,\tQ_n^*(y))}{N^\flat(n,y,\tQ_n^*(y))}\cdot\e^{-7an\epsilon}<1,
    \end{equation*}
      \begin{equation*}
        \limsup_{n\rightarrow\infty}\frac{\hat N^\flat(n,y,\tQ_n^{**}(y))}{N^\flat(n,y,\tQ_n^{**}(y))}\cdot\e^{-7an\epsilon}<1.
    \end{equation*}
\begin{proof}
    By \Cref{lemma 10second} and the definition of $n_2$ we have for $y\in\hat\Gamma,n\geq n_2(y)$
    \begin{equation*}
    m^s_y(\tQ_n(y)\cap\Gamma)\geq m^s_y(B^s(y,\e^{-n})\cap\hat\Gamma)\geq\frac{1}{2}m^s_y(B^s(y,\e^{-n})).
    \end{equation*}
    Together with \Cref{lemma 6second}, it follows that
    \begin{equation}\label{eq86second}
    \begin{aligned}
        N^s(n,y,\tQ_n(y))&\geq\frac{m^s_y(\tQ_n(y)\cap\Gamma)}{\max\{m^s_z(\tP^{an}_{an}(z)):z\in\xi^s(y)\cap\tP(x)\cap\Gamma\}}\\
        &\geq\frac{1}{2}\frac{m^s_y(B^s(y,\e^{-n}))}{\max\{m^s_z(\tP^0_{an}(z)):z\in\xi^s(y)\cap\tP(x)\cap\Gamma\}}\\
        &\geq\frac{1}{2C}\frac{\e^{-d^sn-n\epsilon}}{\e^{-anh+an\epsilon}}.
    \end{aligned}
    \end{equation}
    Now we define the set $$F:=\{y\in\hat\Gamma:\limsup_{n\rightarrow\infty}\frac{\hat N^s(n,y,\tQ_n(y))}{N^s(n,y,\tQ_n(y))}\e^{-7an\epsilon}\geq 1\}.$$ It only remains to prove that $m(F)=0$. If $m(F)>0$, then there exists a set $F'\subseteq F$ with $m(F')=m(F)>0$   such that   
     $$\lim_{r\rightarrow 0}\frac{\log m^s_y(B^s(y,r))}{\log r}=d^s,\ \forall y\in F'.$$ Hence, there exists $y\in F$   such that   $m^s_y(F')=m^s_y(F'\cap\xi^s(y))>0$. By the Frostman's Lemma it follows that $\dim_{\mathrm H}(F'\cap\xi^s(y))\geq d^s$, where $\dim_{\mathrm H}$ denotes the Hausdorff dimension of sets.

    On the other hand, for each $y\in F$ there exists an increasing sequence of integers $\{m_j(y)\}_{j=1}^\infty\subseteq \mathbb N$   such that   for any $j\in\mathbb N$
    \begin{equation}\label{eq29second}
    \begin{aligned}
        \hat N^s(m_j,y,\tQ_{m_j}(y))&\geq\frac{1}{2}N^s(m_j,y,\tQ_{m_j}(y))\e^{7am_j\epsilon}\\
        &\geq\frac{1}{4C}\e^{-d^sm_j+am_jh+5am_j\epsilon}.
    \end{aligned}
    \end{equation}
    Then consider the cover of $F'\cap\xi^s(y)$ by the balls
    $$\tB=\{B(z,C_2^2\e^{-m_j(z)}):z\in F'\cap\xi^s(y),j\in \mathbb N\}.$$
    By the Besicovitch Covering Lemma, for any $L>0$ there exists a countable sub-cover $$\tC=\{B(z_i,C_2^2\e^{-t_i}):i\in\mathbb N\}$$ with  $z_i\in F'\cap\xi^s(y)$ and $t_i\in\{m_j(z_i)\}_{j=0}^\infty$ larger than $L$, and $\tC$ covers each point in  $ \xi^s(y)$ with multiplicity not exceeding $\rho=\rho(\dim M)\in\mathbb N$. Then the $(d^s-\epsilon)$-dimensional Hausdorff sum corresponding to this cover is 
    \begin{equation}\label{equa1second}
        \sum_{B\in\tC}(\mathrm{diam} B)^{d^s-\epsilon}=(2C^2_2)^{d^s-\epsilon}\sum_{i=1}^\infty \e^{-t_i(d^s-\epsilon)}.
    \end{equation}
    By \Cref{eq29second} we have
    \begin{equation}\label{equa2second}
    \begin{aligned}
        \sum_{i=1}^{\infty}\e^{-t_i(d^s-\epsilon)}&\leq\sum_{i=1}^\infty \hat N^s(t_i,z_i,\tQ_{t_i}(z_i))\cdot 4C\e^{-at_ih-4at_i\epsilon}\\
        &\leq 4C\sum_{q=1}^\infty\e^{-aqh-4aq\epsilon}\sum_{i:t_i=q}\hat N^s(q,z_i,\tQ_{t_i}(z_i)). 
    \end{aligned}
    \end{equation}
    Since $\tQ_{t_i}(z_i)\subseteq B(z_i,C^2_2\e^{-t_i})$, by the finite covering property of the Besicovich covering, for any fixed $q\in\mathbb N$ and $t_i=q$, each atom of the partition $\tP^{aq}_{aq}$ appears in $\tQ_{t_i}(z_i)$ at most $\rho$ times. Moreover, since $\xi^s(y)=\xi^s(z_i)$, it follows that
    \begin{equation}\label{equa3second}
        \sum_{i:t_i=q}\hat N^s(q,z_i,\tQ_{t_i}(z_i))\leq\rho \hat N^s(q,y,\tP(y)).
    \end{equation}
    Combining \Cref{equa1second}, \Cref{equa2second}, \Cref{equa3second} and \Cref{lem4second}, a direct calculation gives
    \begin{equation*}
        \sum_{B\in\tC}(\mathrm{diam} B)^{d^s-\epsilon}\leq 4(2C_2^2)^{d^s-\epsilon}D^{-1}C^3\rho\sum_{q=1}^\infty\e^{-aq\epsilon}<\infty.
    \end{equation*}

 Since $L$ can be arbitrarily large, the diameter of the cover $\tC$ can be arbitrarily small, and it follows that
\[
\dim_{\mathrm H}(F'\cap\xi^s(y)) \leq d^s - \epsilon,
\]
which contradicts the assumption that $\dim_{\mathrm H}(F'\cap\xi^s(y)) \geq d^s$. Thus, we conclude that $m(F) = 0$, completing the proof of \Cref{eq 181} for $s$. The proofs for the other inequalities are similar.   
\end{proof}
\end{lemma}

\subsection{Proof of \Cref{main thmsecond}}
In this section, we complete the proof of \Cref{main thmsecond} following the ideas in \cite{Barreira_Pesin_Schmeling} with modifications. 
We start by making some necessary preparations.

By \Cref{lem5second}, for $m$-a.e. $y\in\hat\Gamma$, there exists an integer $n_3(y)\geq n_2(y)$ such that if $n\geq n_3(y)$, then for $\flat\in\{s,u\}$,
\begin{equation*}
   \hat N^\flat(n,y,\tQ_n(y))<N^\flat(n,y,\tQ_n(y))\cdot\e^{7an\epsilon},
\end{equation*}
\begin{equation*}
    \hat N^\flat(n,y,\tQ^*_n(y))<N^\flat(n,y,\tQ^*_n(y))\cdot\e^{7an\epsilon},
\end{equation*}
\begin{equation*}
    \hat N^\flat(n,y,\tQ_n^{**}(y))<N^\flat(n,y,\tQ_n^{**}(y))\cdot\e^{7an\epsilon}.
\end{equation*}

For any $\epsilon>0$ there exists a subset $\Gamma_\epsilon\subseteq\hat\Gamma$   such that  
\begin{equation*}
    m(\Gamma_\epsilon)>m(\hat\Gamma)-\epsilon\text{ and }n_\epsilon:=\sup\{n_1,n_3(y):y\in\Gamma_\epsilon\}<\infty.
\end{equation*}
Similarly, using the Borel density point arguments, there exist $\hat\Gamma_\epsilon\subseteq \Gamma_\epsilon$ and an integer $\hat n_\epsilon>n_\epsilon$   such that   $m(\hat\Gamma_\epsilon)>m(\Gamma_\epsilon)-\epsilon$ and for any $x\in\hat\Gamma_\epsilon$, $n>\hat n_\epsilon$,
\begin{equation*}
    m(B(x,\e^{-n})\cap\hat\Gamma_\epsilon)\geq \frac{1}{2}m(B(x,\e^{-n})),
\end{equation*}
\begin{equation*}
    m^s_x(B^s(x,\e^{-n})\cap\hat\Gamma_\epsilon)\geq \frac{1}{2}m^s_x(B^s(x,\e^{-n})),
\end{equation*}
\begin{equation*}
    m^u_x(B^u(x,\e^{-n})\cap\hat\Gamma_\epsilon)\geq\frac{1}{2}m^u_x(B^u(x,\e^{-n})).
\end{equation*}
The following two lemmas establish the asymptotically ``almost'' local product structure for  hyperbolic measures as in \cite{Barreira_Pesin_Schmeling}.
\begin{lemma}\label{lemma lowboundsecond}
    For any $\epsilon>0$, $y\in\hat\Gamma_\epsilon$, and $n\geq\hat n_\epsilon$,
    \begin{equation*}
        m^s_y(B^s(y,\e^{-n}))\cdot m^u_y(B^u(y,\e^{-n}))\leq m(B(y,C^8_2\e^{-n}))\cdot 4C^3\e^{11an\epsilon}.
    \end{equation*}
    \begin{proof}
    Let $z\in\Gamma_\epsilon\cap\tQ_n(y)$ and $n\geq n_\epsilon$. Then we have 
    \begin{equation*}
        N^u(n,y,\tQ_n(y))\leq\hat N^u(n,y,\tQ_n(y))
    \end{equation*}
    because for any $R\in\tR(n)$, $R\subseteq\tQ_n(y)$ implies $R\in\tF(n)$, since \(\tQ_n(y)\) is a union of some \(\tP^{an}_{an}\)-atoms belonging to
\(\tF(n)\) by definition.
 Moreover, we claim that 
\begin{claim}
    \begin{equation*}
        \hat N^u(n,y,\tQ_n(y))=\hat N^u(n,z,\tQ_n(y)).
    \end{equation*}
    \begin{proof}
     Note that now $z\in B(y,C_2^2\e^{-n})\cap\Gamma_\epsilon$; thus, by \Cref{lemma 7.11}, for any atom of partition $\tP^{an}_{an}$ such that $$\tP^{an}_{an}(k)\in\{R\in\tF(n):R\cap\xi^u(y)\cap \tQ_n(y)\neq \emptyset\},$$
     there exists $q\in\tP^{an}_0(k)\cap\hat\Gamma$  such that $$B^s(q,\e^{-n_0})\cap B^u(y,\e^{-n_0})=\{v\},$$
    $$B^s(q,\e^{-n_0})\cap B^u(z,\e^{-n_0})=\{v'\},$$ 
     and $v,v'\in \tP^{an}_0(k)$, $\tP^{an}_{an}(v)=\tP^{an}_{an}(k)$.
     
    By the definition of $\tQ_n$ and $B^u(z,\e^{-n_0})\subseteq\xi^u(z)$, it can be checked that $\tP^{an}_{an}(v')\subseteq\tQ_n(y)$ and $\tP^{an}_{an}(v')\cap\xi^u(z)\neq\emptyset$. 
     Thus, it follows that
     $$\tP^{an}_{an}(v')\in\{R\in\tF(n):R\cap\xi^u(z)\cap \tQ_n(y)\neq \emptyset\}.$$ 
     Now consider the correspondence that maps $\tP_{an}^{an}(k)$ to $\tP_{an}^{an}(v')$. This correspondence is injective by definition, which implies $$\hat N^u(n,y,\tQ_n(y))\leq \hat N^u(n,z,\tQ_n(y)).$$ Similarly, one has  
     $$\hat N^u(n,y,\tQ_n(y))\geq\hat N^u(n,z,\tQ_n(y)).$$
     Thus, the equality $$\hat N^u(n,y,\tQ_n(y))=\hat N^u(n,z,\tQ_n(y))$$
     is obtained and the proof of the claim is complete.
     \end{proof}
\end{claim}
     
    By \Cref{lem includingsecond}, we have $\tQ_n(y)\subseteq\tQ^*_n(z)\subseteq\tQ_n^{**}(y)$. It follows that
    \begin{equation*}
        \hat N^u(n,z,\tQ_n(y))\leq \hat N^u (n,z,\tQ_n^*(z))< N^u(n,z,\tQ_n^*(z))\cdot\e^{7an\epsilon}\leq N^u(n,z,\tQ^{**}_n(y))\cdot\e^{7an\epsilon}.
    \end{equation*}
    Combining all the above discussions, we obtain 
    \begin{equation*}
        N^u(n,y,\tQ_n(y))\leq \inf \{N^u(n,z,\tQ^{**}_n(y)): z\in \Gamma_\epsilon\cap\tQ_n(y)\}\cdot\e^{7an\epsilon}.
    \end{equation*}
    Now, define
    \begin{equation*}
    \hat N^s_{\Gamma^u_\epsilon}(n,y,\tQ_n(y)):=\#\{\tP^{an}_0(y)\cap\tP^{0}_{an}(z): z\in \Gamma_\epsilon\cap\tQ_n(y)\}.
    \end{equation*}
    By the definition of $N(n,\tQ^{**}_n(y))$, we have that
    \begin{equation*}
    \hat N^s_{\Gamma^u_\epsilon}(n,y,\tQ_n(y))\cdot \inf \{N^u(n,z,\tQ^{**}_n(y)): z\in \Gamma_\epsilon\cap\tQ_n(y)\} \leq N(n,\tQ^{**}_n(y)).
    \end{equation*}
    We further define $$N^s_{\Gamma_\epsilon}(n,y,\tQ_n(y)):=\#\{R\in\tR(n):R\subseteq\tQ_n(y)\,\text{ and }\ R\cap\Gamma_\epsilon\cap\xi^s(y)\neq\emptyset\}.$$ By the definitions, one can directly check that $$\hat N^s_{\Gamma^u_\epsilon}(n,y,\tQ_n(y))\geq N^s_{\Gamma_\epsilon}(n,y,\tQ_n(y)).$$ Therefore,
    \begin{equation*}
        N^s_{\Gamma_\epsilon}(n,y,\tQ_n(y))\times N^u(n,y,\tQ_n(y))\leq N(n,\tQ^{**}_n(y))\cdot\e^{7an\epsilon}.
    \end{equation*}
    Recall that $y\in\hat\Gamma_\epsilon$ and $n\geq\hat n_\epsilon$, as in \Cref{eq86second} we can prove that
    \begin{equation*}
        N^s_{\Gamma_\epsilon}(n,y,\tQ_n(y))\geq m^s_y(B^s(y,\e^{-n}))(2C)^{-1}\e^{anh-an\epsilon},
    \end{equation*}
    \begin{equation*}
        N^u(n,y,\tQ_n(y))\geq m^u_y(B^u(y,\e^{-n}))(2C)^{-1}\e^{anh-an\epsilon}.
    \end{equation*}
    Moreover, by \Cref{eq5second} and \Cref{lemma **second}, it follows that
    \begin{equation*}
        N(n,\tQ^{**}_n(y))\leq\frac{m(\tQ_n^{**}(y))}{\min\{m(\tP^{an}_{an}(z)):z\in\tQ^{**}_n(y)\cap\Gamma\}}\leq m(B(y,C_2^{8}\e^{-n}))\cdot C\e^{2anh+2an\epsilon}.
    \end{equation*}
    Putting these three inequalities together, the proof is complete.
    \end{proof}
\end{lemma}
\begin{lemma}\label{lemma upperboundsecond}
        For any $\epsilon>0$, $y\in\hat\Gamma_\epsilon$, and $n\geq\hat n_\epsilon$, 
        \begin{equation*}
            m(B(y,\e^{-n-\Delta}))\leq m^s_y(B^s(y,C_2^2\e^{-n}))\cdot m^u_y(B^u(y,C^2_2\e^{-n}))\cdot 4C^5\e^{4\Delta\cdot a\epsilon+22an\epsilon}.
        \end{equation*}
    \begin{proof}
        By \Cref{lem2second} and \Cref{lem3second}, it follows that
        \begin{equation*}
            m(B(y,\e^{-n-\Delta}))\leq \hat N^s(n,y,\tQ_n(y))\cdot\hat N^u(n,y,\tQ_n(y))\cdot4C^3\e^{4\Delta\cdot a\epsilon-2anh+6an\epsilon}.
        \end{equation*}
        Hence, by previous discussions,
        \begin{equation*}
            m(B(y,\e^{-n-\Delta}))\leq  N^s(n,y,\tQ_n(y)) \cdot N^u(n,y,\tQ_n(y))\cdot4C^3\e^{4\Delta\cdot a\epsilon-2anh+20an\epsilon}.
        \end{equation*}
        Together with \Cref{lem1second}, the proof is complete.
    \end{proof}
\end{lemma}
Finally, as in \cite[Section 6]{Barreira_Pesin_Schmeling}, we combine \Cref{lemma lowboundsecond} and \Cref{lemma upperboundsecond} and let $\epsilon\rightarrow 0$. 
From \Cref{lemma lowboundsecond}, it follows that
$$d^u+d^s\geq\limsup_{r\rightarrow0}\frac{\log m(B(x,r))}{\log r},\quad m-\text{a.e.}\ x\in M,$$
and from \Cref{lemma upperboundsecond}, one has
$$d^u+d^s\leq\liminf_{r\rightarrow0}\frac{\log m(B(x,r))}{\log r},\quad m-\text{a.e.}\ x\in M.$$
Together, the proof of \Cref{main thmsecond} is complete.

\section*{Acknowledgements}
The author thanks Professors Shaobo Gan, Jiagang Yang, and Fan Yang for their careful proofreading, numerous helpful discussions, and comments on earlier drafts of this paper. The author also thanks them for their guidance and encouragement.

This work was partially completed at the Instituto de Matemática Pura e Aplicada (IMPA). The author thanks Professor Marcelo Viana for the invitation and hospitality during the visit in January–February 2026.

\bigskip
\noindent {\bf Data Availability}
No numerical or categorical data were used in this article.
\bigskip
\bibliographystyle{abbrv}
{\footnotesize\bibliography{library}}

\end{document}